\documentclass[11pt,a4paper,reqno]{amsart}

\usepackage{geometry}
\usepackage{amssymb,amsmath,amsthm,mathrsfs}
\usepackage{graphicx}
\usepackage{float}
\usepackage{colortbl}
\usepackage{enumerate}
\usepackage{upref}
\usepackage{stackrel}
\usepackage{enumitem}
\usepackage{bbm}

\usepackage[bookmarks=false]{hyperref}
\usepackage[T2A,T1]{fontenc}
\DeclareSymbolFont{cyrillic}{T2A}{cmr}{m}{n}
\DeclareMathSymbol{\D}{\mathalpha}{cyrillic}{196}

\usepackage{accents}
\usepackage[colorinlistoftodos]{todonotes}

\usepackage{amsfonts} %% <- also included by amssymb
\DeclareMathSymbol{\shortminus}{\mathbin}{AMSa}{"39}

\usepackage{tikz}
\usetikzlibrary{arrows}
\usetikzlibrary{decorations.pathreplacing}

\usepackage{chngcntr}
\counterwithin{figure}{section}

\theoremstyle{plain}% default
\newtheorem{theorem}{Theorem}[section]
\newtheorem*{theorem*}{Theorem}
\newtheorem{lemma}[theorem]{Lemma}

\newtheorem{proposition}[theorem]{Proposition}
\newtheorem{corollary}[theorem]{Corollary}

\newtheorem*{claim}{Claim}
\theoremstyle{definition}
\newtheorem{definition}[theorem]{Definition}

\newtheorem{example}[theorem]{Example}
\newtheorem*{condition}{Condition}

\theoremstyle{remark}
\newtheorem{remark}[theorem]{Remark}%[section]

\usepackage{enumitem}
\makeatletter
\def\namedlabel#1#2{\begingroup
   #2%
 \def\@currentlabel{#2}%
   \phantomsection\label{#1}\endgroup
}

\def\R{\ensuremath{\mathbb R}}
\def\N{\ensuremath{\mathbb N}}
\def\Q{\ensuremath{\mathbb Q}}
\def\Z{\ensuremath{\mathbb Z}}
\def\I{\ensuremath{{\mathbbm{1}}}}
\def\e{{\ensuremath{\rm e}}}

\def\C{\ensuremath{\mathcal C}}

\def\leb{{\rm Leb}}

\def\p{\ensuremath{\mathbb P}}
\def\F{\ensuremath{\mathcal F}}
\def\O{\ensuremath{\text{O}}}
\def\oo{\ensuremath{\text{o}}}

\def\A{\ensuremath{A^{(q_n)}}}
\def\n{\ensuremath{n}}

\def\X{\mathbf{X}}
\def\x{\mathbf{x}}
\def\z{\mathbf{z}}
\def\w{\mathbf{w}}
\def\V{\mathcal{V}}

\def\TT{\sigma}

\def\1{{\bf 1}}

\def\V{\mathcal{V}}
\def\W{\ensuremath{\mathscr W}}

\def\ux{\underline{x}}
\def\uy{\underline{y}}

\renewcommand{\d}[1]{\ensuremath{\operatorname{d}\!{#1}}}

\def\ie{{\em i.e.}, }
\def\iid{{i.i.d.\ }}

\def\E{\mathbb E}

\def\dist{\ensuremath{\text{dist}}}

\def\eps{\varepsilon}

\def\proj{p}
\def\Proj{P}
\def\h{\upsilon}

\numberwithin{equation}{section}

\begin{document}

\title[Enriched limits for dynamical systems]{Enriched functional limit theorems for dynamical systems}

\author[A. C. M. Freitas]{Ana Cristina Moreira Freitas}
\address{Ana Cristina Moreira Freitas\\ Centro de Matem\'{a}tica \&
Faculdade de Economia da Universidade do Porto\\ Rua Dr. Roberto Frias \\
4200-464 Porto\\ Portugal} \email{\href{mailto:amoreira@fep.up.pt}{amoreira@fep.up.pt}}
\urladdr{\url{http://www.fep.up.pt/docentes/amoreira/}}

\author[J. M. Freitas]{Jorge Milhazes Freitas}
\address{Jorge Milhazes Freitas\\ Centro de Matem\'{a}tica \& Faculdade de Ci\^encias da Universidade do Porto\\ Rua do
Campo Alegre 687\\ 4169-007 Porto\\ Portugal}
\email{\href{mailto:jmfreita@fc.up.pt}{jmfreita@fc.up.pt}}
\urladdr{\url{http://www.fc.up.pt/pessoas/jmfreita/}}

\author[M. Todd]{Mike Todd}
\address{Mike Todd\\ Mathematical Institute\\
University of St Andrews\\
North Haugh\\
St Andrews\\
KY16 9SS\\
Scotland \\} \email{\href{mailto:m.todd@st-andrews.ac.uk }{m.todd@st-andrews.ac.uk }}
\urladdr{\url{https://mtoddm.github.io/}}

\date{\today}


%60G51   	Processes with independent increments; LÃ©vy processes
%37A50  	Relations with probability theory and stochastic processes
%60G70  	Extreme value theory; extremal processes
%37B20  	Notions of recurrence
%60G10  	Stationary processes
%37C25  	Fixed points, periodic points, fixed-point index theory
%--------
%37A25  	Ergodicity, mixing, rates of mixing
%37D25  	Nonuniformly hyperbolic systems (Lyapunov exponents, Pesin theory, etc.)
%37D35  	Thermodynamic formalism, variational principles, equilibrium states
%37C40  	Smooth ergodic theory, invariant measures
%60G55   	Point processes
%60F17   	Functional limit theorems; invariance principles

\begin{abstract}

We prove functional limit theorems for dynamical systems in the presence of  clusters of large values which, when summed and suitably normalised, get collapsed in a jump of the limiting process observed at the same time point.  To keep track of the clustering information, which gets lost in the usual Skorohod topologies in the space of c\`adl\`ag functions, we introduce a new space which generalises the already more general spaces introduced by Whitt.  Our main applications are to hyperbolic and non-uniformly expanding dynamical systems with heavy-tailed observable functions maximised at dynamically linked maximal sets (such as periodic points).  We also study limits of extremal processes and record times point processes for observables not necessarily heavy tailed.  
The applications studied include hyperbolic systems such as Anosov diffeomorphisms, but also non-uniformly expanding maps such as maps with critical points of Benedicks-Carleson type or indifferent fixed points such as Pomeau-Manneville or Liverani-Saussol-Vaienti maps.  The main tool is a limit theorem for point processes with decorations derived from a bi-infinite sequence called the transformed anchored tail process.

\end{abstract}

\maketitle

\tableofcontents

\section{Introduction}

The Donsker functional limit theorem gives an invariance principle for the sum of independent and identically distributed (i.i.d.) random variables with finite second moments. The functional limit is a Brownian motion which lives in the space $C$ of continuous $\R^d$-valued functions defined on a subinterval of the real line, equipped with the uniform norm. Invariance principles such as this have been proved for a large class of dynamical systems with good mixing properties (\cite{HK82,DP84,MT02,FMT03,MN05,HM07,BM08,MN09a,G10a}). Note that weak convergence to a Brownian motion implies the existence of a Central Limit Theorem (CLT) for the ergodic sums of such systems. 

When i.i.d.\ random variables have infinite second moment, the heavy-tailed case, the sums of the variables can become dominated by a few very large values, leading to failure of the CLT and jumps/discontinuities in the functional limit, as well as a connection with maximal processes and Extreme Value Laws (the Type II/Fr\'echet case).  If the tail of the distribution is sufficiently regular the CLT is replaced by an $\alpha$-stable law and Brownian motion is replaced by an $\alpha$-stable L\'evy process.     
In the dynamical setting, in \cite{G08} Gou\"ezel showed that the same regime applies to ergodic sums of unbounded heavy tailed observables for the doubling map.  Later, in \cite{T10}, Tyran-Kami\'nska showed a functional limit theorem for heavy tailed ergodic sums of essentially uniformly expanding maps. The functional limit was an $\alpha$-stable L\'evy process, which lives in the space $D$ of c\`adl\`ag functions, \ie right continuous  $\R$-valued functions with left limits defined, equipped with Skorohod's $J_1$ topology.

The problem we address in this paper is when there are clusters of large $\R^d$-valued observations, so rather than a simple jump, the functional limit would see a sequence of jumps, possibly in different directions: this can be seen when the function is maximised at periodic points, eg \cite{FFT12, FFT13}, or more generally, eg \cite{AFFR16, AFFR17, FFM18, FFM20, PS20}.  The $J_1$ topology is not equipped to handle convergence here.  Indeed, even in quite simple cases, as we note further in an example below, Skorohod's $M_1$ and $M_2$ topologies also fail to give convergence.  
We also mention recent work \cite{CF19}, which gives a very flexible topology, allowing for convergence well beyond the Skorohod topologies, but does not give any information on what may happen in a jump. 
 \cite{BPS18}, which considered certain stochastic processes, gives a route to encoding clustering through point processes decorated with a bi-infinite sequence based on the tail process (introduced in \cite{BS09}) and projection of these to generalised c\`adl\`ag functions which keep some of this information.  
Here, we obtain enriched functional limit theorems for dynamical systems by considering  a new space of functions, denoted by $F'$, to capture the details of the clustering behaviour, which are recorded in a decoration device we call the \emph{transformed anchored tail process}, which is based on the tail process introduced in \cite{BS09}. 
We remark that the design of the transformed anchored tail process is quite general encompassing not only the clustering information but also the extremal information observed `during induced periods', which allows us to apply this theory to non-uniformly hyperbolic systems admitting a nice induced map, such as, the polynomially mixing, Pomeau-Manneville (\cite{PM80}) or Liverani-Saussol-Vaienti (\cite{LSV99}) maps.  We note that in context of this paper, we always assume that we have finite mean cluster size, which is identified by an Extremal Index strictly larger than 0. (We recall that the Extremal Index is a parameter between 0 and 1, which in most cases could be thought of as the reciprocal of the mean cluster size -- see \cite{AFF20}). This assumption is vital for the transformed anchored tail process to be well defined.
 
Extreme Value Laws for heavy-tailed observables with clustering, and laws for ergodic averages of these have significant potential in applications to the theory of climate dynamics (see for example \cite{SKFG16,MCF17,MCBF18,CFMV19}). Moreover, we also believe the space $F'$ will be a very useful tool to obtain functional limits in other contexts. In fact, not only can convergence in $F'$ be proved in situations precluded by the Skorohod's  topologies ($J1,J2,M1,M2$), for which the collapse in the limit of jumps with oscillations and overshooting can bring problems, the information regarding the excursions performed by the finite time processes is still captured by the respective limits.

One possible set of future applications is a related, but different, set of problems in dynamical systems which have received attention recently (see \cite{G04a,T10,MZ15,MV20,JPZ20,JMPV21}): these are systems with bounded observables and modelled by Young towers (\cite{Y98, Y99}) with \emph{return time functions} that are not square integrable (such as the Pomeau-Manneville or Liverani-Saussol-Vaienti  maps or billiards with cusps). These correspond to a situation where the Extremal Index is 0  so that, although the observable is bounded, we get a stacking of numerous observations which add up to create a heavy tailed contribution for the ergodic sums.  These problems are partly relevant to this paper since they present similar problems with finding appropriated Skorohod topologies for convergence, but we can not directly apply the theory here since, as noted above, the transformed anchored tail process is not well defined in these situations.

\subsection{The setting and $F'$}
We consider cases of vector-valued heavy tailed observables for systems modelled by Young towers, and in the presence of general clusters of extremal observations which get collapsed in the same jump of the limiting L\'evy process. The limits that we will consider have discontinuous sample paths and therefore live in the space $D$ of c\`adl\`ag functions. However, this space is not sufficiently rich to keep a record of the fluctuations occurring during the clusters of high observations.  
The height of the limiting jump accounts for the aggregate effect of all the cluster observations. However, the oscillations observed during the cluster may exceed the height of the jump (an `overshoot'), for example. In counter to solve this loss of information, in \cite{W02}, Whitt proposed a new space that he called $E$, which decorates each discontinuity of the limit process in $D$ with an excursion corresponding to a connected set describing the maximum and minimum  fluctuations observed. In fact, when $d=1$, the excursion at the discontinuity time $t^*$ of the limiting process $V(t)$ is decorated with an interval bounded by the smallest and largest values achieved by the process during the collapsed cluster, which must contain $V(t^{*-})$ and $V(t^*)$, where $V(t^{*-})$ denotes the lefthand limit of the c\`adl\`ag function $V$ at $t^*$. 

Nevertheless, most of the information during the cluster is lost and only the maximum oscillations are recorded in $E$, while the intermediate fluctuations are completely disregarded. In  \cite{W02}, Whitt also proposed the space $F$, which keeps track in particular of the ordering in which points are visited within an excursion in $E$. However, the space $F$ still disregards information because while it keeps track of all the changes of direction during the excursions, it does not keep record of the intermediate jumps observed in the same direction. One of our main goals is to consider a space where no information collapsed into a jump is lost. For that purpose we introduce a new space that we will call $F'$. We will endow it with a metric and discuss some of its properties. Then we will use this space to study sums of heavy tailed observables, general extremal processes and records, for which we will obtain enriched functional limit theorems, all of which will be carried out with very minimal loss of information.  Another key goal is to build a theory flexible enough to handle a large class of non-uniformly hyperbolic systems modelled by Young towers.

\iffalse
A sketch of a reduced version of one of our main theorems (Theorem~\ref{thm:extremal-process}) is as follows.

\begin{theorem*}
Let $(X, T, \mu)$ be a dynamical system as in Section~\ref{subsec:systems} and $\phi:X\to [-\infty, \infty]$ a sufficiently regular observable observable with an $\alpha$-heavy tail for $0<\alpha<1$, such that the clustering behaviour of $\{\phi>s\}$ converges as $s\to \infty$.  Then there exists a scaling sequence $(a_n)_n$ and  $V(t)\in F'$ such that the variable 
$$S_n(t) =\sum_{i=1}^{\lfloor nt\rfloor}\phi\circ T^i$$
converges in distribution to $V(t)$ in $F'$, with excursions determined by the clustering behaviour.
\end{theorem*}

There is also a related result for point processes.  Both of these results extend beyond the dynamical and one-dimensional observable settings, and also to $\alpha\in [1, 2)$.

\fi

Though we have not yet fully defined our space $F'$, here we sketch a very simple example (for more details, see Example~\ref{example:oscillatory}) to show some of the features of our theory.
Define $T:[0,1]\to[0,1]$ by $T(x)=3x \mod 1$ and let $\mu=\text{Leb}|_{[0,1]}$ be our invariant measure.  Set $\psi(x)=|x-1/8|^{-2}-|x-3/8|^{-2}$ and define the one-dimensional stochastic process $X_0,X_1,\ldots$ by $X_j=\psi\circ T^j(x)$. Note that $T(1/8)=3/8$ and $T(3/8)=1/8$ and that $\psi$ is regularly varying with index $\alpha=1/2$ at points $1/8$ and $3/8$.  Then the sequence, simulated in Figures~\ref{fig:3xmod1-osc-p2} and \ref{fig:3xmod1-osc-p2-blowup},
$$S_n(t) =\frac1{16n^2}\sum_{i=1}^{\lfloor nt\rfloor}\psi\circ T^i$$
converges in our space $F'$, in such a way that the pattern of jumps seen when $x$ is close to $1/8$ and $3/8$ are recorded by the limit point in $F'$.  The way these jumps oscillate creates overshoots in the limit (most easily seen in Figure~\ref{fig:3xmod1-osc-p2-completed} and also in the form of the sum in \eqref{eq:simple_exc}), which cannot be handled by the Skorohod spaces $J_1, M_1$ or $M_2$.  While convergence is possible in Whitt's spaces $E$ and $F$, the patterns seen here would not be recorded in $E$, and the jump sizes would not be recorded in $F$ (it would not distinguish between a single jump upwards and a string of multiple jumps upwards).   The space $F'$, on the other hand, records the patterns: essentially it adds in something like the function seen in Figure~\ref{fig:3xmod1-osc-p2-blowup} as a decoration around the corresponding jump.  
\begin{figure}[t]
\centering
\includegraphics[width=0.6\textwidth]{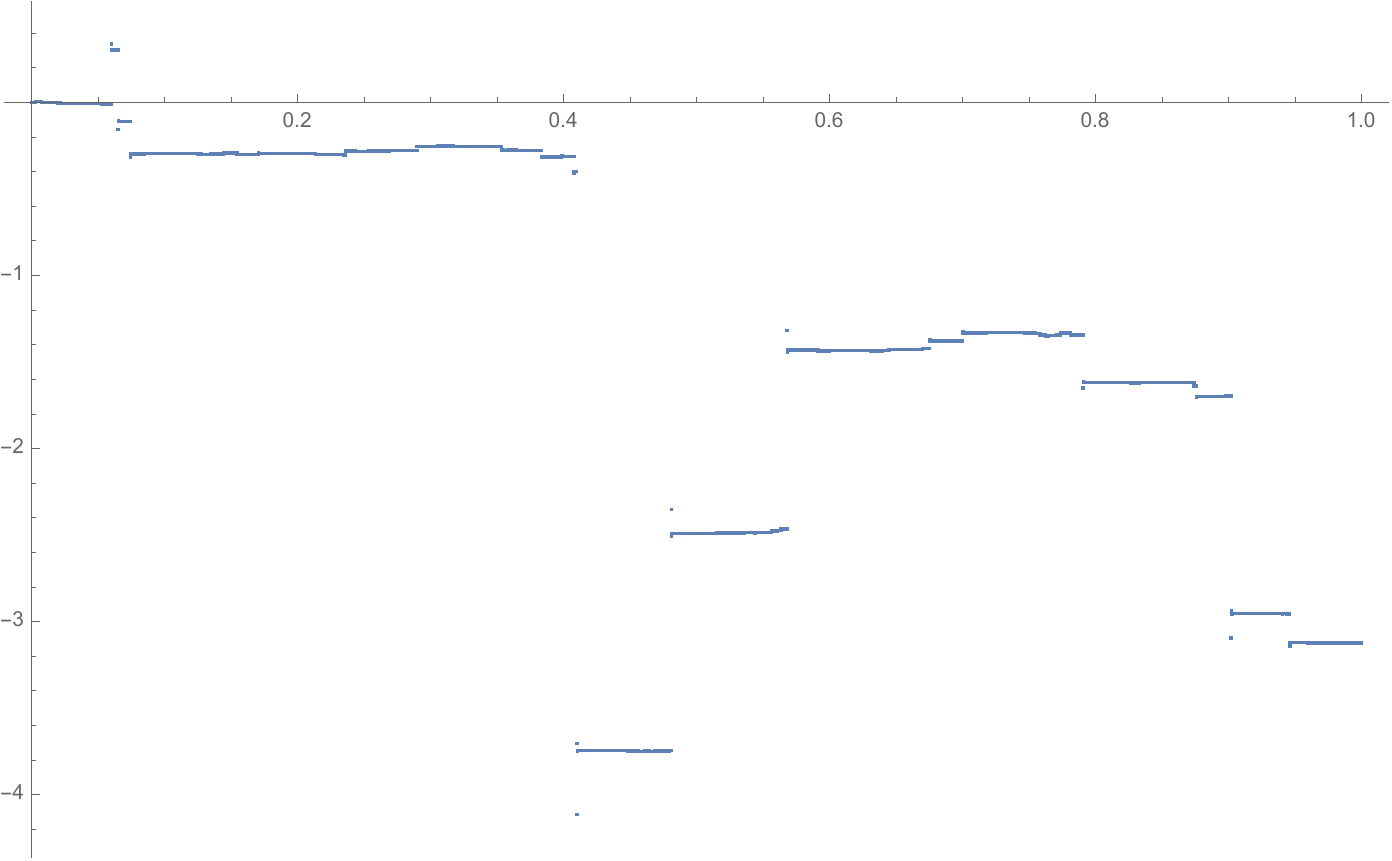}
\caption{Plot of a finite sample simulation of $S_n(t)$ with $n=5000$, where $X_j=\psi\circ T^j(x)$, where $T(x)=3x \mod 1$, $\psi(x)=|x-1/8|^{-2}-|x-3/8|^{-2}$.  
}
    \label{fig:3xmod1-osc-p2}
    \end{figure}
    
\begin{figure}[t]
\centering
\includegraphics[width=0.6\textwidth]{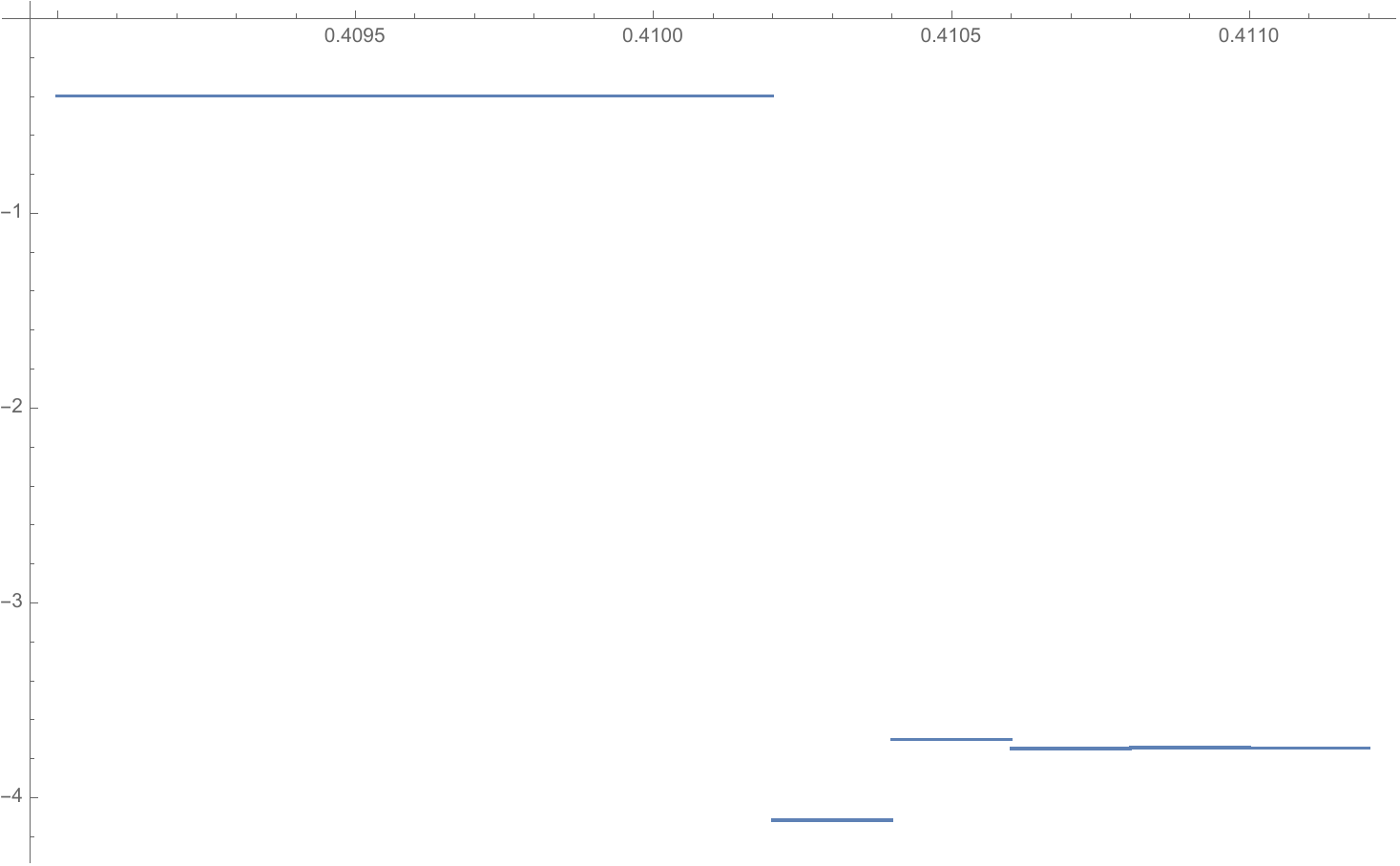}
\caption{Blowup of the previous graph at the jump observed near 0.4: asymptotically the four jumps seen here happen instantaneously, necessitating an appropriate space for convergence.}
    \label{fig:3xmod1-osc-p2-blowup}
    \end{figure}

\begin{figure}[t]
\centering
\includegraphics[width=0.6\textwidth]{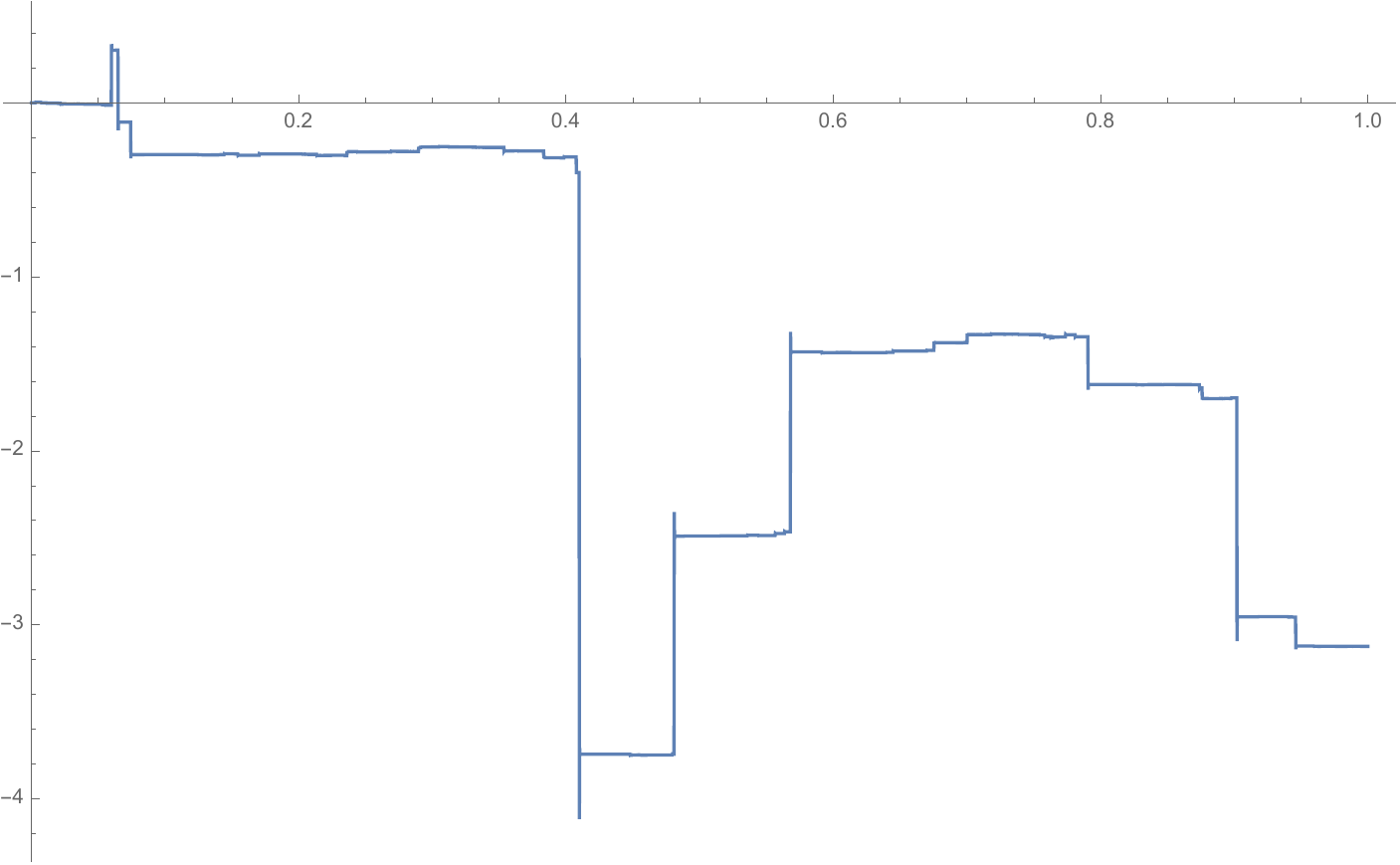}
\caption{Completed graph version of the plot in Figure~\ref{fig:3xmod1-osc-p2}, to give a different view of the jumps.  
}
    \label{fig:3xmod1-osc-p2-completed}
    \end{figure}

\subsection{Point processes, weak mixing conditions and inducing}
Point processes have been used successfully to prove functional limit theorems in $D$ (\cite{R87}). Here, since we need to keep control of the information during the excursions, we will use a new type of point processes introduced in \cite{BPS18}, where the authors considered point processes in non-locally compact spaces, designed to maintain the ordering of the cluster observations collapsed into the same instant of time, and proved their convergence, under some conditions, for stationary jointly regularly varying sequences. In \cite{BPS18}, the authors then applied this convergence of point processes to obtain convergence of sums of jointly regularly varying sequences of random variables in the space $E$. In order to prove convergence in $F'$ we will first generalise their results to more general stationary processes and under weaker conditions, in particular, under a much weaker mixing condition, which is essential to apply to processes arising from dynamical systems.   We then push this further using inducing methods, which essentially means that only the induced system needs to satisfy the mixing conditions.  In this framework, each sequence for the induced system has an attached sequence which corresponds to what happens in the uninduced case.  The point processes then have a $\Z^2$ component at each $t^*$.  These can then be projected to $F'$, incorporating all the information from the uninduced system into the limit functionals. 
We observe that the limit point processes in \cite{PS20}, in common with  \cite{FFM20}, do not record the order of observations in the cluster so our theory is more general than these previous results, even in the simple example given above.
  These generalisations may have an interest on its own, in the more classical probabilistic setting.

\subsection{Organisation}
This paper develops a large suite of new tools which apply in a wide variety of contexts.  We have tried to put our main results and applications as early as practicable in the text in order to motivate the reader and give the theory a more clear context.  This sometimes necessitates leaving full explanations and definitions for some of the results to later in the paper, but these are signposted.

In Section~\ref{sec:limit-theorems} we outline the theory from the point of view of dynamical systems, though the theory extends beyond that: we give the relevant observables, focussing on the heavy tailed cases, and briefly describe some of the basic dynamical systems models to which the theory applies.   We also define our functional spaces, culminating in the new space $F'$, and are then able to state our first limit theorem on convergence to L\'evy processes.  We then put the theory into the context of Extreme Value Theory and show convergence to the relevant extremal process.  This naturally leads to a theorem on the convergence of record point processes.

In Section~\ref{sec:point-processes} we focus on the convergence of point processes for general stochastic processes and rather general observables (not only heavy tailed). 
Conditions $\D_{q_n}$ and $\D'_{q_n}$ are given and the appropriate sequence spaces to record our exceedances are defined.  The transformed tail process is then defined which then leads to the definition of the transformed anchored tail process.  With all of these conditions met and the transformed anchored tail process existing, we then prove complete convergence of the point process to a Poisson point process.  We also show the connection to our setting of jointly regularly varying sequences.

In Section~\ref{sec:proof-FLT} we show how the results from the previous section can be applied in the dynamical context and then prove the functional limit theorems.  We conclude by developing the theory for induced systems.  In the appendices we cover some of the required background for the results here: Appendix~\ref{sec:completeness-separability} contains completely new theory for our space $F'$; Appendices~\ref{sec:app_weak_conv} and ~\ref{appendix:convergence-point-processes} give some adaptations of classical theory for sequence spaces and point processes; 
and Appendix~\ref{Appendix:dyn} contains the remaining arguments to show that the theory in this paper applies to the dynamical systems models claimed.

\section*{Acknowledgements}
ACMF, JMF were partially financed by Portuguese public funds through FCT  -- Funda\c{c}\~ao para a Ci\^encia e a Tecnologia, I.P., in the framework of the projects PTDC/MAT-PUR/28177/2017, PTDC/MAT-PUR/4048/2021, 2022.07167.PTDC and CMUP's project with reference UIDB/00144/2020. 
We would like to thank Ian Melbourne for insightful discussions on Skorohod spaces and for various suggestions for improvement of this text. We thank also the referees for useful suggestions which improved the exposition.

\section{Enriched functional limit theorems for non-uniformly hyperbolic dynamics}
\label{sec:limit-theorems}

In this section we start by introducing the setting and in particular the new space $F'$ and its properties. Then we state the main results regarding the enriched functional limits for sums of vector-valued heavy tailed observables, extremal processes and record point processes. We emphasise that the results in Section~\ref{sec:point-processes} hold well beyond the dynamical setting.

\subsection{Dynamically defined stochastic processes}

 Let $(\mathcal X,\mathcal B_{\mathcal X}, \mu, T)$ be a discrete time dynamical system, where $\mathcal X$ is a compact manifold equipped with a norm $\|\cdot\|$ (for definiteness, whenever it makes sense and unless specified otherwise, we assume that $\|\cdot\|$ is the usual Euclidean norm), 
 $\mathcal B_{\mathcal X}$ is its Borel $\sigma$-algebra, $T:{\mathcal X}\to{\mathcal X}$ is a measurable map and $\mu$ is a $T$-invariant probability measure, \ie $\mu(T^{-1}(B))=\mu(B)$ for all $B\in\mathcal B_{\mathcal X}$. Let $\Psi:{\mathcal X}\to\R^d$ be an observable (measurable) function and define the stochastic process $\X_0, \X_1,\ldots$ given by
\begin{equation}
\label{eq:dynamics-SP}
\X_n=\Psi\circ T^n,\qquad\text{for every $n\in\N_0$}.
\end{equation}
High values of $\|\Psi(\cdot)\|$ will correspond to entrances in a neighbourhood of a zero measure maximal set $\mathcal M$, which we express in the following way. Let $\mathcal M\subset \mathcal X$ be such that $\mu(\mathcal M)=0$ and let $g:[0,\infty)\to\R\cup\{+\infty\}$ be such that $0$ is a global maximum, where we allow $g(0)=+\infty$, and $g$ is a strictly decreasing bijection in a neighbourhood of $0$.
We assume that, on a neighbourhood of $\mathcal M$, 
\begin{equation}
\label{eq:dynamics-SP-norm}
\|\Psi(x)\|=g(\dist(x,\mathcal M)), \quad\text{where $\dist(x,\mathcal M)=\inf\{\dist(x,\zeta)\colon \;\zeta\in\mathcal M\}$.}
\end{equation}
where $g$ has one of the three types of behaviour: 
\begin{enumerate}
\item[Type $g_1$:] there exists some strictly positive function\footnote{A possible choice for $h$ is given in \cite[Chapter~4.2.1]{LFFF16}.} 
$h:W\to\R$ such that for all $y\in\R$
\begin{equation}\label{eq:def-g1}\displaystyle \lim_{s\to
g_1(0)}\frac{g_1^{-1}(s+yh(s))}{g_1^{-1}(s)}=\e^{-y};
\end{equation}
\item [Type $g_2$:] $g_2(0)=+\infty$ and there exists $\alpha>0$ such that
for all $y>0$
\begin{equation}\label{eq:def-g2}\displaystyle \lim_{s\to+\infty}
\frac{g_2^{-1}(sy)}{g_2^{-1}(s)}=y^{-\alpha};\end{equation}
\item [Type $g_3$:] $g_3(0)=D<+\infty$ and there exists $\gamma>0$ such
that for all $y>0$
\begin{equation}\label{eq:def-g3}\lim_{s\to0}
\frac{g_3^{-1}(D-sy)}{g_3^{-1}(D-s)}= y^\gamma.
\end{equation}
\end{enumerate}

We remark that these three types of limit behaviours for $g$ are directly connected with the three types of Extreme Value Laws given by the Fisher-Tippett-Gnedenko Theorem. We refer to \cite[Section~4.2.1]{LFFF16} for further details on this relationship.

Most of the results regarding hitting times and extreme values for dynamical systems were obtained when $\mathcal M$ is reduced to a single point $\zeta$. Recent results have considered $\mathcal M$ to be a countable set (\cite{AFFR17}), submanifolds (\cite{FGGV18, CHN21}) and fractal sets (\cite{MP16,FFRS20}). Our results can be applied to general maximal sets $\mathcal M$ and general $\Psi$, under the assumption that the transformed anchored tail process is well defined, which will be verified and illustrated for the more common case where $\mathcal M=\{\zeta\}$, for some hyperbolic point $\zeta\in\mathcal X$, and where $\Psi$, on neighbourhood of $\zeta$, can be written as 
\begin{equation}
\label{eq:dynamics-particular-observable}
\Psi(x)=g(\dist(x,\zeta))\frac{\Phi^{-1}_\zeta(x)}{\|\Phi^{-1}_\zeta(x)\|}\I_W(x),
\end{equation}
where  $\Phi_\zeta:V\to W$ denotes a diffeomorphism, defined on an open ball $V$ around zero in $T_\zeta\mathcal X$, the tangent space at $\zeta$, onto a neighbourhood $W$ of $\zeta$ in $\mathcal X$, such that $\Phi_\zeta(E^{s,u}\cap V)=W^{s,u}(\zeta)\cap W$.
\begin{remark}
If $g$ is of type $g_2$ (for example, $g(x)=x^{-1/\alpha}$), the measure is sufficiently regular and the geometry of the maximal set $\mathcal M$ is simple (for example, $\mathcal M=\{\zeta\}$, for some $\zeta\in \mathcal X$), then the distribution of $\X_0$ is regularly varying (see \cite[Section~4]{LFFF16}, for example). 
\end{remark}

\subsection{Applications to specific systems}
\label{subsec:systems}

The theory developed in this paper applies to general stochastic processes, but the applications we focus on are dynamical systems.  In this subsection we give some preliminary examples of such applications ranging from systems with good exponential mixing properties to some with poor mixing behaviour, leaving further applications to future works.  Since some aspects of the proofs of the facts used here require the establishment of some new tools, we postpone these to Section~\ref{Appendix:dyn} (see also the discussion following Theorem~\ref{thm:point-process-convergence}).

\subsubsection{Non-uniformly expanding systems}
\label{sssec:non-invertible}
\hfill\\
\textbf{Uniformly expanding systems.}
We first list some well-behaved dynamical systems which are essentially uniformly expanding, though they need not have Markov, or compactness, properties.  More details about the required properties are at the beginning of Section~\ref{sec:proof-FLT}.

\begin{itemize}
\item Uniformly expanding continuous maps of the interval/circle;
\item Markov maps;
\item Piecewise expanding maps of the interval with countably many branches such as Rychlik maps;
\item Saussol's class of higher-dimensional expanding maps

\end{itemize}

Here we are always assuming that $\mu$ is an absolutely continuous (with respect to Lebesgue) invariant probability measure or \emph{acip} since these are a very natural class, though the theory extends beyond these.

Non-trivial examples of observables on these systems to which we can apply the theory 
include those maximised at a repelling periodic point, first studied in \cite{FFT12}, but there is huge scope to study further clustering behaviour such as that shown in  \cite{AFFR16, AFFR17, FFRS20} for example.  We require that the density $\frac{d\mu}{\text{Leb}}$, where $\text{Leb}$ is Lebesgue, is bounded at the periodic point so that the conformal properties of $T$ there are reflected in the measure $\mu$ as well as $\text{Leb}$, but this is automatic in the Rychlik case.  

In the heavy-tail applications we will restrict ourselves to $0<\alpha<1$ for simplicity, but observe that techniques to prove \eqref{eq:small-jumps}, required in the $1\le \alpha<2$ case, are provided in \cite{T10}: these immediately apply to some of the simpler cases above. Conditions  \eqref{eq:Q-condition->1} and \eqref{eq:Q-condition-=1} are easy computations in the periodic case.

Finally for this introductory discussion on dynamical applications, we note that there is a condition on $(Q_j)_j$ in Theorem~\ref{thm:record-theorem} which is easily satisfied in all the dynamical examples, see Appendix~\ref{Appendix:dyn}.

\textbf{Benedicks-Carleson quadratic maps.} 
Here we provide a class of maps which are far from uniformly expanding, indeed there are critical points.  

Here we set $I=[-1, 1]$ and for $a\in (0, 2]$ define $f_a:I\to I$ by $f_a(x) = 1-ax^2$.  This map satisfies the \emph{Benedicks-Carleson} conditions (here $a$ is close to 2) if there exists $c, \gamma>0$ ($c$ should be close to $\log 2$ and $\gamma$ is small) such that 
$$|Df_a^n(f(0))|\ge e^{cn} \text{ for } n\in \N_0 \text{ and } |f_a^n(0)|\ge e^{-\gamma\sqrt n}\text{ for } n\in \N.$$
It is known that there is a positive Lebesgue measure set $BC$ of $a$ such that $f_a$ satisfies these conditions.

 When for definiteness we consider that our observable is maximised at a periodic point of $f_a$, we note that the fact that the density at our periodic point is bounded is more delicate than above, but the existence of suitable periodic points for (a positive Lebesgue measure set of) maps in $BC$ is shown in \cite[Section 6]{FFT13}.

\textbf{Manneville-Pomeau maps.}
Our previous examples all have exponential decay of correlations, but our final interval map example shows that this is not necessary.  This is what is often referred to as the Liverani-Saussol-Vaienti (LSV) version of the Manneville-Pomeau map: for $\gamma\in (0,1)$, define $T_\gamma:[0, 1]\to [0, 1]$ by
\begin{equation}
\label{eq:LSV}
T_\gamma(x):= \begin{cases} x(1+2^\gamma x^\gamma) &\text{ if } x\in [0, 1/2),\\
2x-1&\text{ if } x\in [1/2, 1].\end{cases}
\end{equation}
This has an acip $\mu=\mu_\gamma$ and (sharp) polynomial decay of correlations for H\"older observables against observables in $L^\infty(\mu)$, but does not have decay of correlations for observables on some Banach space against all observables in $L^1(\mu)$. (We recall that, as seen in \cite[Theorem~B]{AFLV11}, summable decay of correlations against all $L^1(\mu)$ observables is a strong property which in particular would imply the existence of exponential decay of correlations). To handle this particular case, we introduce in Section~\ref{subsec:inducing} a new type of point processes encompassing the idea of inducing. This will allow us to obtain the same functional limit theorems in the full range $\gamma\in(0,1)$, for observables of the type $\varphi\colon [0,1]\to\R$, where $\varphi(x)=g(\dist(x,\mathcal M))$, $g(x)=x^{-1/\alpha}$ and $\mathcal M$ is finite set of points, possibly periodic. For simplicity, we consider again that $0<\alpha<1$ .

\subsubsection{Invertible hyperbolic systems}\label{subsec:anosov}

We also consider hyperbolic invertible systems consisting of Anosov linear diffeomorphism,  $T:\mathbb T^2\to \mathbb T^2$, defined on the flat torus  $\mathbb{T}^2=\mathbb{R}^2/\mathbb{Z}^2$. We can associate $T$ to a $2\times2$ matrix, $L$, with integer entries, determinant $1$,  and without eigenvalues of absolute value $1$. 
As the determinant of $L$ is equal to $1$, the Riemannian structure induces a Lebesgue measure on $\mathbb{T}^2$ which is invariant by $T$. These systems are Bernoulli and have exponential decay of correlations with respect to H\"older observables.

\subsection{The functional spaces}
\label{subsec:functional-spaces}

Let $D=D([0, 1], \R^d)$ be the space of c\`adl\`ag functions defined on $[0,1]$. 
One can define several metrics in $D$. The most usual is the so-called $J_1$ Skorohod's metric, which generalises the uniform metric by allowing a small deformation of the time  scale.  In this metric a jump in the limit function must be matched by a similar one in the approximating functions. In order to establish limits with unmatched jumps, Skorohod introduced the $M_1$ and $M_2$ topologies which use completed  graphs 
of the functions. We will use a metric motivated by $M_2$, which considers the Hausdorff distance between compact sets. We refer to \cite{W02} for precise definitions and properties.

In order to keep some of the information collapsed in limit jumps, and to broaden the class of convergent sequences, Whitt introduced the space $E=E([0, 1], \R^d)$ (see  \cite[Sections~15.4 and 15.5]{W02}), as the space of excursion triples
$$
\left(x,S^x,\{I(s)\}_{s\in S^x}\right),
$$
where $x\in D$, $S^x$ is a countable set containing the discontinuities of $x$, denoted by $disc(x)$, \ie $disc(x)\subset S^x$, and, for each $s\in S^x$,  $I(s)$ is a compact connected subset of $\R^d$ containing at least $x(s^-)$ and $x(s)$. We may identify each element of $E$ with the set-valued function
\begin{equation}
\label{eq:set-function}
\hat x(t)=\begin{cases}
    I(t)& \text{if } t\in S^x\\
    \{x(t)\}              & \text{otherwise}
\end{cases}, 
\end{equation}
and its graph $\Gamma_{\hat x}=\{(t,z)\in [0,1]\times\R^d\colon\; z\in \hat x(t)\}$.  Letting $p_\ell:\R^d\to\R$ denote the projection onto the $\ell$-th coordinate for $\ell=1, \ldots, d$, we define $\hat x^\ell(t)=p_\ell(\hat x(t))$ and $\Gamma^\ell_{\hat x}=\{(t,z)\in [0,1]\times\R\colon\; z\in \hat x^\ell(t)\}$.

We embed $D$ into $E$, in the following way. For $a,b\in\R^d$, we define the product segment
$$
[[a,b]]:=[a^1,b^1]\times\cdots\times[a^d,b^d],
$$
where the one-dimensional segment $[a^\ell,b^\ell]$ coincides with the interval\footnote{Recall the notation, used throughout this paper, $x\wedge y=\min\{x, y\}$ and $x\vee y=\max\{x, y\}$} $[a^\ell\wedge b^\ell,a^\ell\vee b^\ell]$. If we have $a^\ell=b^\ell$, then $[a^\ell,b^\ell]=\{a^\ell\}=\{b^\ell\}$. We identify $x\in D$ with the element of $E$
$$
\left(x, disc(x), \{I(s)\}_{s\in disc(x)}\right),\quad \text{where $I(s)=[[x(s^-),x(s)]]$}.
$$

We use the Hausdorff metric to define a metric on $E$. Namely, recall that for compact sets $A,B\subset \R^d$, the Hausdorff distance between $A$ and $B$ is given as
$$
m(A,B)=\max\left\{\sup_{x\in A}\left\{\inf_{y\in B}\|x-y\|\right\}, \sup_{y\in B}\left\{\inf_{x\in A}\|x-y\|\right\} \right\}.
$$
For $A\subset \R^d$, let $d(A)=\sup_{x,y\in A}\{\|x-y\|\}$ be the diameter of $A$. In order to use the Hausdorff metric, we assume that the elements of $E$ satisfy the condition:
\begin{equation}
\label{eq:finitebigjumps} 
 \text{for all $\epsilon >0$, there exist only finitely many $s\in S^x$ such that $d(I(s))>\epsilon$.} 
 \end{equation}
 This guarantees that for each element of $E$, the associated graph $\Gamma_{\hat x}$ is a compact set. This way, we endow $E$ with the Hausdorff metric simply by establishing that
\begin{equation}
\label{eq:me}
m_E(\hat x, \hat y)=\max_{\ell=1, \ldots, d}m(\Gamma^\ell_{\hat x},\Gamma^\ell_{\hat y}).
\end{equation}
Endowed with this metric $E$ is separable but not complete. Alternatively, we can define the stronger metric, called uniform metric given by
\begin{equation}
\label{eq:me*}
m_E^*(\hat x, \hat y)=\max_{\ell=1, \ldots, d}\sup_{t\in [0,1]}  m(\hat x^\ell(t), \hat y^\ell(t)).
\end{equation}
When endowed with the metric $m_E^*$, the space $E$ is complete but not separable. We refer to \cite{W02} for further details about $E$ and its properties.

The space $E$ only records the maximal oscillations when information collapses into a jump in the limit. In order to keep a closer track of the fluctuations during the excursions, Whitt introduced the space $F$,  in \cite[Section~15.6]{W02}, which corresponds to the set of equivalence classes of the set of all the parametric representations of the graphs $\Gamma_{\hat x}$ of elements $\hat x$ of $E$, by setting that two parametrisations (continuous functions from $[0,1]$ into $\Gamma_{\hat x}$)  $(r_1,u_1)$ and $(r_2,u_2)$ are equivalent if there exist continuous nondecreasing onto functions $\lambda_1,\lambda_2:[0, 1]\to [0, 1]$ such that  $(r_1,u_1)\circ\lambda_1=(r_2,u_2)\circ\lambda_2$. This means that, in particular, the two functions $u_1,u_2$ in the two equivalent parametric representations of $\Gamma_{\hat x}$ visit all the points of $\Gamma_{\hat x}$ in each $I(s)$ the same number of times and in the same order. However, note that $F$ still misses some fluctuations such as intermediate jumps in the same direction which give rise to a big jump in the limit. In $F$, for example, we do not distinguish between a big jump or a collection of smaller jumps which aggregate to perform an excursion with the same big jump, in the limit. See the excursions given in \eqref{eq:excursions-F-F'} and the discussion that follows it.  

\subsubsection{The new functional space $F'$ recording all fluctuations}
\label{subsec:F'}

In order to keep track of all the fluctuations without missing information we introduce the space $F'=F'([0,1])$. We start by considering $\tilde D= \tilde D([0, 1], \R^d)= D([0, 1], \R^d)/{\sim}$ where $x\sim y$ if there exists a reparametrisation $\lambda:[0, 1]\to [0, 1]$, i.e., a continuous strictly increasing bijection such that $x\circ \lambda=y$.  Denote the equivalence class of $x$ by $[x]$.  We define
$$d_{\tilde D}([x], [y]) = \inf_{\lambda}\|x\circ \lambda- y\|,$$
where $\|\cdot\|$ is the supremum norm and $\lambda$ is the set of continuous strictly increasing bijections of $[0,1]$ to itself (this could be thought of as the induced metric from the $J_1$ metric on $\tilde D$).

We abuse notation within $\tilde D$ by writing $x$ to refer to both a representative of its equivalence class $[x]$ and the equivalence class itself.

Now define 
$$F':=\left\{\ux=\left(x, S^x, \{e_x^s\}_{s\in S^x}\right)\right\},$$
where $x\in D([0,1], \R^d)$, $S^x\subset [0,1]$ is an at most countable set containing $disc(x)$, the discontinuities of $x$ and $e_x^s\in \tilde D([0,1], \R^d)$  is the excursion at $s\in S^x$, which is such that $e_x^s(0)=x(s^-)$ and $e_x^s(1)=x(s)$.

We can embed 
$D$ into $F'$, by associating to $x\in D([0,1],\R^d)$ the element 
\begin{equation}
\label{eq:embedding-F}
\ux=\left(x, disc(x), \{e_x^s\}_{s\in disc(x)}\right),
\end{equation} where $e_x^s\in \tilde D([0,1], \R^d)$ is the equivalence class represented, for example, by $$e_x^s(t)=x(s^-)+(x(s)-x(s^-))\I_{[1/2,1]}(t),\quad t\in[0,1].$$

We project $F'$ into Whitt's space $E$ and into $\tilde D$, which will give us a metric and a space with more information than Whitt's space $F$.  

Let $\pi_E(\ux)=\ux^E:=\left(x, S^x, \{I(s)\}_{s\in S^x}\right)$ where
$$
I(s)=\left[\inf_{t\in [0,1]} e_x^{s,1}(t), \sup_{t\in [0,1]} e_x^{s,1}(t)\right]\times\cdots\times \left[\inf_{t\in [0,1]} e_x^{s,d}(t), \sup_{t\in [0,1]} e_x^{s,d}(t)\right], 
$$
with $e_x^{s,\ell}(t)=p_\ell(e_x^{s}(t))$.
We project $F'$ into $\tilde D$ as follows.
 Suppose that $S^x$ is countable  and write $S=\{s_i\}_{i=1}^\infty$, since the finite case follows more straightforwardly.
Let $0=a_1<a_2<\cdots <1$ be such that $a_i\to 1$ as $i\to \infty$ (the choice of $(a_n)_n$ really is arbitrary).  We insert the intervals $[a_i, a_{i+1}]$ at the points $s\in S^x$.  This is a simple idea, but since $S$ may be complicated, we need some notation.  Define for $i\in \N$, 
$$\bar a_i:= \sum_{s_j\le s_i} (a_{j+1}-a_j), \; c_i:=s_i+\bar a_i-(a_{i+1}-a_i), \; d_i:=s_i+\bar a_i, \; \bar t:=\sup\{\bar a_i:s_i<t\}.$$
Thus $[c_i, d_i]$ is the interval corresponding to $s_i$ and $\bar t$ is the accumulated length of the intervals inserted before $t$, so we can think of $t$ being shunted to $t+\bar t$.  Note also that  our time line is now of length 2, so we will need to rescale back to a length 1 interval. We then define a representative of the equivalence class $\tilde\pi(\ux)$ by
\begin{equation}
\label{eq:projection-Dtilde}
\ux^{\tilde D}(t)= \begin{cases} x(2t-\bar t) &\text{ if } 2t\notin \cup_i[c_i, d_{i}],\\
e_x^{s_i}\left(\frac{2t-c_i}{d_i-c_i}\right) &\text{ if } 2t\in [c_i, d_{i}].
\end{cases}
\end{equation}

Define
\begin{equation}
\label{eq:dF-prime}
d_{F'}(\ux, \uy)=d_E(\pi_E(\ux), \pi_E(\uy))+ d_{\tilde D}(\tilde \pi(\ux), \tilde\pi(\uy)),
\end{equation}
where $d_E$ denotes $m_E$ in \eqref{eq:me}.  Note that we could have chosen $m_E^*$ defined in \eqref{eq:me*} rather than $m_E$ here (see Proposition~\ref{prop:compsep} and discussions in \cite[Section 5.4]{W02} and \cite[Section 4.1]{BPS18}), but we fix $m_E$ for definiteness and a more direct comparison with \cite{BPS18}.

Note that we could also project into Whitt's $F$ space here (using $\pi_E$, but with $ \{e_x^s\}_{s\in S^x}$ to give the order of the parametrisation in $I(s)$), and that convergence in $F'$ implies convergence in $F$.  In $F'$ we keep the information of the  displacement in $\R^d$ of all (including intermediate) jumps in the discontinuities, while $F$ only keeps the order and information on the `local range', \ie it only captures local extrema.  To see this in a 1d example note that the excursions denoted
\begin{equation} 
\label{eq:excursions-F-F'}
e_x(t)= \begin{cases} 0 & \text{ if } 0\le t<\frac13,\\
 \frac12& \text{ if } \frac13 \le t<\frac23,\\
 1 & \text{ if } \frac23\le t\le 1,\\
\end{cases}
\qquad e_y(t)= \begin{cases} 0 & \text{ if } 0\le t<\frac13,\\
 1 & \text{ if } \frac13\le t\le 1,\\
\end{cases}
\end{equation}
yield the same representations as part of $F$: namely the line $[0,1]$ with any parametrisation which is an orientation preserving homeomorphism.  However, as components of $F'$ they are distinct with $d_{\tilde D}(e_x, e_y)=\frac12$.  Indeed, if $\ux$ and $\uy$ differ only by having a discontinuity $s$ having $e_x$ and $e_y$ as the corresponding excursion respectively, then $d_{F'}(\ux, \uy) =\frac12$.

\begin{figure}[h]
\begin{tikzpicture}[thick, scale=0.5]

\draw[ -> ] (0,0) -- (17,0);

\draw[ dashed] (1, 3) -- (3,3);
\draw[ -o ] (3, 3) -- (4,3);
\fill (3.8,6) circle (5.5pt);
\draw[ -o ] (4, 6) -- (6,6);

\fill (5.8,7.5) circle (5.5pt);
\draw[ -o ] (6, 7.5) -- (8,7.5);

\fill (7.8,8.5) circle (5.5pt);
\draw[ -o ] (8, 8.5) -- (10,8.5);

\fill (9.8,9) circle (5.5pt);
\draw[ -] (10, 9) -- (11,9);
\draw[ dashed] (11, 9) -- (13,9);

\draw[-] (4, 0)--(4, -0.3);
\draw (4, -0.3) node[below] { \small $j$};

\draw[-] (6, 0)--(6, -0.3);
\draw (6, -0.3) node[below] {\small  $j+1$};

\draw[-] (8, 0)--(8, -0.3);
\draw (8, -0.3) node[below] {\small $j+2$};

\draw[-] (10, 0)--(10, -0.3);
\draw (10, -0.3) node[below] {\small $j+3$};

\end{tikzpicture}

\caption{A piece of $S_n(t)$ (see \eqref{eq:sum})  in the Gou\"ezel example \cite{G08} where there is a close approach to zero at time $1\le j\le n-3$.}
\label{fig:Gou}

\end{figure}
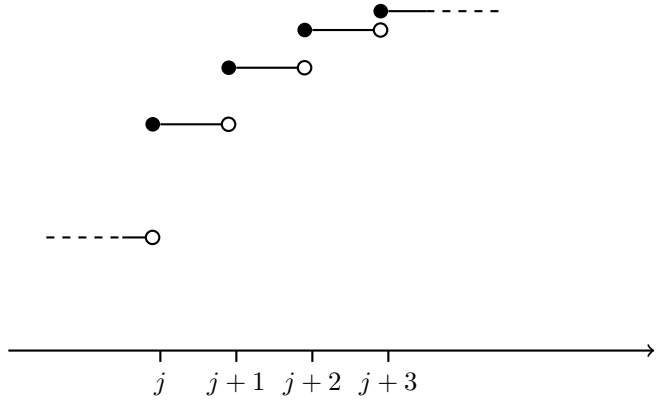

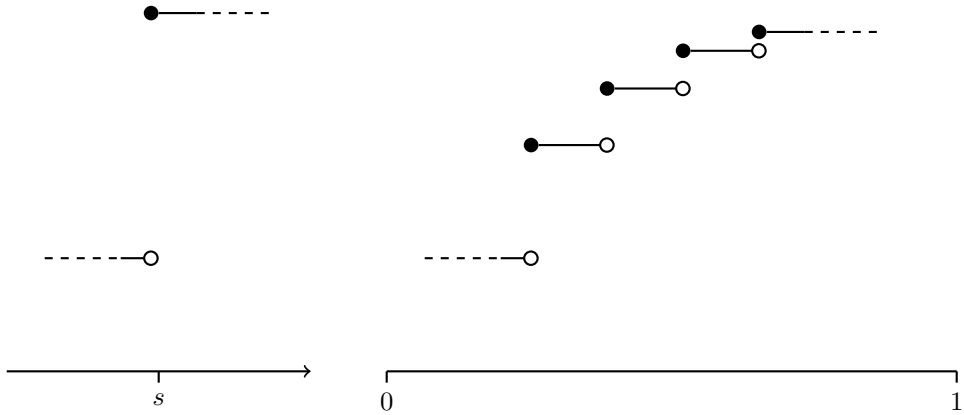
\begin{figure}[h]
\begin{tikzpicture}[thick, scale=0.5]

\draw[ - >] (0,0) -- (8,0);

\draw[ dashed] (1, 3) -- (3,3);
\draw[ -o ] (3, 3) -- (4,3);

\fill (3.8,9.5) circle (5.5pt);
\draw[ -] (4, 9.5) -- (5,9.5);
\draw[ dashed] (5, 9.5) -- (7,9.5);

\draw[-] (4, 0)--(4, -0.3);
\draw (4, -0.3) node[below] { \small $s$};

\draw[ -] (10,0) -- (25,0);
\draw[-] (10, 0)--(10, -0.3);
\draw (10, -0.3) node[below] {\small $0$};
\draw[-] (25, 0)--(25, -0.3);
\draw (25, -0.3) node[below] {\small $1$};

\draw[ dashed] (11, 3) -- (13,3);
\draw[ -o ] (13, 3) -- (14,3);
\fill (13.8,6) circle (5.5pt);
\draw[ -o ] (14, 6) -- (16,6);

\fill (15.8,7.5) circle (5.5pt);
\draw[ -o ] (16, 7.5) -- (18,7.5);

\fill (17.8,8.5) circle (5.5pt);
\draw[ -o ] (18, 8.5) -- (20,8.5);

\fill (19.8,9) circle (5.5pt);
\draw[ -] (20, 9) -- (21,9);
\draw[ dashed] (21, 9) -- (23,9);

\end{tikzpicture}

\caption{A sequence of jumps as in the previous figure, after rescaling in time and space, may converge to a part of $(x, S, \{e^s\}_{s\in S})\in F'$.  We sketch the $x$ part on the left and a representative of the corresponding $e^s$ on the right (this belongs to an equivalence class up to time rescaling).  Clearly this convergence cannot take place in $J_1$, and, while convergence does hold in  $M_1$ or $M_2$ (and indeed in $E$ and $F$), the height of the individual (rather than cumulative) jumps will not be recorded.}
\label{fig:Gou-blow}

\end{figure}

\begin{remark}
\label{rem:non-compact-domain}
The definition of the space and topology can be generalised trivially to other compact time domains such as $[t_1,t_2]$, with $0\leq t_1<t_2$. In order to consider a notion of convergence on non-compact domains such as   $F'((0,\infty), \R^d)$, we say that $\underline x_n \to \ux$, in $F'((0,\infty), \R^d)$ if the same holds for the respective restrictions to $F'([t_1,t_2], \R^d)$, for all  $0\leq t_1<t_2$ such that $t_1,t_2\notin S^x$.
\end{remark}

We illustrate schematic versions of a sequence of elements in $F'$ in Figure~\ref{fig:Gou} which converge to the element of $F'$ in Figure~\ref{fig:Gou-blow}.  The jumps stack up to one big jump, but the jumping behaviour is recorded in $F'$.  Figures~\ref{fig:3xmod1-osc-p2}, \ref{fig:3xmod1-osc-p2-blowup} show more complex behaviours in simulations.
 
We discuss some properties such as completeness and separability of the space $F'$ in Appendix~\ref{sec:completeness-separability}.

\subsection{Rare events and point processes}
\label{subsec:rare-events-PP}

Since \cite{LWZ81, D83}, it has been known that the behaviour of the mean of heavy tailed processes is determined by the extremal observations. Hence, as in the context of extremal processes and records, we are lead to the study of rare events corresponding to the occurrence of abnormally large observations. In particular, this means that we need to impose some regularity of the tails of the distributions.

\subsubsection{Normalising threshold functions}
\label{subsec:normalising}

We assume that the stationary sequence of random variables $\X_0, \X_1, \ldots$ has proper tails, in the sense that there exists a normalising sequence of threshold functions $(u_n)_{\n\in\N}:\R^+\to \R^+$, where $\R^+=(0,+\infty)$, satisfying the following properties (see \cite{H87}):
\begin{enumerate}

\item For each $n$, the function $u_n$ is nonincreasing, left continuous and such that
$$
\lim_{\tau_1\to0,\tau_2\to\infty}\p\left(u_n(\tau_2)<\|\X_0\|<u_n(\tau_1)\right)=1;
$$

\item For each, $\tau\in\R^+$,
\begin{equation}
\label{un}
\lim_{n\to\infty} n\p\left(\|\X_0\|>u_n(\tau)\right)= \tau.
\end{equation}

\end{enumerate}
We observe that \eqref{un} requires that the average number of exceedances of $u_n(\tau)$, \ie events of the type $\|\X_j\|\geq u_n(\tau)$, for $j=0,\ldots,n-1$, is asymptotically constant and equal to $\tau>0$, which can be interpreted as the asymptotic frequency of exceedances.  The nonincreasing nature of $u_n$ reflects the fact that the higher the frequency $\tau$ of observed exceedances, the lower the corresponding threshold $u_n(\tau)$ should be. 

For every $z\in\R^+$, we define
\begin{equation}
\label{un-1}
u_n^{-1}(z)=\sup\{\tau>0\colon\; z\leq u_n(\tau)\}
\end{equation}
Observe that for each value $z$ in the range of the r.v. $\|\X_0\|$, the function $u_n^{-1}$ returns the asymptotic frequency $\tau=u_n^{-1}(z)$ that corresponds to the average number of exceedances of a threshold placed at the value $z$, among the $n$ \iid observations of the r.v. $\|\X_0\|$.

Also note that for all $\tau,s\in\R^+$, 
\begin{equation}
\label{eq:un-1-property}
u_n^{-1}(s)<\tau\quad\mbox{if and only if}\quad s>u_n(\tau).
\end{equation}

\subsubsection{Point processes of rare events}
Multidimensional point processes are a powerful tool to record information regarding rare events (see for example \cite{P71,R87}), which can then be used to study record times (\cite{R75,R87}), 
extremal processes \cite{D64, L64, P71, R75, HT19}, sums of heavy tailed random variables \cite{D83,DH95,T10,T10a, BKS12,BPS18}. In particular, they are very useful to keep track of the information within the clusters \cite{M77,H87,DH95,N02, BKS12,BPS18, FFM20, PS20}. More specifically, the Mori-Hsing characterisation tells us that in nice situations the limiting process can essentially be described as having two components: a Poisson process determining the occurrences of clusters and an ``orthogonal'' point process describing the clustering structure. The description of the clustering component can be accomplished in different ways. In \cite{N02} a natural generalisation of a compound Poisson point process is used. In \cite{FFM20}, the authors use an outer measure to describe the piling of observations at the base cluster points of the Poisson process, which in the dynamical systems setting can be computed based on the action of the derivative of the map generating the dynamics. In \cite{BS09}, the authors introduced the \emph{tail process}, which is a mechanism to describe the clusters and was subsequently used in \cite{BKS12,BPS18,KS20}, for example. 

Here, we are going to use a description provided by the \emph{transformed anchored tail process}, 
 which is an adaptation 
 of the tail process introduced in \cite{BS09} (see Section~\ref{subsec:regular-variation} for the relation between the two).  Devices of this kind are very well understood in the probabilistic literature, and we refer in particular to a recent work \cite{BP21} (and references therein) where the authors consider anchored tail processes in a more general framework of processes indexed over integer lattices (in our setting the index is restricted to $\Z$). 
 We will show how the transformed anchored tail process relates with the outer measure of \cite{FFM20} and compute it in the context of dynamical systems.\footnote{Due to this piling phenomenon, in an earlier preprint version of this paper, we also referred to the anchored tail process as `piling process'.}

Although we defer the formal definitions of point processes and point processes of rare events designed to keep all clustering information  to Section~\ref{subsubsec:complete-convergence} and  Appendix~\ref{appendix:convergence-point-processes}, we give here a brief description of the latter, which again has two components. The first is an underlying Poisson process on $\R_0^+\times\R_0^+$ with intensity measure $\leb\times\theta\leb$, which can be represented by  
$$M=\sum_{i=1}^\infty\delta_{(T_i,U_i)}, \quad\text{where $\delta_x$ denotes the Dirac measure at $x\in \R_0^+\times\R_0^+$},$$
so that for any measurable disjoint sets $A, A_1, \ldots, A_\ell \subset \R_0^+\times\R_0^+$, we have that $M(A)$ is a Poisson distributed random variable of intensity $\leb\times\theta\leb(A)$ and $M(A_1),\ldots, M(A_\ell)$ are mutually independent. The parameter $\theta$ is defined formally in Section~\ref{subsubsec:EI} and can be thought of as the reciprocal of the average number of exceedances in a cluster.  The sequence $(T_i)_i$ in the first component are points in time and the second component is the angular component of the spectral decomposition of the transformed anchored tail process, which will be defined in Section~\ref{subsubsec:piling-process}. For each time $T_i$ there is a bi-infinite sequence $(Q_{i, j})_{j\in\Z}$ which will decorate the second coordinate of the mass point of $M$ at time $T_i$ and is such that $\|Q_{i, j}\|\to\infty$, as $|j|\to\infty$, 
$\min_{j\in\Z} \|Q_{i,j}\|=1$.  For a given $i$, the sequence $(Q_{j})_{j\in\Z}= (Q_{i, j})_{j\in\Z}$ is given in \eqref{eq:Z-polar}. The sequences $(Q_{i,j})_{i,\in\N,j\in\Z}$ are mutually independent and also independent of the sequences $(T_i)_i$ and $(U_i)_i$.  The distribution of each transformed anchored tail process, which  corresponds to a sequence $(U_iQ_{i,j})_{j\in\Z}$ at $T_i$,  is designed to capture the behaviour of the observations within a cluster of exceedances, which was initiated at (`vertical') time $j=0$ and whose most severe exceedance has a corresponding asymptotic frequency given by $U_i$, in the sense of the interpretation we provided for the $u_n^{-1}$ function given in \eqref{un-1} (recall that the larger the exceedance, the smaller the corresponding asymptotic frequency).

\begin{remark}
\label{rem:Qj-periodic-point}
 In order to have some intuition regarding the sequence $(Q_{j})_{j\in\Z}$, we mention that in the case $\Psi$ has the form \eqref{eq:dynamics-particular-observable}, with $\mathcal M$ reduced to a repelling fixed point $\zeta$ where the invariant density is sufficiently regular, then, in the non-invertible case, $Q_{0}$ has a uniform distribution on $\mathbb S^{d-1}$, the unit sphere in $\R^d$, and for all $j\in\N$, 
 \begin{equation}
 \label{eq:def-Q-periodic}
Q_{j}=\|(DT_\zeta)^{j}(Q_0)\|^d\frac{(DT_\zeta)^{j}(Q_0)}{\|(DT_\zeta)^{j}(Q_0)\|}, 
 \end{equation}
where $DT_\zeta$ denotes the derivative of $T$ at $\zeta$, $(DT_\zeta)^j$ its $j$-fold product and the norm is just the usual Euclidean norm in $\R^d$. For all negative $j$ we have $Q_{j}=\infty$ a.s. (note that here `$\infty$' can be thought of as any point in the completion of $\R^d$ which is not contained in $\R^d$: this has infinite norm). Note that, for $d=1$, if $\chi:=|DT_\zeta|$ then  $(Q_{j})_{j\in\Z}$ is such that $Q_j=\chi^j$, for all $j\geq0$ and $Q_{j}=\infty$ for all $j<0$.
\end{remark}

\begin{remark}
\label{rem:bidimensional-leb}
We observe that in line with \cite{BPS18} we have placed the $\theta$ in the second component of the intensity measure ($\leb\times\theta\leb$) of the point process $M$.  However, as can be seen, for example, from Corollary~\ref{cor:average-convergence}, we could have put it in the first coordinate which would be in line with \cite[Equation~(2.9)]{FFM20} or \cite[Remark~7.3.2]{KS20}, for example. One could also express it as $\theta$ times bidimensional Lebesgue measure as in \cite[Corollary~3.7]{H87}.
\end{remark}

\subsection{Functional limit theorems for heavy tailed dynamical sums}
\label{subsec:heavy-tail-limits}

Throughout this section we assume that the process $\X_0,\X_1,\ldots$ is obtained from a system as described in \eqref{eq:dynamics-SP} and \eqref{eq:dynamics-SP-norm}, where $g$ is of type $g_2$, which together with some regular behaviour of the invariant measure $\mu$ in the vicinity of the maximal set $\mathcal M$, guarantees that there exists a sequence of positive real numbers $(a_n)_{n\in\N}$, such that 
\begin{equation}
\label{eq:heavy-normalisation}
\lim_{n\to\infty}n\p(\|\X_0\|>ya_n)=y^{-\alpha}.
\end{equation}
We refer to \cite[Chapter~4]{LFFF16} on how \eqref{eq:heavy-normalisation} can be verified and to \cite{FFT12, GHN11, HNT12, FFT13, CFFH15,AFFR16, FFRS20,CHN21} for several examples of particular dynamical systems and maximal sets satisfying the regularity conditions. 
Hence, taking $\tau=x^{-\alpha}$ and $u_n(\tau)=a_n\tau^{-\frac1\alpha}$, 
equation \eqref{un} holds. 

We are also going to assume that the transformed anchored tail process given in Definition~\ref{def:piling-process} exists and is well defined, which implies that existence of the sequences $(Q_j)_j$ as above.   Section~\ref{subsec:systems} provides examples of systems satisfying all our requirements.

The L\'evy-It\^o representation gives a nice way to describe the L\'evy process as a functional of Poisson point process, whose intensity measure gives the L\'evy measure that determines the process (see \cite{S13}). 
In the case of an $\alpha$-stable L\'evy process, it is usually identified through a limit of a Poisson integral of a Poisson point process $M_\alpha=\sum_{i=1}^\infty\delta_{(T_i,P_i)}$ with intensity measure $\leb\times\nu_{\alpha}$, where the L\'evy measure $\nu_\alpha$ is such that $\nu_\alpha(\{x:\|x\|>y \})= y^{-\alpha}$. Namely, when there is no clustering, for example, the limiting L\'evy process can be written as 
$$
V(t)=\lim_{\eps\to 0}\left(\sum_{T_i\leq t} P_i\I_{\{\|P_i\|>\eps\}} -\int_{\eps<\|x\|\leq1} x \d\nu_{\alpha}(x)\right).
$$
Hence, in this case, as explained in more detail in Section~\ref{subsec:regular-variation}, we consider a transformed version of the general rare events point processes mentioned earlier so that the limit has a Poisson component which can be written as $M_\alpha=\sum_{i=1}^\infty\delta_{(T_i,P_i)}=\sum_{i=1}^\infty\delta_{(T_i,U_i^{-1/\alpha})}$, with intensity measure $\leb\times \theta\,\nu_\alpha$, with $\nu_\alpha(y)=\d(-y^{-\alpha})$, while the decorations, which we denote by $(\mathcal Q_{j})_{j\in\Z}$, in this case, are related to the $(Q_j)_{j\in\Z}$ above through \eqref{eq:Q-relation}, in Section~\ref{subsec:regular-variation}.

Consider the partial sum process in $D([0,1],\R^d)$ defined by:
\begin{equation}
\label{eq:sum}
S_n(t)=\sum_{i=0}^{\lfloor nt\rfloor-1} \frac{1}{a_n}\X_i-tc_n,\qquad t\in[0,1],
\end{equation}
where the sequence $(c_n)_{n\in\N}$ is such that $c_n=0$ if $0<\alpha<1$ and 
$$
c_n=\frac{n}{a_n} \E\left(\X_0\I_{\|\X_0\|\leq a_n}\right), \qquad \text{for } 1\leq\alpha<2.
$$
Our main goal in this section is to establish an invariance principle for $S_n$ which keeps record of all the fluctuations during a cluster of high values which are responsible for a jump of $S_n$. As usual, for $1\leq\alpha<2$, we need that the small contributions for the sum are close to the respective expectation, namely, for all $\delta>0$
\begin{equation}
\label{eq:small-jumps}
\lim_{\varepsilon\to0}\limsup_{n\to\infty}\p\left(\max_{1\leq k\leq n}\left\|\sum_{j=1}^k\left(\X_j\I_{\|\X_j\|\leq \varepsilon a_n}\right)-\E\left(\X_j\I_{\|\X_j\|\leq \varepsilon a_n}\right)  \right\|\geq \delta a_n\right)=0.
\end{equation}
In order to describe the limit, we assume the existence of the transformed anchored tail process, as in Definition~\ref{def:piling-process}. For $1<\alpha<2$, we will also need to assume that the sequence $(Q_j)_{j\in\Z}$, obtained from the spectral decomposition of the transformed anchored tail process, satisfies the assumption
\begin{equation}
\label{eq:Q-condition->1}
\E\left(\left(\sum_{j\in\Z}\|\mathcal{ Q}_j)\|\right)^\alpha\right)<\infty,
\end{equation}
which in the case $\alpha=1$ should be replaced by
\begin{equation}
\label{eq:Q-condition-=1}
\E\left(\sum_{j\in\Z}\|\mathcal{Q}_j\|\log\left(\|\mathcal{Q}_j\|^{-1}\sum_{i\in\Z}\|\mathcal{Q}_j\|\right)\right)<\infty.
\end{equation}

In order to describe the excursions during the clusters, we will need an orientation preserving bijection from $[0, 1]$ to $[-\infty, \infty]$, continuous on $(0, 1)$. For definiteness, we take:
\begin{equation}
\label{eq:upsilon-def}
\h(t):= \tan\left(\pi\left(t-\frac12\right)\right)
\end{equation}

We can now state our main theorem regarding the behaviour of sums of heavy tailed observables.

\begin{theorem}
\label{thm:sums-heavy-tails}
Let $T:\mathcal X \to \mathcal X$ be a dynamical system as described in Section~\ref{subsec:systems}. Let $\X_0,\X_1,\ldots$ be obtained from such a system as described in 
\eqref{eq:dynamics-SP-norm}, where $g$ is of type $g_2$ and condition \eqref{eq:heavy-normalisation} holds. Assume also that the transformed anchored tail process given in Definition~\ref{def:piling-process} is well defined. Consider the continuous time process $S_n$ given by \eqref{eq:sum}. For $1\leq\alpha<2$ assume further that condition \eqref{eq:small-jumps} holds and, for $1<\alpha<2$, assume also \eqref{eq:Q-condition->1}, while for $\alpha=1$ assume \eqref{eq:Q-condition-=1}, instead. Then $S_n$ converges in $F'$ to $\underline V:=(V, disc(V),\{e_{V}^s\}_{s\in disc(V)})$, where $V$ is an $\alpha$-stable L\'evy process on $[0,1]$ which can be written as
\begin{align}
\label{eq:alpha-stable-limit}
V(t)&=\sum_{T_i\leq t}\sum_{j\in\Z} U_{i}^{-\frac1\alpha}\mathcal{Q}_{i,j},
\end{align}
 for $0<\alpha<1$ and 
\begin{align*}
V(t)&=\lim_{\varepsilon\to 0}\Bigg(\sum_{T_i\leq t}\sum_{j\in\Z} U_{i}^{-\frac1\alpha}\mathcal{Q}_{i,j}\I_{\{\|U_{i}^{-\frac1\alpha}\mathcal{Q}_{i,j}\|>\varepsilon\}}\nonumber\\
&\hspace{5cm}
-t\theta\int_{0}^{+\infty}\E\Bigg(y\sum_{j\in\Z}\mathcal{Q}_{j}\I_{\{\eps< y\|\mathcal{Q}_{j}\|\leq 1\}}\Bigg)d(\shortminus y^{\shortminus\alpha}) \Bigg)
\end{align*}
for $1\le \alpha<2$; and the excursions  can be represented by
$$
e_{V}^{T_i}(t)=V(T_i^-)+U_{i}^{-\frac1\alpha}\sum_{j\leq\left\lfloor \h(t) \right\rfloor}\mathcal{Q}_{i,j}, \qquad t\in[0,1],
$$
where $(T_i)_{i\in\N}$, $(U_i)_{i\in\N}$ are as described above (see also \eqref{eq:limit-process}),  $(\mathcal{Q}_{i,j})_{i\in\N, j\in\Z}$ is such that $\mathcal{Q}_{i,j}=\xi(Q_{i,j})$, where $(Q_{i,j})_{i\in\N, j\in\Z}$ are as in \eqref{eq:Z-polar}  and $\xi$ as in \eqref{eq:xi}.
\end{theorem}

\begin{remark}\label{rem:general-M}
Recall that when $\Psi$ is as in \eqref{eq:dynamics-particular-observable} and $\mathcal M$ is reduced to an hyperbolic periodic point,  the transformed anchored tail process is well defined and the $(Q_{i,j})_{i\in\N, j\in\Z}$ are as in Remark~\ref{rem:Qj-periodic-point}. 
Also note that when $\mathcal M$ is reduced to a generic point, we have no clustering and then the result holds with the trivial $Q_j=\infty$, for all $j\in\Z\setminus\{0\}$ and $Q_0$ as before. 
\end{remark}

\begin{example}
\label{example:oscillatory}
We illustrate the theorem with a concrete application which was mentioned in the introduction.  We consider the system $T:[0,1]\to[0,1]$, given by $T(x)=3x \mod 1$, which is a uniformly expanding systems as those described in Section~\ref{sssec:non-invertible}. The probability measure $\mu=\text{Leb}|_{[0,1]}$ is invariant and $T$ has exponential decay of correlations against $L^1(\mu)$ (see Definition~\ref{def:DC}, below). We consider the  set $\mathcal M=\{1/8,3/8\}$ which maximises the observable $\psi(x)=|x-1/8|^{-2}-|x-3/8|^{-2}$ and the one-dimensional stochastic process $X_0,X_1,\ldots$ given by $X_j=\psi\circ T^j(x)$. Note that $T(1/8)=3/8$ and $T(3/8)=1/8$.

We observe that $\psi$ is regularly varying with index $\alpha=1/2$. In fact, we take $a_n=16n^2$, then 
$$
\lim_{n\to\infty}n\mu(\{x\in[0,1]\colon\;|\psi(x)|>a_n u\})\sim u^{-1/2}.
$$
We are now ready to describe the limit of $S_n(t)=\sum_{i=0}^{\lfloor nt\rfloor -1}X_i/a_n$, with $t\in[0,1]$ in $F'$. For that purpose, let $\theta=2/3$ and $M=\sum_{i=1}^\infty\delta_{(T_i,U_i)}$ be a Poisson point process defined on $\R_0^+\times\R_0^+$, with intensity measure $\text{Leb}\times\theta\,\text{Leb}$. Let $E_1, E_2,\ldots$ be a sequence of \iid discrete random variables independent of $(T_i)_{i}$ and $(U_i)_{i}$ and such that $\p(E_1=1)=1/2=\p(E_1=-1)$. For each $i\in\N$, we set $Q_{i,j}=(-3)^jE_i$, for all $j\in\N_0$ and $Q_{i,j}=\infty$ for all $j\in\Z\setminus\N_0$. We recall that $\mathcal Q_{i,j}=\xi(Q_{i,j})=(-9)^{-j}E_i$, for all $j\in\N_0$ and $\mathcal Q_{i,j}=\xi(Q_{i,j})=0$, for all $j\in\Z\setminus\N_0$.

Hence, $S_n(t)$ converges in $F'$ to $(V,disc(V),{e^s_V}_{s\in disc(V)})$, where $V$ is given in \eqref{eq:alpha-stable-limit} and the excursions can be easily written as:
\begin{equation}
e^{T_i}_V(t)=V(T_i^-)+U_i^{-2}E_i \sum_{0\leq j\leq\left\lfloor \h(t)\right\rfloor} (-9)^{-j}, \qquad t\in[0,1].
\label{eq:simple_exc}
\end{equation}

Numerical simulations of the finite sample behaviour of $S_n$, for large $n$, are given in Figures~\ref{fig:3xmod1-osc-p2} and \ref{fig:3xmod1-osc-p2-blowup}.  The latter blow-up shows how the stacking of jumps can occur in $S_n$, and the requirement for convergence in $F'$ in the limit.  As noted in the introduction, here we have overshooting in the discontinuities (namely, $\inf_{t\in[0,1]}e^{s}_V(t)<\min\{V(s^-),V(s)\}<\sup_{t\in[0,1]}e^{s}_V(t)$), which means that convergence of $S_n(t)$ in $D$ would be precluded in any of the Skrorohod's topologies.

We leave the proof of the form of the transformed anchored tail process of this example to Appendix~\ref{Appendix:dyn}.

\iffalse
\begin{remark}
\label{rem:example}
Note that in this example, we have overshooting in the discontinuities (namely, $\inf_{t\in[0,1]}e^{s}_V(t)<\min\{V(s^-),V(s)\}<\sup_{t\in[0,1]}e^{s}_V(t)$), which means that convergence of $S_n(t)$ in $D$ would be precluded in any of the Skrorohod's topologies.

We also observe that the information we retrieve from the limiting process in this simple application could not be obtained directly from the results in \cite{FFM20, PS20}. 
\end{remark}
\fi

\end{example}

\subsection{Enriched extremal process dynamics in the presence of clustering}

Extremal processes are a very useful tool to study the stochastic behaviour of maxima and records (see \cite{R87}). We define the partial maxima associated to the sequence $\X_0, \X_1, \ldots$ by 
\begin{equation}
\label{eq:Mn}
M_n:=\max\{\|\X_0\|,\ldots,\|\X_{n-1}\|\}=\bigvee_{i=0}^{n-1}\|\X_i\|.
\end{equation}
Finding a distributional limit for $M_n$ is one of the first goals in Extreme Value Theory (see for example \cite{LLR83,EKM97,BGTS04,HF06,FHR11}).
\begin{definition}
We say that we have an \emph{Extreme Value Law} (EVL) for $M_n$ if there is a non-degenerate d.f. $H:\R_0^+\to[0,1]$ with $H(0)=0$ and, for every $\tau>0$, there exists a sequence of thresholds $u_n(\tau)$, $n=1,2,\ldots$, satisfying equation \eqref{un} and for which the following holds:
\begin{equation}
\label{eq:EVL-law}
\p(M_n\leq u_n(\tau))\to \bar H(\tau),\;\mbox{ as $n\to\infty$.}
\end{equation}
where $\bar H(\tau):=1-H(\tau)$ and the convergence is meant at the continuity points of $H(\tau)$.
\end{definition}
It turns out that the limit $H$ allows us to describe the functional limit for associated extremal processes.

In this context, we now consider the continuous time 
process $\{Z_n(t):0\leq t\leq 1\}$ defined by
\begin{equation}
\label{eq.Z-process}
Z_n(t):=u_n^{-1}(M_{\lfloor nt\rfloor}), \qquad t\in(0,\infty)
\end{equation}
Recall that $u_n^{-1}(z)$ gives the asymptotic frequency of exceedances of a threshold placed at $z$ and therefore $Z_n$ is non-increasing.

For each $n\geq 1$, $Z_n(t)$ is a random graph with values in $D((0,\infty),\R))$, which can be embedded into $F'$, as in \eqref{eq:embedding-F}. The process $Z_n$ will be shown to converge, in $F'$, to the process $\underline{Z}$, whose first component in $D$ is $Z_H$, which can be described by the finite-dimensional distributions:
\begin{equation}
\label{eq:extremal-process}
\p(Z_H(t_1)\geq y_1,\ldots, Z_H(t_k)\geq y_k)\!=\!\bar H^{t_1}\!\!\left(\bigvee_{i=1}^k\{y_i\}\!\right)\!\!\bar H^{t_2-t_1}\!\!\left(\bigvee_{i=2}^k\{y_i\}\!\right)\!\!\cdots\bar H^{t_k-t_{k-1}}(y_k),
\end{equation}

with $0\leq t_1<t_2<\cdots<t_k\leq 1$.  By the Kolmogorov extension theorem such a process is well defined and we call it an \emph{extremal process}, although, strictly speaking, this is a transformed version of the original extremal processes studied by Resnick \cite{R87}. The relation between the two is obtained through the connection between the levels $u_n(\tau)$ and the more classical linear normalising sequences $(a_n)_{n\in\N}\subset\R^+$ and $(b_n)_{n\in\N}\subset\R$ such that we can write $u_n=y/a_n+b_n$ and
\[n\p(X_0>u_n)=n\p(a_n(X_0-b_n)>y)\to\tau\]
with $\tau=f(y)$ for some homeomorphism $f$, then, as in \eqref{eq:Mn},
\[\p(M_n\leq u_n)=\p(a_n(M_n-b_n)\leq y)\to G(y)\]
where $G=\bar H\circ f$. Then, if $Y_n(t)=a_n(M_{\lfloor nt\rfloor+1}-b_n)$ and $Y_G$ denotes the respective extremal process obtained in \cite{R87}, we have that $Z_n(t)=f(Y_n(t))$ and $Z_H(t)=f(Y_G(t))$. 
\begin{remark}
\label{rem:homeomorphism}
Depending on the type of limit law that applies, $f(y)$ is of one of the following three types: $f_1(y)=\e^{-y}$ for $y\in{\mathbb R}$, $f_2(y)=y^{-\alpha}$ for $y>0$, and $f_3(y)=(-y)^{\alpha}$ for $y\leq 0$.
\end{remark}

\begin{theorem}
\label{thm:extremal-process}
Let $T:\mathcal X \to \mathcal X$ be a dynamical system  as described in Section~\ref{subsec:systems}. Let $\X_0,\X_1,\ldots$ be obtained from such a system as described in \eqref{eq:dynamics-SP} and assume that the transformed anchored tail process given in Definition~\ref{def:piling-process} is well defined. Consider the continuous time process $Z_n$ defined by \eqref{eq.Z-process}. Then $Z_n$ converges in $F'$ to $(Z_H, disc(Z_H),\{e_{Z_H}^s\}_{s\in disc(Z_H)})$, where $Z_H$ is defined as in \eqref{eq:extremal-process}, with $\bar H(\tau)=\e^{-\theta\tau}$, and the excursions  can be represented by
$$
e_{Z_H}^s(t)=\min\left\{ Z_H(s^-),\inf_{j\leq\left\lfloor \h(t)\right\rfloor}Z_H(s)\cdot Q^s_j\right\}, \quad t\in[0,1]
$$
where each sequence $(Q^s_j)_{j\in\N}$ is independent of $Z_H(s^-)$ and with common distribution given by \eqref{eq:Z-polar}.
Moreover, $Z_H$ can be seen as a Markov jump process with 
$$
\p(Z_H(t+s)\geq y \mid Z_H(s)=z)=\begin{cases}
\e^{-\theta t y}&\text{if $y< z$}\\
0&\text{if $y\geq z$}
\end{cases}, \qquad\text{for $t,s>0$.}
$$
The parameter of the exponential holding time in state $z$ is $\theta z$ and given that a jump is due to occur the process jumps from $z$ to $[0,y)$ with probability
$$
\Pi(z,[0,y))=\begin{cases}
\frac y z&\text{if $y< z$}\\
1&\text{if $y\geq z$}
\end{cases}.
$$

\end{theorem}

\begin{remark}
Alternatively and similarly to extremal processes that can be described as a projection into $D$ of a point process (see \cite[Equation~(4.20)]{R87}), we can describe the limit of $Z_n$ as the projection into $F'$ of the point process $N=\sum_{i=1}^\infty \delta_{(T_{i}, U_{i}\tilde{\mathbf Q}_i)}$, given in \eqref{eq:limit-process}. Namely, $Z_H(t)=\inf\{U_i\colon \; T_i\leq t\}$, $disc(Z_H)=\{T_i\colon  i\in\N\}$ and 
$$e_{Z_H}^{T_i}(t)=\min\left\{ \inf\{U_j\colon \; T_j<T_i\},\inf_{j\leq\left\lfloor\h(t)\right\rfloor}U_i\cdot Q_{i,j}\right\}, \quad t\in[0,1].$$
\end{remark}

\subsection{Record point processes}  

The study of record times of observational data has important applications in the study of natural phenomena. Consider the original sequences $(\X_n)_{n\in\N_0}$ and $(M_n)_{n\in\N}$ given in \eqref{eq:Mn}, and let $t_1=0$. Define the strictly increasing sequence $(t_k)_{k\in\N}$:
\begin{equation}\label{eq:record-time}
t_k:=\inf\{j>t_{k-1}:\|\X_j\|>M_j\}.
\end{equation}
This sequence $(t_k)_k$ corresponds to the \emph{record times} associated to $M_n$, namely the times where $M_n$ jumps. 
In the presence of clustering, record times may collapse in the limit and it is important to keep track of possible increments on the number of records occurring during the clusters. Two possible approaches are to consider the enriched limits in $F'$ of the extremal processes $Z_H$ or simply to project directly from the point processes $N$, as done in \cite{BPS18}. Since following the former approach, to be able to handle the possibility of having vanishing jump points in the limit we would need to consider more restrictive subspaces to obtain continuity and then apply the Continuous Mapping Theorem (CMT), 
we will use the latter approach so that we can benefit from the work already carried in \cite{BPS18} and reduce the length of the exposition.  

We follow \cite[Section 5]{BPS18} closely, although we make some adjustments, in particular, due to the fact that  we are using a transformed version of processes associated to the asymptotic frequencies given by the normalisation by $u_n^{-1}$. 
The main advantage here is that we obtain the convergence of the record times point processes for stationary vector-valued sequences with much more general distributions rather than regularly varying sequences (or sequences that could be monotonically transformed into regularly varying ones) as in \cite{BPS18}.  The notation and notions of convergence for point processes used here are detailed in Appendix~\ref{appendix:convergence-point-processes}.

In order to count the number of records of the process $\X_0,\X_1,\ldots$, we introduce the record point process
\begin{equation}
\label{eq:Record-n-def}
\mathfrak R_n=\sum_{i=0}^\infty \delta_{\frac i n}\I_{\{\|\X_i\|>M_i\}}.
\end{equation}
\begin{theorem}
\label{thm:record-theorem}
Let $T:\mathcal X \to \mathcal X$ be a dynamical system  as described in Section~\ref{subsec:systems}. Let $\X_0,\X_1,\ldots$ be obtained from such a system as described in \eqref{eq:dynamics-SP} and assume that the transformed anchored tail process given in Definition~\ref{def:piling-process} is well defined and that 
  $$
  \p ( \mbox{all finite } Q_{j}\mbox{'s are mutually different}) = 1 \; .
  $$  
  Then
  $
  \mathfrak R_n$ converges weakly to $\mathfrak R$, in $\mathcal N_{(0,+\infty)}^\#$.
  The limiting process is a compound Poisson process which can be represented as 
  \begin{equation}
  \label{eq:Record-def}
  \mathfrak R = \sum_{i\in\Z} \delta_{\tau_i} \kappa_i \,,
  \end{equation}
  where $\sum_{i\in\Z} \delta_{\tau_i}$ is a Poisson point process on $(0,\infty)$ with intensity
  measure $x^{-1}\; dx$.  Here $(\kappa_i)_i$ is a sequence of \iid random variables independent of $(\tau_i)_i$ with distribution corresponding to the number of record lower values observed in the sequence $\mathbf Q=(Q_j)_j$, which beat (dropped below) the threshold $U^{-1}$, where $U$ is a uniformly distributed random variable independent of  $\mathbf Q$.
\end{theorem}
\begin{remark}
In the periodic point case, the condition regarding the fact that all finite $Q_{j}$'s are a.s. mutually different is trivially satisfied.  
\end{remark}

\section{General complete convergence of multidimensional cluster point processes}
\label{sec:point-processes}

In this section we present our technical tools and various results in the context of general stochastic processes: the shift $\sigma$ marries this with a dynamical point of view, but the work here holds in wide generality.

Let $\V=\R^d$, for some $d\in\N$, where we consider a norm which we denote by $\|\cdot\|$. For definiteness, we may consider the usual Euclidean norm. 
We will be considering the spaces of one-sided and two-sided $\V$-valued sequences, which we will denote, respectively, by $\V^{\N_0}$ and $\V^{\Z}$, where we consider the one-sided and two-sided shift operators defined by $\TT:\V^{\N_0,\Z}\to\V^{\N_0,\Z}$, 
where 
\begin{equation}
\label{eq:shift}
\TT((x)_i)=((x)_{i+1}).
\end{equation}

Consider a stationary sequence of random vectors $\X_0, \X_1, \ldots$, taking values on $\V=\R^d$, which we will identify with the respective coordinate-variable process on $(\V^{\N_0}, \mathcal B^{\N_0}, \p)$, given by Kolmogorov's existence theorem, where $\mathcal B^{\N_0}$ is the $\sigma$-field generated by the coordinate functions $Z_n:\V^{\N_0}\to\mathcal \V$, with $Z_n(x_0,x_1,\ldots)=x_n$, for $n\in\N_0$. 
Note that, under these identifications, we can write:
\[
Z_{i-1}\circ \TT =Z_{i}, \quad \mbox{for all $i\in\N$}.
\]
Since, we assume that the process is stationary, then $\p$ is $\TT$-invariant.  Note that $Z_i=Z_0\circ \TT^i$, for all $i\in\N_0$, where $\TT^i$ denotes the $i$-fold composition of $\TT$, with the convention that $\TT^0$ denotes the identity map on $\V^{\N_0,\Z}$.

In what follows, for every $A\in\mathcal B^{\N_0,\Z}$, we denote the complement of $A$ as $A^c:=\V^{\N_0,\Z}\setminus A$.

\subsection{Identifying clusters}
\label{subsec:blocking}

Our goal is to study the impact of clustering on the convergence of general multidimensional point processes. 
As mentioned earlier, information regarding the observations within the same cluster gets collapsed at the same time point, which makes it hard to recover it from the limiting process. In order to keep track of that information, we need to start by identifying clusters.

There are two main approaches to identify clusters, which  are commonly referred to as declustering procedures. One is the blocking method and the other is the runs declustering procedure. 

We start with the blocking procedure, which, in fact, serves two purposes. Namely, not only will it separate clusters, it will also introduce time gaps between the blocks in order to restore some independence between them. The size of the blocks must be sensitively tuned so that the blocks are neither too long so that they do not separate clusters, nor too small so that they do not split clusters apart. The same applies to time gaps between the blocks, which are created by disregarding observations. 
Following the classical scheme \cite{LLR83}, considering a finite sample of size $n$, we split the data into $k_n\in\N$ blocks of size $r_n:=\lfloor n/k_n\rfloor$ and take time gaps of size $t_n\in\N$. This way, we define sequences $(k_n)_{n\in\N}$, $(r_n)_{n\in\N}$, $(t_n)_{n\in\N}$, which we assume to be such that 
\begin{equation}
\label{eq:kn-sequence}
k_n,r_n,t_n\xrightarrow[n\to\infty]{}\infty\quad \mbox{and}\quad  k_n t_n = o(n).
\end{equation}

Regarding the runs declustering, for a finite sample of size $n$, we set the run length $q_n\in\N$ with the aim that all abnormal observations occurring within a time difference of at most $q_n$ units between each other belong to the same cluster. The sequence $(q_n)_{n\in\N}$ must be chosen so that 
\begin{equation}
\label{eq:qn-sequence}
q_n= o(r_n),
\end{equation}
and also so that it satisfies conditions $\D_{q_n}$, $\D'_{q_n}$, below. Note that $q_n=q$ for all $n\in\N$ and some $q\in\N$ is a possibility here. Namely, when applying to periodic points $q_n$ can be taken as $q$, the period of the point (see \cite{FFT12} and \cite[Section~2.1]{AFF20}).

To summarise, we will define objects motivated by a runs declustering scheme (see \eqref{eq:Aq} below), the dependence conditions that we introduce stem from a blocking procedure, which will eventually determine the identification of clusters when we introduce the point processes of clusters in Section~\ref{subsec:complete-convergence}. The connection between the two approaches is essentially provided by condition $\D'_{q_n}$. See also Remarks~\ref{rem:D-prime-one} and \ref{rem:D-prime-two}.

\subsection{Dependence structure}
\label{subsec:dependence}

In order to prove the main convergence results we need to introduce some conditions on the dependence structure of the stationary processes and therefore introduce the following objects. We follow more or less the notation used in \cite{FFM18, FFM20}.

Let $A\in\mathcal B^{\N_0,\Z}$ be an event, let $J$ be an interval contained in $[0,\infty)$. 
We define 
\begin{equation}
\label{eq:W-def}
\mathscr W_{J}(A):=\bigcap_{i\in J\cap \N_0}\TT^{-i}(A^c),\quad \mathscr W_{J}^c(A):=(\mathscr W_{J}(A))^c=\bigcup_{i\in J\cap \N_0}\TT^{-i}(A).
\end{equation}
Let $i,j\in\N$ and set $\mathscr W_{j}(A):=\mathscr W_{[0,j)}(A)$,  $\mathscr W_{i,j}(A):=\mathscr W_{[i,j)}(A)$, when $i\leq j-1$, $\mathscr W_{i,j}(A):=\V^{\N_0,\Z}$, when $i\geq j$.
Moreover,  $\mathscr W^c_{j}(A)=\left(\mathscr W_{j}(A)\right)^c$ and $\mathscr W^c_{i,j}(A)=\left(\mathscr W_{i,j}(A)\right)^c$ .

For the event $A\in\mathcal B^{\N_0,\Z}$ and $j\in \N$,  
\begin{equation} 
\label{eq:Aq}
A^{(j)}:=A\cap \TT^{-1}(A^c)\cap\cdots\cap\TT^{-j}(A^c),
\end{equation} and for $j=0$ we simply define $A^{(0)}:=A$.

In what follows, for some $a>0$, $y\in\R^d$ and a set $A\subset\R^d$, we set $aA=\{ax:\, x\in A\}$ and $A+y=\{x+y:\, x\in A\}$.
We define the class of sets   
\begin{equation}
\label{eq:A-field}
\mathscr F=\left\{\{(x_j)_j\in \V^{\N_0,\Z}\colon x_{j}\in H_{j},\, j=0,\ldots,m\}\colon  H_{j}\in\mathcal F_\V,\, j=0,\ldots,m,\, m\in\N \right\},
\end{equation}
where $\mathcal F_\V$ denotes the field generated by the rectangles of $\V$ of the form $[e_1,f_1)\times\cdots\times[e_d,f_d)$. Note that $\mathscr F$ is a field.

For each $\ell=1, \ldots,m\in\N$ suppose
\begin{align}
\label{eq:rectangles}
&   A_\ell\in \mathscr F \text{ and } J_\ell=[a_\ell,b_\ell),
\end{align}
where $0\leq a_1<b_1\leq a_2<b_2\leq\cdots\leq a_m<b_m\leq1$. 
Then for each  $n\in\N$ define 
\begin{align}
&\tilde J_{k_n,\ell}:=k_nJ_\ell,\quad J_{n,\ell}:=\big[(\lceil k_na_\ell\rceil-1)r_n,(\lfloor k_n b_\ell\rfloor+1)r_n\big), \label{eq:Jnl-def} \\
&A_{n,\ell}=\left\{\left(u_n^{-1}(\|\X_j\|)\frac{\X_j}{\|\X_j\|}\right)_j\in \V^{\N_0}\colon \; \left(u_n^{-1}(\|\X_j\|)\frac{\X_j}{\|\X_j\|}\right)_j\in A_\ell \right\},
\label{eq:rectangles-n-bi}
\end{align}
where $u_n$ is defined as in \eqref{un}, $k_n$ is as in \eqref{eq:kn-sequence}. 
We discuss this normalisation further in \eqref{eq:block-def} and \eqref{eq:normalisation}.

We introduce a mixing condition which is specially designed for the application to the dynamical setting.

\begin{condition}[$\D_{q_n}$]\label{cond:D} We say that $\D_{q_n}$ holds for the sequence $\X_0,\X_1,\ldots$ if there exist sequences $(k_n)_{n\in\N}$, $(r_n)_{n\in\N}$, $(t_n)_{n\in\N}$ and $(q_n)_{n\in\N}$ satisfying \eqref{eq:kn-sequence} and \eqref{eq:qn-sequence}, such that for every  $m, t,n\in\N$ and every $J_\ell$ and $A_\ell$, with $\ell=1, \ldots, m$, chosen as in \eqref{eq:rectangles}, we have 
\begin{equation}\label{eq:D1}
\left|\p\left(\A_{n,\ell}\cap \bigcap_{i=\ell}^m\W_{J_{n,i}}\left(\A_{n,i}\right) \right)-\p\left(\A_{n,\ell}\right)
  \p\left( \bigcap_{i=\ell}^m\W_{J_{n,i}}\left(\A_{n,i}\right)\right)\right|\leq \gamma(n,t),
\end{equation}
where 
$\min\{J_{n,\ell}\cap\N_0\}\geq t$ and $\gamma(n,t)$ is decreasing in $t$ for each $n$ and 
$\lim_{n\to\infty}n\gamma(n,t_n)=0,$
 where $\A_{n,\ell}$ is given by \eqref{eq:rectangles-n-bi} and \eqref{eq:Aq}.\end{condition}

This mixing condition is much milder than similar conditions used in the literature and is particularly suited for applications to dynamical systems, since it is easily verified for systems with sufficiently fast decay of correlations, see for example the discussion in Section~\ref{subsec:systems}.

\begin{condition}[$\D'_{q_n}$]\label{cond:D'} We say that $\D'_{q_n}$
holds for the sequence $\X_0,\X_1,\X_2,\ldots$ if there exist sequences $(k_n)_{n\in\N}$, $(r_n)_{n\in\N}$ and $(q_n)_{n\in\N}$ satisfying \eqref{eq:kn-sequence} and \eqref{eq:qn-sequence}, such that  for every  $A_1\in \F$, 
we have
\begin{equation*}
\lim_{n\rightarrow\infty}\,n\p\left( \A_{n,1}\cap \mathscr W^c_{[q_n+1,r_n)}\left(A_{n,1}\right)
\right)=0,
\end{equation*}
where $\A_{n,1}$ is given by \eqref{eq:rectangles-n-bi} and \eqref{eq:Aq}.
\end{condition}

\begin{remark}
\label{rem:D-prime-one}
Condition $\D'_{q_n}$ forbids the appearance of new abnormal observations (the occurrence of $A_{n,\ell}$), within the same block, once a run of $q_n$ consecutive non-abnormal observations has been realised. This means that $\D'_{q_n}$ establishes a connection between the two declustering procedures and, in particular, requires that no more than one cluster should occur within one block. 
\end{remark}
\begin{remark}
\label{rem:D-prime-two}
Suppose $(q_n)_{n\in\N}$ and $(\tilde q_n)_{n\in\N}$ are two sequences satisfying \eqref{eq:qn-sequence} and $q_n\leq\tilde q_n$, for all $n\in \N$.  Then $\D'_{q_n}$ implies $\D'_{\tilde q_n}$. As in Section~\ref{subsec:systems}, for nice dynamical systems when clustering is created by a periodic point of prime period $q$, then condition $\D'_{q}$ holds, in the sense that $q_n=q$, for all $n$.
\end{remark}

Next we state a stronger version of $\D'_{q_n}$, which, in some cases may be easier to check.
We set 
\begin{equation}
\label{eq:U}
U(\tau)=\left\{(x_j)_j\in \V^{\N_0, \Z}\colon x_0\in B(0,\tau)\right\}
\end{equation}
where, for $x\in \V$ and $\epsilon>0$, we denote by $B(x,\epsilon)$,
the ball centred at $x$ of radius $\epsilon$. Then, following \eqref{eq:rectangles-n-bi} and \eqref{eq:Aq}, we define: 
\begin{align*}
\label{eq:Un}
U_n(\tau)&=\{\|\X_0\|>u_n(\tau)\}\\ 
U_n^{(q_n)}(\tau)&=\{\|\X_0\|>u_n(\tau), \|\X_1\|\leq u_n(\tau), \ldots,  \|\X_{q_n}\|\leq u_n(\tau)\}.
\end{align*}
\begin{condition}[$\tilde\D'_{q_n}$]\label{cond:barD'} We say that $\tilde \D'_{q_n}$
holds for the sequence $\X_0,\X_1,\X_2,\ldots$ if there exist sequences $(k_n)_{n\in\N}$, $(r_n)_{n\in\N}$, $(t_n)_{n\in\N}$ and $(q_n)_{n\in\N}$ satisfying \eqref{eq:kn-sequence} and \eqref{eq:qn-sequence}, such that  for every  $\tau>0$, we have
\begin{equation}
\label{eq:D'rho-un}
\lim_{n\rightarrow\infty}\,n\p\left(U_n(\tau)\cap \mathscr W^c_{[q_n+1,r_n)}\left(U_n(\tau)\right)\right)=0.
\end{equation}
\end{condition}

For $A_\ell$, as in \eqref{eq:rectangles}, let $\tau^*>\sup\{\|x\|\colon \; x\in H_0\}$. Then, by \eqref{eq:un-1-property}, we have 
$$\left\|u_n^{-1}(\|\X_0\|)\frac{\X_0}{\|\X_0\|}\right\|< \tau^*\Leftrightarrow \|\X_0\|>u_n(\tau^*). $$
Therefore, it is clear that $A_{n,\ell}\subset U_n(\tau^*)$, so if $\tilde\D'_{q_n}$ holds then so does $\D'_{q_n}$.
\begin{remark}
Condition $\tilde\D'_{q_n}$ is already weaker than \cite[Assumption~1.1]{BPS18}, which had been used in previous papers (see, for example, \cite[Equation~(2.8)]{DH95}, \cite[Equation~(3)]{S05}, \cite[Condition~4.1]{BS09}, \cite[Condition~2.1]{BKS12}) and was introduced in \cite{S92}.  We also remark that all these conditions allow for the appearance of clustering which already makes them weaker than conditions $D'$ from \cite{D83} or $\text{LD}(\phi_0)$ from \cite{T10a}, which imply $\D'_{q_n}$ with $q_n=1$ for all $n\in\N$.
\end{remark}

\subsection{Bookkeeping of clusters}

In this section, we introduce a device called the \emph{transformed anchored tail process}, which is designed to keep track of the clustering oscillations. It is an adaptation of the \emph{tail process}, introduced in \cite{BS09}, to a tool more applicable in the dynamical setting.  
In \cite{BS09}  and subsequent papers (for example, \cite{BKS12,BPS18}), the tail process was always defined under the assumption that the original process $\X_0, \X_1,\ldots$ is \emph{jointly regularly varying} (see Definition~\ref{def:jointly-rv}), which is not natural to assume \emph{a priori} in the dynamical systems setting.  This assumption (joint regular variation) together with an assumption on the dependence  structure stronger than $\tilde \D'_{q_n}$ (\cite[Condition~4.1]{BS09}) allowed the authors there to prove the existence of the tail process and several very useful properties about it. Motivated by the applications to dynamical systems, here, 
we will not assume, \emph{a priori}, joint regular variation and since $\D'_{q_n}$ is even weaker than $\tilde \D'_{q_n}$, some of the properties of the tail process will be required as adapted assumptions in the definition of the transformed anchored tail process. One of the advantages is that we obtain very general enriched functional limits for extremal processes and record point processes, for example, without assuming regularly varying tails.

\subsubsection{The normalisation of the blocks}
\label{subsubsec:normalisation-blocks}

We recall that under assumption $\D'_{q_n}$ the information regarding to the structure of the clusters is kept in each block of size $r_n$, which we are going to normalise in the following way, by defining for each 
$i< j\in\{0,\ldots,n\}$
\begin{equation}
\label{eq:block-def}
\mathbb X_{n}^{i,j}=\left({u_n^{-1}(\|\X_{i}\|)}\frac{\X_{i}}{\|\X_{i}\|},\ldots,{u_n^{-1}(\|\X_{j-1}\|)}\frac{\X_{j-1}}{\|\X_{j-1}\|}\right), \quad \mathbb X_{n,i}:=\mathbb X_{n}^{(i-1)r_n,ir_n} 
\end{equation}
so that $\mathbb X_{n,i}$ denotes the $i$-th normalised block. 
We also use the notation $\mathbb X_{n}^{j}=\mathbb X_{n}^{j,j+1}$, for all $j=0,\ldots, n$.

Note that the normalisation used is such that the norm of each normalised variable in each block is equal to the asymptotic frequency $\tau=u_n^{-1}(\|\X_{j}\|)$ corresponding to the mean number of exceedances of a threshold placed at  the value $\|\X_{j}\|$, among the first $n$ observations of the process.

In particular, observe that by \eqref{eq:un-1-property}, for all $\tau>0$, we have
\begin{equation}
\label{eq:normalisation}
\left\|\mathbb X_{n}^{j}\right\|
=\left\|u_n^{-1}(\|\X_{j}\|)\frac{\X_{j}}{\|\X_{j}\|}\right\|<\tau\Leftrightarrow u_n^{-1}(\|\X_{j}\|)<\tau\Leftrightarrow\|\X_{j}\|>u_n(\tau).
\end{equation}

\subsubsection{The Extremal Index}
\label{subsubsec:EI}
Before we characterise the transformed anchored tail process, we define the Extremal Index (EI), denoted by $\theta\in[0,1]$, which was formally introduced by Leadbetter in \cite{L83} and measures the degree of clustering of exceedances. When $\theta=1$ we have no clustering and a small $\theta$ means intense clustering. A common interpretation for the EI is that it is reciprocal of the average cluster size (see \cite{AFF20}). We define the EI following O'Brien's formula (\cite{O87}) and assume that, for all $\tau>0$, we have 
\begin{equation}
\label{eq:EI}
\theta=\lim_{n\to\infty}\frac{\p(U_n^{(q_n)}(\tau))}{\p(U_n(\tau))}.
\end{equation}

\subsubsection{Underlying spaces}
\label{subsubsec:spaces}

Let $\dot{\V}=\overline{\R^d}\setminus\{0\}=\left(\R^d\cup\{\infty\}\right)\setminus\{0\}$, $\V=\R^d$ and recall that $\infty\in \dot{\V}$ can be thought of as any point in the completion of $\R^d$ which is not contained in $\R^d$: this has infinite norm. Define
\begin{align*}
l_\infty&=\left\{\mathbf{x}=(x_j)_j\in \dot{\V}^{\Z}\colon\;\lim_{|j|\to\infty}\|x_j\|=\infty\right\}\qquad
l_0
=\left\{\mathbf{x}=(x_j)_j\in {\V}^{\Z}\colon\;\lim_{|j|\to\infty}\|x_j\|=0\right\},
\end{align*}
where for definiteness we are taking the usual Euclidean norm $\|\cdot\|$ in $\R^d$.
The transformed anchored tail process will be defined to take values in $l_\infty$, while the tail process lives in $l_0$. The space $l_\infty$ will borrow the metric structure of $l_0$ by means of the map $\Proj \colon l_\infty \longrightarrow l_0$ given by $\Proj((x_j)_j)=(\proj(x_j))_j$, where
\begin{align*}
\proj \colon \dot{\V} &\longrightarrow \V \\
x &\mapsto\begin{cases}
    \frac{x}{\|x\|^2}\,,& \text{if } x\neq \infty,\\
    0,              & \text{otherwise}.
\end{cases}
\end{align*}

\begin{lemma}
The map $\Proj$ is invertible.
\label{lem:Pinv} 
\end{lemma}
\begin{proof}
By definition of $\Proj$, we only need to show that $\proj$ is invertible. First note that $\proj(x)=0$ if and only if $x=\infty$. Let $x,y\in\dot\V\setminus\{\infty\}$. Then $\proj(x)=\proj(y)$ implies that 
\begin{equation}
\label{eq:aux1}
x=\frac{\|x\|^2}{\|y\|^2}\;y.
\end{equation}
 Let $c=\frac{\|x\|^2}{\|y\|^2}$. Then $x=c y$ which means that $\|x\|=c\|y\|$. Substituting back in \eqref{eq:aux1} we obtain $x=c^2y$. Hence, we must have $c y=c^2y$, which implies that $c=1$ and, therefore, $x=y$. Hence $\Proj$ is one-to-one.
 
 To see that $\Proj$ is onto, we let $\mathbf y\in l_0$ and show that there exists $\mathbf x\in l_\infty$ such that $\Proj(\mathbf x)=\mathbf y$. For all the $j\in\Z$ such that $y_j=0$, we set $x_j=\infty$. For all the other $j\in\Z$, we require $x_j/\|x_j\|^2=y_j$, which implies that $\|y_j\|=1/\|x_j\|$. But then, $x_j=\|x_j\|^2y_j=y_j/\|y_j\|^2$, which means that by setting $x_j=\proj(y_j)$, for all such $j\in\Z$, we have defined the desired $\mathbf x\in l_\infty$.
 \end{proof}

As in \cite{BPS18}, in $l_0$, we consider the supremum norm given by
$$
\|\mathbf x\|_\infty=\sup_{j\in\Z}\|x_j\|,
$$
and the complete metric defined on $l_0\setminus \{\mathbf 0\}$, where $\mathbf 0=(0)_j$,
\begin{equation}
\label{eq:d'-metric}
d'(\mathbf x,\mathbf y)=\left(\|\mathbf x-\mathbf y\|_\infty\wedge 1\right)\vee\left|\frac1{\|\mathbf x\|_\infty}-\frac1{\|\mathbf y\|_\infty}\right|.
\end{equation} 
Now, we consider the metric $d$ defined on $l_\infty\setminus \{(\infty)_j\}$ given by
\begin{equation}
\label{eq:metric}
d(\mathbf x,\mathbf y)=d'(\Proj(\mathbf x),\Proj(\mathbf y)).
\end{equation}
Recall that $l_0\setminus \{\mathbf 0\}$ equipped with the metric $d'$ is a complete separable metric space.  Since $\Proj$ is invertible and componentwise continuous, one can easily show that $l_\infty\setminus\{(\infty)_j\}$ equipped with the metric $d$ is a complete separable metric space. 

Note that we can embed $\cup_{n\in\N}\dot \V^n$ ($\cup_{n\in\N} \V^n$) into $l_\infty$ ($l_0$) simply by adding a sequence of $\infty$ (0) before and after the $n$ entrances of any element of  $\dot \V^n$ ($\V^n$). For example, 
$\mathbb X_{n}^{i,j}$ can be seen as an element of $l_\infty$ by identifying it with 
\begin{equation}
\label{eq:embedding}
\left(\ldots,\infty,\infty,u_n^{-1}(\|\X_{i}\|)\frac{\X_{i}}{\|\X_{i}\|},\ldots,{u_n^{-1}(\|\X_{j-1}\|)}\frac{\X_{j-1}}{\|\X_{j-1}\|},\infty,\infty,\ldots\right).
\end{equation}

We define the quotient spaces $\tilde l_\infty=l_\infty/{\sim}$ and $\tilde l_0=l_0/{\sim}$, where $\sim$ is the equivalence relation defined on both $l_\infty$ and $l_0$ by $\mathbf x\sim\mathbf y$ if and only if there exists $k\in\Z$ such that $\TT^k(\mathbf x)=\mathbf y$, where $\TT$ is the shift operator defined in \eqref{eq:shift}. Also let $\tilde \pi$ denote 
the natural projection from $ l_\infty$ ($l_0$) to $\tilde l_\infty$ ($\tilde l_0$), which assigns to each element $\mathbf x$ of $ l_\infty$ ($l_0$) the corresponding  equivalence class $\tilde\pi(\mathbf x)=\tilde{\mathbf x}$ in $ \tilde l_\infty$ ($\tilde l_0$). Given any vector $v$ of $\dot \V^m$ ($\V^m$), for some $m\in\N$, we write $\tilde\pi( v)$ for the projection of the natural embedding of $v$ into $ l_\infty$ ($l_0$) to the quotient space $\tilde l_\infty$ ($\tilde l_0$). Namely,
\begin{equation}
\label{eq:projection-embedding}
\tilde\pi(\mathbb X_{n}^{i,j})=\tilde\pi\left(\left(\ldots,\infty,u_n^{-1}(\|\X_{i}\|)\frac{\X_{i}}{\|\X_{i}\|},\ldots,{u_n^{-1}(\|\X_{j-1}\|)}\frac{\X_{j-1}}{\|\X_{j-1}\|},\infty,\ldots\right)\right).
\end{equation}

Observe that, since $\Proj$ is invertible, we may define $\tilde \Proj:\tilde l_\infty\to \tilde l_0$ so that $\tilde \Proj (\tilde \pi(\mathbf x))=\tilde \pi(\Proj(\mathbf x))$.

Consider the metric $\tilde d'$ in  $\tilde l_0\setminus\{\tilde{\mathbf 0}\}$  given by
\begin{align}
\label{eq:metric-d-prime-tilde}
\tilde d'(\tilde{\mathbf x},\tilde{\mathbf y})&=\inf\{d'(\mathbf x',\mathbf y')\colon\; \mathbf x'\in\pi^{-1}( \tilde{\mathbf x}),\; \mathbf y'\in \pi^{-1}(\tilde{\mathbf y})\}
=\inf\{d'(\TT^k(\mathbf x),\TT^m(\mathbf y))\colon\; k,m\in\Z \}.
\end{align}
This metric makes $\tilde l_0\setminus\{\tilde{\mathbf 0}\}$ a complete separable metric space. (See Lemma~2.1 and Lemma~6.1 of \cite{BPS18}). Accordingly, on $\tilde l_\infty\setminus\{\tilde{\infty}\}$, where $\tilde{\infty}=(\infty)_j$, we define the metric
$$
\tilde d(\tilde{\mathbf x},\tilde{\mathbf y})=\tilde d'(\tilde \Proj(\tilde{\mathbf x}),\tilde \Proj(\tilde{\mathbf y})),
$$
which also gives a complete separable metric space.

\begin{remark}
\label{rem:bounded-sets}
The choice of the metric implies that a set $A\subset \tilde l_\infty\setminus\{\tilde{\infty}\}$ is bounded if and only if there exists $\eps>0$, such that for all $\tilde \x\in A$ we have $\|\tilde \Proj(\tilde \x)\|_\infty>\eps$ or, equivalently, that $\inf_{j\in\Z}\|x_j\|<1/\eps$.
\end{remark}

For $A\in\mathscr F$, as in \eqref{eq:A-field}, we define 
\begin{equation}
\label{eq:A-tilde}
\tilde A=\{\tilde \x\in \tilde l_\infty\colon\; \tilde\pi^{-1}(\tilde\x)\cap A\neq \emptyset\};\qquad \tilde{\mathscr J}=\{\tilde A\colon\;A\in\mathscr F\}.
\end{equation}
Using that $\mathscr F$ is a field, one can show that the class of subsets $\tilde{\mathscr J}$ is closed for unions. Indeed, this follows easily by observing that
$$
\tilde A\cup \tilde B=
\{\tilde \x\in \tilde l_\infty\colon\; \tilde\pi^{-1}(\tilde\x)\cap (A\cup B)\neq \emptyset\}.
$$
Let $\tilde{\mathscr R}$ denote the class of subsets of $\tilde l_\infty,\tilde l_0$ corresponding to the ring generated by $\tilde{\mathscr J}$, which is actually a field because by definition of  $\mathscr F$, we have that $l_\infty, l_0\subset \mathscr F$.

For $A\in\mathscr F$, let 
\begin{align}
\label{eq:bb-A}
\mathbb A&
=\{\x\in l_\infty\colon \tilde\pi(\x)\in \tilde A\}=\left\{\x\in l_\infty, l_0\colon \x\in \bigcup_{j\in\Z}\TT^{-j}(A) \right\},\\
\mathscr J&=\{\mathbb A\colon\;A\in\mathscr F\} \text{ and } \mathscr R\;\text{be the ring generated by}\; \mathscr J. \label{eq:mathscrR}
\end{align}
If we start with some $A_\ell\in \mathscr F$ here, we correspondingly write $\mathbb A_\ell$.
Note that $\TT^{-1}(\mathbb A)=\mathbb A$, which means that both $\mathscr J$ and $\mathscr R$ are $\TT$-invariant classes of subsets of $l_\infty, l_0$. Also observe that $\mathscr R=\tilde\pi^{-1}(\tilde{\mathscr R})$ and $\mathscr J=\tilde\pi^{-1}(\tilde{\mathscr J})$, which is also closed for unions.

\subsubsection{The transformed anchored tail process}
\label{subsubsec:piling-process}
We can now define the transformed tail process, which presupposes  the existence of a process $(Y_j)_{j\in\Z}\in l_\infty$ satisfying the following assumptions: 
\begin{enumerate}

\item \label{Y-def} 
$\mathcal L\left(\frac1\tau \mathbb X_n^{r_n+s, r_n+t}\;\middle\vert\; \|\X_{r_n}\|>u_n(\tau) \right)\xrightarrow[n\to\infty]{}\mathcal L\left((Y_j)_{j=s,\dots,t}\right),$ for all $s<t\in\Z$ and all $\tau>0$;

\item \label{spectral-process} the process $(\Theta_j)_{j\in\Z}$ given by $\Theta_j=\frac{Y_j}{\|Y_0\|}$ is independent of $\|Y_0\|$; 

\item \label{Y-infinity} $\lim_{|j|\to\infty} \|Y_j\|=\infty$ a.s.;

\item \label{positive-EI} $\p\left(\inf_{j\leq -1}\|Y_j\|\geq 1\right)>0$.

\end{enumerate}

Here $(r_n)_n$ is assumed to satisfy \eqref{eq:kn-sequence}: in our applications it is the sequence appearing in $\D_{q_n}$ and $\D'_{q_n}$.
\begin{remark}
Most of the applications given here are to non-invertible discrete dynamical systems, which means that the sequence $X_0, X_1,\ldots$ is one-sided and therefore we needed to recentre by $r_n$ so that we can obtain a bi-infinite sequence which includes the past. The particular role of $r_n$ is not important as long as its is asymptotically larger than $q_n$. Alternatively, we could have considered the natural extension of the system to obtain a two-sided sequence $\ldots, X_{-1}, X_0, X_1,\ldots$ and then condition on $\|\X_{0}\|>u_n(\tau) $, instead. 
\end{remark}
\begin{remark}
\label{rem:Y-sequence-tail-process}
We remark that in the setting of heavy tailed distributions, the process $(Y_j)_{j\in\Z}$ defined here (which lives in $l_\infty$) is a transformed version of the tail process introduced in \cite{BS09}, which takes values in $l_0$ and was used later in \cite{BKS12,BPS18}, for example. The existence of such a sequence for stationary heavy tailed stochastic processes was proved to be equivalent to joint regular variation, which we define in Section~\ref{subsec:regular-variation}, where further details on the relations with the tail process are also given.
\end{remark}
We finally define the transformed anchored tail process by considering the canonical anchor used in \cite{BS09, BPS18}, which corresponds to conditioning on the fact that $Y_0$ is marking the beginning of a new cluster. For more general anchors we refer to \cite{BP21}.
\begin{definition}
\label{def:piling-process}
Assuming the existence of a sequence $(Y_j)_{j\in\Z}$ satisfying conditions \eqref{Y-def}--\eqref{positive-EI}, we define the \emph{transformed anchored tail process} $(Z_j)_{j\in\Z}$ as a sequence of random vectors satisfying 
$$
\mathcal L\left((Z_j)_{j\in\Z} \right)=\mathcal L\left((Y_j)_{j\in\Z}\;\middle\vert\; \inf_{j\leq -1}\|Y_j\|\geq 1\right).
$$
\end{definition}
We consider a polar decomposition of the transformed anchored tail process by defining the random variable $L_Z$ and the process $(Q_j)_{j\in\Z}$ by
\begin{equation}
\label{eq:Z-polar}
L_Z= \inf_{j
\in\Z}\|Z_j\| \qquad Q_j=\frac{Z_j}{L_Z}.
\end{equation}
We carry this polar decomposition to $\tilde l_{\infty}$ by letting $\mathbb S=\{\tilde \x\in \tilde l_\infty\colon \|\tilde \Proj(\tilde \x)\|_\infty=1\}$ and defining the map 
\begin{align}
\hbar \colon \tilde l_\infty\setminus\{\tilde\infty\} &\longrightarrow \R^+\times \mathbb S \nonumber\\
\tilde \x &\mapsto \left(\frac1{\|\tilde \Proj(\tilde \x)\|_\infty}, \frac{\tilde x}{(\|\tilde \Proj(\tilde \x)\|_\infty)^{-1}}\right).
\label{eq:psi.def}
\end{align}
We define $\tilde{\mathbf Q}:=\tilde\pi((Q_j)_{j\in\Z})$ and observe that $\hbar(\tilde\pi((Z_j)_{j\in\Z}))=(L_Z,\tilde{\mathbf Q}).$

In order to illustrate the advantage of considering the transformed version of tail process rather than the original version given by \cite[equation (1.1)]{BS09} (or \cite[equation (1.6)]{BPS18}, \cite[equation (5.2.3)]{KS20}), we consider a concrete dynamical system with an observable which commonly arrises in the study of extremal dynamics.
\begin{example}
 Let $T\colon[0,1\to[0,1]$ be the doubling map: $T(x)=2x\mod 1$. Let $\zeta_1=1/3$ and $\zeta_2=T(\zeta_1)=2/3$ denote the period two orbit of $T$. Define also the observable function $\psi:[0,1]\to\R$ by:
$$
\psi(x)=-\log|x-\zeta_1|+\log|x-\zeta_2|
$$
Consider now the stochastic process $X_0, X_1, \ldots$ given by $X_n=\psi\circ T_\gamma^{n}$.

For this stochastic process, the tail process, $(\tilde Y_j)_{j\in\Z}$, given by \cite[equation (1.1)]{BS09} is ill defined. In fact, it is easy to observe that $\p(|\tilde Y_0|>y)=0$ for all $y>1$, which contrasts with formula \cite[equation (5.2.4)]{KS20} which establishes that $\p(|\tilde Y_0|>y)=(y\vee 1)^{-\alpha}$. In contrast, the transformed tail process is well defined and we easily obtain that the transformed anchored tail process  is equal to (see Appendix~\ref{subappendix:systems-perioidc-points}): 
$$
\left(\ldots,\infty,\infty,U\cdot E,U\cdot E(-1)2, U\cdot E(-1)^22^2, \ldots,U\cdot E(-1)^k2^k,\ldots\right)
$$
where $U$ is a uniformly distributed random variable on $[0,1]$ and the independent random variable $E$ is such that $\p(E=1)=\frac12=\p(E=-1)$.
\end{example}

Note that due to assumption  \eqref{Y-infinity} both $(Y_j)_{j\in\Z}$ and the transformed anchored tail process $(Z_j)_{j\in\Z}$ take values in $l_\infty$.

\subsubsection{Properties of the transformed anchored tail process}
\label{subsubsec:properties-piling-process}

The following lemma is a nice consequence for $\|Y_0\|$ of our transformed anchored tail process being based in $l_\infty$. 
\begin{lemma}
\label{lem:Y0}
The random variable $\|Y_0\|$ is uniformly distributed.
\end{lemma}
\begin{proof}
Using stationarity \eqref{un} and \eqref{eq:un-1-property}, it follows that for all $v\in[0,1]$,
\begin{align*}
\p(\|Y_0\|&<v)=\lim_{n\to\infty}\p\left(\left\|\frac{u_n^{-1}(\|\X_{r_n}\|)}{\tau}\frac{\X_{r_n}}{\|\X_{r_n}\|}\right\|<v
\;\middle\vert\; \|\X_{r_n}\|>u_n(\tau)\right)\\
&=\lim_{n\to\infty}\p\left(\|\X_{r_n}\|>u_n(\tau v)
\;\middle\vert\; \|\X_{r_n}\|>u_n(\tau)\right)=
\frac{\tau v}{\tau}=v.
\end{align*}
\end{proof}
Next we show a relation that will provide a connection between the transformed anchored tail process and the outer measure 
used in \cite{FFM20}: this describes clustering by splitting the events into annuli of different cluster lengths.

\begin{proposition}
\label{prop:block}
Let $A_1\in \mathscr F$. Under $\D'_{q_n}$, for $\mathbb X_{n,1}$ as in \eqref{eq:block-def},
$$
\lim_{n\to\infty}\left|k_n\p(\mathbb X_{n,1}\in\mathbb A_{1})-n\p(\A_{n,1})\right|=\lim_{n\to\infty}\left|k_n\p\left(\W_{r_n}^c(A_{n,1})\right)-n\p(\A_{n,1})\right|=0, 
$$
where $A_{n,1}$ is defined in \eqref{eq:rectangles-n-bi} and $\mathbb A_{1}$ is from \eqref{eq:bb-A} applied to $A_1$.
\end{proposition}
\begin{proof}
We start by estimating $\p(\mathbb X_{n,1}\in\mathbb A_{1})=\p(\W_{r_n}^c(A_{n,1}))$, which we do by decomposing $\W_{r_n}^c(A_{n,1})$ according to the last occurrence of event $A_{n,1}$. Namely,
\begin{align*}
\p(\W_{r_n}^c(A_{n,1}))
&=\sum_{j=0}^{r_n-1}\p(\TT^{-j}(A_{n,1}))-\sum_{j=0}^{r_n-1}\p(\TT^{-j}(A_{n,1}), \W^c_{j+1,r_n}(A_{n,1}))).
\end{align*}
It follows that
\begin{align}
\Bigg|\p(\W_{r_n}^c(A_{n,1}))&-\left(\sum_{j=0}^{r_n-q_n-2}\left(\p(\TT^{-j}(A_{n,1}))-\p(\TT^{-j}(A_{n,1}), \W^c_{j+1,r_n}(A_{n,1}))\right)\right)\Bigg|\nonumber\\
&\leq 
\sum_{j=r_n-q_n-1}^{r_n-1}\p(\TT^{-j}(A_{n,1}), \W_{j+1,r_n}(A_{n,1}))) 
\leq (q_n+1)\p(A_{n,1})\label{eq:computation1}
\end{align}
Using stationarity,
\begin{align*}
\sum_{j=0}^{r_n-q_n-2}\p(\TT^{-j}(A_{n,1}), \W^c_{j+1,r_n}(A_{n,1}))&=\sum_{s=q_n+2}^{r_n}\p(A_{n,1}, \W^c_{1,s}(A_{n,1})).
\end{align*}
For $s\geq q_n+2$,
\begin{align*}
\p(A_{n,1}, \W^c_{1,s}(A_{n,1}))-\p(A_{n,1}, \W^c_{1,q_n+1}(A_{n,1}))
&=\p(\A_{n,1},\W^c_{q_n+1,s}(A_{n,1}) )
\end{align*}
and therefore
\begin{align}
\Bigg|\sum_{s=q_n+2}^{r_n}&\p(A_{n,1}, \W^c_{1,s}(A_{n,1}))-\sum_{s=q_n+2}^{r_n}\p(A_{n,1}, \W^c_{1,q_n+1}(A_{n,1}))\Bigg|\nonumber\\&\leq \sum_{s=q_n+2}^{r_n}\p(\A_{n,1},\W^c_{q_n+1,s}(A_{n,1}) )
\leq r_n \p(\A_{n,1},\W^c_{q_n+1,r_n}(A_{n,1}) ).\label{eq:computation2}
\end{align}
Combining \eqref{eq:computation1} and \eqref{eq:computation2} and using stationarity, we obtain
\begin{align*}
\Bigg|\p(\W_{r_n}^c(A_{n,1}))&-\left(\sum_{s=q_n+2}^{r_n}\p(A_{n,1})-\sum_{s=q_n+2}^{r_n}\p(A_{n,1}, \W^c_{1,q_n+1}(A_{n,1}))\right)\Bigg|\nonumber\\
&\leq (q_n+1)\p(A_{n,1})+r_n \p(\A_{n,1},\W^c_{q_n+1,r_n}(A_{n,1}) )
\end{align*}
Noting that the term between big brackets is equal to \\$(r_n-q_n+1)\p(A_{n,1}, \W_{1,q_n+1}(A_{n,1})) =(r_n-q_n+1)\p(\A_{n,1})$ then
multiplying by $k_n$ we obtain 
\begin{align*}
\left|k_n\p(\mathbb X_{n,1}\in \mathbb A_{1})-n\p(\A_{n,1})\right|\leq 2q_nk_n\p(\A_{n,1})+n\p(\A_{n,1},\W^c_{q_n+1,r_n}(A_{n,1}) ).
\end{align*}
The second term on the right vanishes by $\D'_{q_n}$. Since by definition of $A_{n,1}$, we have that
$A_{n,1}\subset \{\|X_0\|>u_n(h_0)\}$, where $h_0=\inf\{\|x\|\colon x\in H_0\}\geq0$,  
then $n\p( A_{n,1})\leq n\p(\|X_0\|>u_n(h_0))\xrightarrow[n\to\infty]{}h_0$. Recalling that $q_n=o(r_n)$, it follows that the first term on right also vanishes. 
\end{proof}

\begin{corollary}
\label{cor:useful-limits}
Under $\D'_{q_n}$,
\begin{equation*}
\lim_{n\to\infty} k_n\p(\W_{r_n}^c(U_n(\tau))=\theta\tau\qquad\text{and}\qquad
\lim_{n\to\infty} \frac{\p(\W_{r_n}^c(U_n(\tau))}{r_n\p(U_n(\tau))}=\theta.
\end{equation*}
\end{corollary}
\begin{proof}
By Proposition~\ref{prop:block}, \eqref{un} and \eqref{eq:EI},
\begin{align*}
\lim_{n\to\infty} k_n\p(\W_{r_n}^c(U_n(\tau))&=\lim_{n\to\infty} n\p(U_n^{q_n}(\tau))=\lim_{n\to\infty}n\p(U_n(\tau)) \frac{\p(U_n^{q_n}(\tau))}{\p(U_n(\tau))}=\tau\theta. 
\end{align*}
It follows that
\begin{align*}
\lim_{n\to\infty}& \frac{\p(\W_{r_n}^c(U_n(\tau))}{r_n\p(U_n(\tau))}=\frac{k_n\p(\W_{r_n}^c(U_n(\tau))}{n\p(U_n(\tau))}=\frac{\tau\theta}{\tau}=\theta. 
\end{align*}
 \end{proof}

\begin{lemma}
\label{lem:EI-tail}
Recalling the definition of the EI given in \eqref{eq:EI}, 
$$
\theta=\p\left(\inf_{j\geq 1}\|Y_j\|\geq 1\right)=\p\left(\inf_{j\leq -1}\|Y_j\|\geq 1\right).
$$
\end{lemma}
\begin{proof}
By \eqref{eq:EI}, stationarity and definition of the process $(Y_j)_{j\in\Z}$, given in assumption \eqref{Y-def}, we may write
\begin{align*}
\theta&
=\lim_{n\to\infty} \p\left(\frac1\tau \mathbb X_{n}^{r_n,r_n+q_n} \in  \W_{r_n+1,r_n+q_n+1}\left(U(1)\right)\;\middle\vert\; \|\X_{r_n}\|>u_n(\tau)\right)\\
&=\p\left(\inf_{j\geq 1}\|Y_j\|\geq 1\right).
\end{align*}
The second equality in the statement of the lemma follows easily from stationarity and standard arguments.
\end{proof}

The next result is instrumental because it shows how the transformed anchored tail process can be used to encode the information regarding 
the clustering. Essentially, it says that the joint distribution of the random variables in a block where an exceedance is observed (which makes it a cluster) is given by the transformed anchored tail process.  Recall that by  \eqref{eq:embedding},  $\mathbb X_{n, i}$ can be thought of as lying in $l_\infty$.
\begin{proposition}
\label{prop:piling-property}
Under the assumptions used to define the transformed anchored tail process and condition $\D'_{q_n}$, for every $\tau>0$ and for $\mathbb X_{n,1}$ as in \eqref{eq:block-def},
$$
\mathcal L\left(\tilde\pi\left(\frac1\tau\mathbb X_{n,1}\right)\;\middle\vert\; \mathbb X_{n,1}\in \W^c_{r_n}\left(U(\tau)\right) \right)\longrightarrow \mathcal L\left(\tilde\pi\left((Z_j)_{j\in\Z}\right)\right).
$$
\end{proposition}
\begin{proof}
In what follows we write $U^\tau$ for $U(\tau)$. As in Appendix~\ref{sec:app_weak_conv}, $\tilde{\mathscr J}$ is a convergence determining class and therefore we need to show that for all $A\in\mathscr F$ and corresponding $\tilde A\in\tilde{\mathscr J}$, such that $\p\left(\tilde\pi\left((Z_j)_{j\in\Z}\right)\in\partial \tilde A\right)=0$, we have 
\begin{align*}
\p\left(\tilde\pi\left(\frac1\tau \mathbb X_{n,1}\right) \in \tilde A\;\middle\vert\; \mathbb X_{n,1}\in \W^c_{r_n}\left(U^\tau\right) \right)&\longrightarrow \p\left(\tilde\pi\left((Z_j)_{j\in\Z}\right)\in \tilde A\right),
\end{align*}
which will follow if we show that 
\begin{align*}
\p\left(\frac1\tau \mathbb X_{n,1} \in \mathbb A\;\middle\vert\; \mathbb X_{n,1}\in \W^c_{r_n}\left(U^\tau\right) \right)&\longrightarrow \p\left((Y_j)_{j\in\Z}\in \mathbb A\;\middle\vert\; \inf_{j\leq -1}\|Y_j\|\geq 1\right),
\end{align*}
for all $\mathbb A\in \mathscr J$, such that $\p((Y_j)_{j\in\Z}\in \partial \mathbb A)=0$.
Recall that $\mathbb A$ is $\TT$-invariant, \ie $\TT^{-1}(\mathbb A)=\mathbb A$. 

We start estimating 
$\p(\mathbb X_{n,1} \in \tau \mathbb A\cap \W^c_{r_n}\left(U^\tau\right) )$ by decomposing the event on the right (which essentially says that at least one exceedance of $u_n(\tau)$ has occurred up to time $r_n-1$) with respect to the first time, $i=0,\ldots, r_n-1$, when that exceedance occurs, \ie $\{\mathbb X_{n,1}\in  \TT^{-i} (U^\tau\}=\{\|\X_i\|>u_n(\tau)\}$:
$$
\p\left(\mathbb X_{n,1} \in \tau \mathbb A\cap \W^c_{r_n}\left(U^\tau\right) \right)=\sum_{i=0}^{r_n-1}\p\left(\mathbb X_{n,1} \in \tau \mathbb A\cap \W_{i}\left(U^\tau\right)\cap \TT^{-i} (U^\tau) \right)
$$
Since $B_{n,i}:=\p\left(\mathbb X_{n,1} \in \tau \mathbb A\cap \W_{i}\left(U^\tau\right)\cap \TT^{-i} (U^\tau) \right)\leq \p(\|\X_i\|>u_n(\tau))$, 
\begin{equation}
\label{eq:estimate1}
\left|\p\left(\mathbb X_{n,1} \in \tau \mathbb A\cap \W^c_{r_n}\left(U^\tau\right) \right)-\sum_{i=q_n}^{r_n-1}B_{n,i}\right|\leq 
q_n\p(U_n(\tau))=:I(n)
\end{equation}
For $i\geq q_n$, we use $D_{n,i}:=\p\left(\mathbb X_{n,1} \in \tau \mathbb A\cap \W_{i-q_n,i}\left(U^\tau\right)\cap \TT^{-i} (U^\tau)\right)$ to estimate $B_{n,i}$. Namely,
\begin{align*}
|B_{n,i}-D_{n,i}|&\leq\p\left(\mathbb X_{n,1} \in U^\tau \cap \W_{1,q_n+1}\left(U^\tau\right)\cap \W^c_{q_n+1,i+1}\left(U^\tau\right)\right)
\end{align*}
Therefore,
\begin{align}
\left|\sum_{i=q_n}^{r_n-1}B_{n,i}\right.&\left. -\sum_{i=q_n}^{r_n-1}D_{n,i}\right|\leq (r_n-q_n)\p\left( U_{n}^{(q_n)}(\tau)\cap \mathscr W^c_{q_n+1,r_n}\left(U_n(\tau)\right)\right)=:I\!I(n)\label{eq:estimate2}
\end{align}
By stationarity and because $\mathbb A$ is $\TT$-invariant, for all $i=q_n, \ldots,  r_n-1$,
\begin{align*}
D_{n,i}&=\p\left(\mathbb X_{n,2} \in \tau \mathbb A\cap \W_{r_n+i-q_n,r_n+i}\left(U^\tau\right)\cap \TT^{-(r_n+i)} (U^\tau)\right)\\
&=\p\left(\mathbb X_{n}^{r_n-q_n,2r_n} \in\tau \mathbb A \cap  \W_{r_n-q_n,r_n}\left(U^\tau\right)\cap \TT^{-r_n} (U^\tau)  \right)=: D_{n}
\end{align*}
Then using estimates \eqref{eq:estimate1} and \eqref{eq:estimate2}, we obtain
\begin{equation*}
\Big|\p\left(\mathbb X_{n,1} \in \tau \mathbb A\cap \W^c_{r_n}\left(U^\tau\right) \right)-r_n\p\left(D_{n}\right)\Big|\\ \leq 2I(n)+I\!I(n).
\end{equation*}
Hence, letting $P_n:=\frac{r_n\p(U_n(\tau))}{\p\left(\mathbb X_{n,1}\in \W^c_{r_n}\left(U^\tau\right)\right)}\p\left(D_{n}\;\middle\vert\;\|\X_{r_n}\|>u_n(\tau)\right)$, we can write
\begin{align*}
\Bigg|\p\left(\frac1\tau \mathbb X_{n,1} \in \mathbb A\;\middle\vert\; \mathbb X_{n,1}\in \W^c_{r_n}\left(U^\tau\right) \right)&-
P_n\Bigg|
\leq (2I(n)+I\!I(n))\tfrac1{\p\left(\mathbb X_{n,1}\in \W^c_{r_n}\left(U^\tau\right)\right)}=:E(n).
\end{align*}
Since, by Corollary~\ref{cor:useful-limits}, $\lim_{n\to\infty} \frac{r_n\p(U_n(\tau))}{\p\left(\mathbb X_{n,1}\in \W^c_{r_n}\left(U^\tau\right)\right)}=\theta^{-1}$ and, by definition of $U_n(\tau)$ we have $\lim_{n\to\infty}k_n r_n\p(U_n(\tau))=\lim_{n\to\infty}n\p(U_n(\tau))=\tau$, it follows that for some $C>0$, we have
$$
E(n)\leq C(k_nI(n)+k_nI\!I(n)).
$$
By definition of the sequences $(k_n)_n$ and $(q_n)_n$, we have that $\lim_{n\to\infty}k_nq_n\p(U_n(\tau))=0,$ which means that $\lim_{n\to\infty}k_nI(n)=0$.
Observe also that $\D_{q_n}'$ implies that 
$\lim_{n\to\infty}k_nI\!I(n)=0$ and therefore
$
\lim_{n\to\infty}E(n)=0.
$

In order to get the result we need to check that
$$
\lim_{n\to\infty}
P_n=\p\left((Y_j)_{j\in\Z}\in\mathbb A\;\middle\vert\;\inf_{j\leq -1}\|Y_j\|\geq 1\right):=P.
$$
Since by Corollary~\ref{cor:useful-limits} and Lemma~\ref{lem:EI-tail}, we have  $\lim_{n\to\infty} \frac{\p\left(\mathbb X_{n,1}\in \W^c_{r_n}\left(U^\tau\right)\right)}{r_n\p(U_n(\tau))}=\theta$ and $\theta=\p(\inf_{j\leq -1}\|Y_j\|\geq 1)$, then we need to show that
$$
\lim_{n\to\infty}\p\left(\mathbb X_{n}^{r_n-q_n,2r_n} \in\tau \mathbb A \cap  \W_{r_n-q_n,r_n}\left(U^\tau\right) \;\middle\vert\;\|\X_{r_n}\|>u_n(\tau)\right)=P
$$
Note that $\{(Y_j)_{j\in\Z}\in \partial \W_{-\infty,0}\left(U^1\right)\}\subset \cup_{j\leq1}\{\|Y_j\|=1\}.$ Since, by assumption \eqref{spectral-process}, for every $j\in\Z$, we have $Y_j=\|Y_0\|\Theta_j$, with $\Theta_j$ independent of $\|Y_0\|$ and since the latter is uniformly distributed by Lemma~\ref{lem:Y0}, then $\p(\|Y_j\|=1)=0$, for all $j\in\Z$. 
Then 
the desired limit follows by definition of the sequence $(Y_j)_{j\in\Z}$ given in assumption \eqref{Y-def}.
\end{proof}

\begin{corollary}
\label{cor:spectral-independence}
Under the assumptions of Proposition~\ref{prop:piling-property}, we have that $L_Z$ and the process $(Q_j)_{j\in\Z}$ defined in \eqref{eq:Z-polar} satisfy
\begin{enumerate}

\item $L_Z$ is uniformly distributed on $[0,1]$;

\item $L_Z$ and $\tilde\pi\left((Q_j)_{j\in\Z}\right)$ are independent.

\end{enumerate}

\end{corollary}
\begin{proof}
For part $(1)$, note that $L_Z\leq \|Z_0\|\leq \|Y_0\|\leq 1$ a.s. Let $0\leq v\leq 1$. By Proposition~\ref{prop:piling-property}, \eqref{eq:normalisation} and Corollary~\ref{cor:useful-limits}, we may write
\begin{align*}
&\p(L_Z<v)=\lim_{n\to\infty}\p\left( \frac1\tau\mathbb X_{n,1} \in \W^c_{r_n}\left(U(v)\right)\;\middle\vert\; \mathbb X_{n,1}\in \W^c_{r_n}\left(U(\tau)\right)\right)\\
&=\lim_{n\to\infty}\p\left(  \W^c_{r_n}\left(U_n(\tau v)\right)\;\middle\vert\; \mathbb  \W^c_{r_n}\left(U_n(\tau)\right)\right)=\lim_{n\to\infty}\frac{\p\left(  \W^c_{r_n}\left(U_n(\tau v)\right)\right)}{\p\left(\mathbb  \W^c_{r_n}\left(U_n(\tau)\right)\right)}
=\frac{\tau v\theta}{\tau\theta}=v.
\end{align*}
To prove $(2)$, we start by observing that $\frac1{\|\Proj(\x)\|_\infty}=\inf_{j\in\Z}\|x_j\|$ and that the map $\tilde \x\mapsto \left(\tilde \x, \frac1{\|\tilde \Proj(\tilde \x)\|_\infty}\right)$ is continuous on $\tilde l_\infty\setminus\{\tilde \infty\}$. Then, by Proposition~\ref{prop:piling-property} and the CMT 
$$
\mathcal L\left(\tilde\pi\left(\tfrac1\tau\mathbb X_{n,1}\right),\; \tfrac1{\|\tilde \Proj\left(\tilde\pi\left(\tfrac1\tau\mathbb X_{n,1}\right)\right)\|_\infty}\;\middle\vert\; \mathbb X_{n,1}\in \W^c_{r_n}\left(U(\tau)\right) \right)\longrightarrow \mathcal L\left(\tilde\pi\left((Z_j)_{j\in\Z}\right), L_Z\right)
$$
Since the map $(\tilde \x,a)\mapsto \left(\frac{\tilde \x}{b}, a \right)$ is continuous on $\tilde l_\infty\setminus\{\tilde \infty\}\times (0,\infty)$, then  
\begin{multline}
\mathcal L\left({\left\|\tilde \Proj\left(\tilde\pi\left(\tfrac1\tau\mathbb X_{n,1}\right)\right)\right\|_\infty}\tilde\pi\left(\tfrac1\tau\mathbb X_{n,1}\right),\; \tfrac1{\|\tilde \Proj\left(\tilde\pi\left(\tfrac1\tau\mathbb X_{n,1}\right)\right)\|_\infty}\;\middle\vert\; \mathbb X_{n,1}\in \W^c_{r_n}\left(U(\tau)\right) \right)\\ \longrightarrow \mathcal L\left(\tilde\pi\left((Q_j)_{j\in\Z}\right), L_Z\right)
\label{eq:joint-limit}
\end{multline}
Since $\tilde{\mathscr J}$ is a convergence determining class (see Appendix~\ref{sec:app_weak_conv}), the result will follow if we show that for all $A\in\mathscr F$ and corresponding $\tilde A\in\tilde{\mathscr J}$, such that $\p(\tilde\pi\left((Q_j)_{j\in\Z}\right)\in \partial \tilde A)=0$, and all $v\in[0,1]$, we have 
\begin{equation}
\label{eq:independence}
\p\left(\tilde\pi\left((Q_j)_{j\in\Z}\right)\in \tilde A, L_Z<v \right)=\p\left(\tilde\pi\left((Q_j)_{j\in\Z}\right)\in \tilde A\right)\cdot\p(L_Z<v)
\end{equation}
Letting $\mathbb A\in \mathscr J$ be such that $\mathbb A= \tilde \pi^{-1}(\tilde A)$ and $m_{r_n}=\min\{\|\mathbb X_{n}^{0}\|,\ldots,\|\mathbb X_{n}^{r_n-1}\|\}$, by \eqref{eq:joint-limit}, we can write that 
\begin{align*}
&\p\Big((Q_j)_{j\in\Z}\in \mathbb A, L_Z<v \Big)=
\lim_{n\to\infty}\p\left( \tfrac1{m_{r_n}}\mathbb X_{n,1}\in\mathbb A\;\middle\vert\; \mathbb X_{n,1}\in\W^c_{r_n}\left(U(\tau v)\right)\right)\cdot\\
&\cdot\tfrac{\p\left(\W^c_{r_n}\left(U_n(\tau v)\right)\right)}{\p\left(\W^c_{r_n}\left(U_n(\tau)\right)\right)}=\p\left(\tilde\pi\left((Q_j)_{j\in\Z}\right)\in \tilde A\right)\cdot v=\p\left(\tilde\pi\left((Q_j)_{j\in\Z}\right)\in \tilde A\right)\cdot\p(L_Z<v).
\end{align*}
\end{proof}
The next result, an analogue of \cite[Lemma 3.3]{BPS18} in our setting,  formally establishes the convergence of the intensity measures of the cluster point processes we introduce later. We refer to Appendix~\ref{appendix:convergence-point-processes} for the definitions of weak$^\#$ convergence and boundedly finite measures.

\begin{corollary}
\label{cor:average-convergence}
Under the assumptions used to define the transformed anchored tail process and condition $\D'_{q_n}$, the sequence of boundedly finite measures $\eta_n=k_n\p(\tilde\pi(\mathbb X_{n,1})\in \cdot)$ in $\mathcal M_{\tilde l_\infty\setminus\{\tilde\infty\}}^\#$ converges in the $w^\#$ topology to $\eta=\theta(\text{Leb}\times\p_{\tilde {\mathbf Q}} )\circ\hbar$, where $\p_{\tilde{\mathbf Q}}$ is the distribution of $\tilde \pi((Q_j)_{j\in\Z}))$.
\end{corollary}

\begin{proof}
By Lemma~\ref{lem:cdc-hash}, we only need to check the convergence for all bounded $\tilde A \in \tilde{\mathscr J}$ such that $\mu(\partial\tilde A)=0$. By Remark~\ref{rem:bounded-sets}, since $\tilde A$ is bounded, there exists $\tau>0$ such that for all $\tilde \x\in\tilde A$, we have $\inf_{j\in\Z}\|x_j\|< \tau$. Hence, by \ref{eq:normalisation},  if $\tilde\pi(\mathbb X_{n,1})\in \tilde A$, then  $\|\X_j\|>u_n(\tau)$ for $j=0, \ldots, r_n-1$ and hence $\mathbb X_{n,1}\in \W^c_{r_n}\left(U(\tau)\right)$. Let, as before, $\mathbb A=\tilde\pi^{-1}(\tilde A)$. Then
\begin{align*}
&k_n\p(\tilde\pi(\mathbb X_{n,1})\in \tilde A)=k_n\p(\mathbb X_{n,1}\in \mathbb A)=k_n\p\left(\mathbb X_{n,1}\in \mathbb A\cap \W^c_{r_n}\left(U(\tau)\right)\right)\\
&=n\p(\|\X_0\|>u_n(\tau))\tfrac{\p\left(\mathbb X_{n,1}\in \W^c_{r_n}\left(U(\tau)\right)\right)}{r_n\p(\|\X_0\|>u_n(\tau))}\p\left(\frac{\mathbb X_{n,1}}{\tau}\in \tau^{-1}\mathbb A\;\middle\vert\; \mathbb X_{n,1}\in\W^c_{r_n}\left(U(\tau)\right)\right)
\end{align*}
By \eqref{un} and Corollary~\ref{cor:useful-limits}, we have that the first term converges to $\tau$ and the second to $\theta$, as $n\to\infty$. By Proposition~\ref{prop:block}, the third term goes to $\p((Z_j)_{j\in\Z}\in \tau^{-1}\mathbb A)$.
We now use Corollary~\ref{cor:spectral-independence} in order to finish the proof.
\begin{align*}
&\p((Z_j)_{j\in\Z}\in \tau^{-1}\mathbb A)=\p(\tau (Z_j)_{j\in\Z}\in \mathbb A)=\int_0^1\p\left(\tau (Z_j)_{j\in\Z}\in \mathbb A\;\middle\vert\; L_Z=v\right) dv\\
&
= \int_0^1\p\left(\tau v (Q_j)_{j\in\Z}\in \mathbb A\;\middle\vert\; L_Z=v\right) dv =\frac1\tau\int_0^\tau\p\left(s (Q_j)_{j\in\Z}\in \mathbb A\right) ds=\frac{\eta(\tilde A)}{\tau\theta}. \end{align*}
\end{proof}

\subsection{Complete convergence of point processes}\label{subsubsec:complete-convergence}
\label{subsec:complete-convergence}

We define and prove the weak convergence of the point processes that keep all the cluster information. We refer to Appendix~\ref{appendix:convergence-point-processes} for the precise definition of point processes and their weak convergence. Essentially, we consider a random element on the space $\mathcal N_{\R_0^+\times\tilde l_\infty\setminus\{\tilde\infty\}}^\#$ of boundedly finite point measures on $\R_0^+\times\tilde l_\infty\setminus\{\tilde\infty\}$. Namely, similarly to \cite{BPS18}, we define the point processes of clusters by  
\begin{equation}
\label{eq:point-process-def}
N_n=\sum_{i=1}^\infty \delta_{(i/k_n,\tilde\pi(\mathbb X_{n,i}))}.
\end{equation}
Now, we define the point process that will appear as the limit of the cluster point process.  Let $(T_{i})_{i\in\N}$ and  $(U_{i})_{i\in\N}$ be such that $\sum_{i=1}^\infty\delta_{(T_i,U_i)}$ is a bidimensional  Poisson point process on $\R_0^+\times \R_0^+$ with intensity measure $\leb\times\theta\,\leb$.
Also let $(\tilde{\mathbf Q}_i)_{i\in\N}$ be an \iid  sequence of random elements in $\mathbb S$ such that each $\tilde{\mathbf Q}_i$ has a distribution given by \eqref{eq:Z-polar}.
We assume that the sequences $(T_{i})_{i\in\N}$, $(U_{i})_{i\in\N}$ and $(\tilde{\mathbf Q}_i)_{i\in\N}$ are mutually independent. 
We define
\begin{equation}
\label{eq:limit-process}
N=\sum_{i=1}^\infty \delta_{(T_{i}, U_{i}\tilde{\mathbf Q}_i)}.
\end{equation}
\begin{remark}
\label{rem:limit-description}
Note that $N$ above is a Poisson point process on $\R_0^+\times\tilde l_\infty\setminus\{\tilde\infty\}$ with intensity $\text{Leb}\times \eta$, where $\eta$ is as in Corollary~\ref{cor:average-convergence}.  Here $\eta$ describes the distribution of the transformed anchored tail process $\tilde\pi((Z_j)_{n\in\Z})$, which is characterised by means of the spectral decomposition $\hbar$ given in \eqref{eq:psi.def}, which allows us to identify the contribution from each component, namely, the $\theta\,\leb$ part associated to $L_Z$ (related to the second coordinate of the bidimensional  Poisson point process $\sum_{i=1}^\infty\delta_{(T_i,U_i)}$) and the $\p_{\tilde{\mathbf Q}}$ part which is the distribution of $\tilde \pi((Q_j)_{j\in\Z}))$. 
\end{remark}
\begin{remark}
\label{rem:oscillatory-example-Qs}
In order to have some intuition of what is being encoded in each $\tilde{\mathbf Q}_i$, we recall that in the case of Exemple~\ref{example:oscillatory}, then $\tilde{\mathbf Q}_i= \tilde \pi((Q_{i,j})_{j\in\Z}))$, where  $Q_{i,j}=(-3)^jE_i$, for all $j\in\N_0$ and $Q_{i,j}=\infty$ for all $j\in\Z\setminus\N_0$; moreover 
 $E_1, E_2,\ldots$ is an \iid sequence independent of $(T_i)_{i\in\N}$ and $(U_i)_{i\in\N}$, with $\p(E_1=1)=1/2=\p(E_1=-1)$.
\end{remark}

We are now ready to state a general complete convergence result.

\begin{theorem}\label{thm:complete-convergence} 
Let $\X_0, \X_1, \ldots$ be a stationary process of random vectors in $\R^d$ with proper tails, in the sense given in Section~\ref{subsec:normalising}. Assume that the transformed anchored tail process given in Definition~\ref{def:piling-process} is well defined and conditions $\D_{q_n}$ and $\D'_{q_n}$ hold. Then point process $N_n$, given in \eqref{eq:point-process-def} converges weakly in $\mathcal N_{\R_0^+\times\tilde l_\infty\setminus\{\tilde\infty\}}^\#$ to the Poisson point process $N$ given by \eqref{eq:limit-process}.
\end{theorem}

\begin{remark}
\label{rem:earlier-results}
We observe that the point processes convergence stated here corresponds essentially to a transformed version of similar statements in \cite[Theorem~3.6]{BPS18} and \cite[Corollary~7.4.2]{KS20}, but we emphasise that the dependence conditions assumed here are both weaker, which is of crucial importance for the application to stochastic processes arising from dynamical systems. We refer to Corollary~\ref{cor:point-process-heavytail}, in Section~\ref{subsec:regular-variation}, where our point processes, and convergence properties, are transformed back to provide an easier comparison with previous results and also formulae on how to relate the various objects involved.
\end{remark}

In order to prove the theorem, we need essentially to show two things. The first is a sort of \emph{independent increments} property together with convergence in distribution of joint random variables using the avoidance function, \ie by computing the probability of having no extremal occurrences (see the definition of the avoidance function in Appendix~\ref{appendix:convergence-point-processes}). This is done in Proposition~\ref{prop:indep-increments}, whose complete proof is rather lengthy because we are using the very weak mixing assumption $\D_{q_n}$, though a lot of the work for that has been done in previous papers by the authors. Then, in the second step, we need to show that the intensity measures of the processes converge to the right intensity measure. This has actually already been done in Corollary~\ref{cor:average-convergence}. Finally we need to join the pieces using the theory of weak$^\#$ convergence on non locally compact spaces developed in \cite{DV03,DV08}. In fact, we needed to redo one of the results to correct a typo and improve it in order to be able to use the convergence of the intensity measures (see Appendix~\ref{appendix:convergence-point-processes}).

In \cite{FFM20} the existence of a $\sigma$-finite outer measure $\nu$ on $\V^{\N_0,\Z}$ was assumed, so that the following limit exists
\begin{equation}
\label{eq:limit-Anq}
\lim_{n\to\infty} n\p(\A_{n,\ell})=\nu(A_\ell),
\end{equation}
for all $A_\ell\in\mathscr F$ and $\A_{n,\ell}$  given by equation \eqref{eq:rectangles-n-bi}. This outer measure described the piling of points on the multidimensional point processes created by clustering in \cite{FFM20}.  
Note that if we associate $\tilde A_\ell$, $\mathbb A_\ell$ to $A_\ell\in\mathscr F$ as in \eqref{eq:A-tilde} and \eqref{eq:bb-A}, respectively, then using Proposition~\ref{prop:block} and Corollary \ref{cor:average-convergence}, it follows that when we have the existence of a transformed anchored tail process and condition $\D_{q_n}'$ then \eqref{eq:limit-Anq} holds and
\begin{equation}
\nu(A_\ell)=\mu(\tilde A_\ell).
\label{eq:outer-measure}
\end{equation}

We state now the main result that provides independence of disjoint time pieces and convergence of joint distributions by use of the avoidance function. 

\begin{proposition}\label{prop:indep-increments} 
Let $m\in\N$ and for each $\ell=1,\ldots,m$ let $J_\ell,\;A_\ell$ be given, as in  \eqref{eq:rectangles}. For $n\in\N$, consider the respective versions $J_{n,\ell},\;A_{n,\ell}$ given in \eqref{eq:Jnl-def} and \eqref{eq:rectangles-n-bi}. Assume that $\D_{q_n}$ and $\D'_{q_n}$ hold. Also assume that there exists a $\sigma$-finite outer measure $\nu$ such that \eqref{eq:limit-Anq} holds. Then
\[
\lim_{n\to\infty}\p\left(\bigcap_{\ell=1}^m\W_{J_{n,\ell}}\left(A_{n,\ell}\right)\right)=\lim_{n\to\infty}\p\left(\bigcap_{\ell=1}^m\W_{J_{n.\ell}}\left(\A_{n,\ell}\right)\right)=\prod_{\ell=1}^m \e^{-\nu(A_\ell)|J_\ell|}.
\]
\end{proposition}

The first equality follows from the fact that the non-occurrence of the asymptotically rare event $A_{n,\ell}$  can be replaced by the non-occurrence of the event $\A_{n,\ell}$, up to an asymptotically negligible error. This idea goes back to \cite[Proposition~1]{FFT12} and was further developed in \cite[Proposition 2.7]{FFT15}. We refer to \cite[Proposition~3.2]{FFM20} for a proof. The second equality in Proposition~\ref{prop:indep-increments} follows from minor adjustments to the  argument used to prove \cite[Theorem~3.3]{FFM20}.

We are now ready to prove the weak convergence of the cluster point processes.

\begin{proof}[Proof of Theorem~\ref{thm:complete-convergence}]
By Proposition~\ref{prop:weak-limit-PP} and Lemma~\ref{lem:cdc}, we first need to check that for all bounded sets $B\in \mathcal I$, defined in \eqref{eq:sub-ring-def} we have $
\lim_{n\to\infty}\p(N_n(B)=0)=\p(N(B)=0).
$
Let $B=\cup_{\ell=1}^m J_\ell\times\tilde A_\ell$, where for each $\ell=1,\ldots,m$, we have $J_\ell=[a_\ell,b_\ell)$ and $\tilde A_{\ell}$ is associated to some $A_\ell\in \mathscr F$ as in \eqref{eq:A-tilde}. 

Since $\tilde\pi(\mathbb X_{n,i})\in\tilde A_\ell$ corresponds to the event $\mathbb X_{n,i}\in\mathbb A_\ell=\cup_{i\in\Z}\TT^{-i}(A_\ell)$, then, by definition of the sets $J_{n,\ell}$ and $A_{n,\ell}$ given in \eqref{eq:Jnl-def} and \eqref{eq:rectangles-n-bi}, we have
\begin{align*}
\p(N_n(B)=0)=\p\left(\bigcap_{\ell=1}^m \{N_n(J_\ell\times\tilde A_\ell)=0\}\right)=\p\left(\bigcap_{\ell=1}^m\W_{J_{n,\ell}}\left(A_{n,\ell}\right)\right).
\end{align*}

Now, by definition of the Poisson point process $N$, we have $\p(N(B)=0)=\prod_{\ell=1}^m\e^{-|J_\ell|\mu(\tilde A_\ell)}$. Hence, condition \ref{eq:kallenberg1} of Proposition~\ref{prop:weak-limit-PP} follows from Proposition~\ref{prop:indep-increments} and \eqref{eq:outer-measure}.

In order to check condition \ref{eq:kallenberg2} of Proposition~\ref{prop:weak-limit-PP}, we observe that by stationarity and Corollary~\ref{cor:average-convergence} we have  
\begin{align*}
\E(N_n(B))&=\E\left(\sum_{\ell=1}^m\sum_{i= \lceil k_n a_\ell\rceil}^{\lceil k_n b_\ell\rceil-1}\I_{\{\tilde\pi(\mathbb X_{n,i})\in\tilde A_\ell\}}\right)\sim \sum_{\ell=1}^m|J_\ell| k_n\p\left(\tilde\pi(\mathbb X_{n,i})\in\tilde A_\ell\right)\\
&\xrightarrow[n\to\infty]{} \sum_{\ell=1}^m|J_\ell|\mu(\tilde A_\ell)=\E(N(B)).
\end{align*}
\end{proof}

\subsection{Applications to jointly regularly varying sequences} 
\label{subsec:regular-variation}

In statistics of extremes there is a special interest in the shape of the tail of the distributions and, particularly, heavy tails play a significant role. In this setting, the study of the mean and of the extremes is linked and point processes and regularly varying measures have revealed to be very useful tools (see \cite{R87}, for example). In this heavy tail context, the information regarding clustering is particularly well captured by the tail process introduced in \cite{BS09} and, in fact, the process $(Y_j)_{j\in\Z}$ used to define the transformed anchored tail process, in Section~\ref{subsubsec:piling-process}, can be identified as a transformed version in $l_\infty$ of the original tail process, which lives in $l_0$.

In the study of rare events for dynamical systems, one is not so interested in different tail behaviours but rather on the statistical properties of the system, which are intimately related with quantitative recurrence properties (see \cite{FFT10}). For this reason the particular homeomorphism $f$ establishing the shape of $\tau$ (see Remark~\ref{rem:homeomorphism}) is not as relevant as proving the existence of a limiting law and therefore we tried to establish the definitions and devices in a more general setting so that they are more amenable for application to dynamical systems. The universality of the uniform distribution we get, for example, in Lemma~\ref{lem:Y0} (with no parameter involved as opposed to the $\alpha$ appearing in the tail process version) and the fact that the limiting point process has a Poisson component with $\leb\times\theta\leb$ intensity measure partly motivated introducing the transformed anchored tail process, which is related to the piles of clustering points obtained in the limiting rare events point processes studied in \cite{FFM20}. 

In the stationary heavy tail setting, the existence of the tail process is equivalent to joint regular variation of the process $(\mathbf X_i)_{i\in\Z}$ (see \cite[Theorem~2.1]{BS09}), a notion that we define next.
\begin{definition}
\label{def:jointly-rv}
A $k$-dimensional random vector $\mathbf X$ is said to be \emph{jointly regularly varying},  with index $\alpha>0$, if there exists a sequence of constants $(a_n)_{n\in\N}$ and a random vector $\mathbf \Theta$ with $\p(\|\mathbf \Theta\|=1)=1$, such that  
$$
n\p(\|\mathbf X\|>x a_n, \; \mathbf X/\|\mathbf X\|\in \cdot)\xrightarrow[n\to\infty]{w} x^{-\alpha}\p(\mathbf \Theta\in\cdot).
$$
where we are considering weak convergence of measures on $\mathbb S^{k-1}$, the unit sphere in $\R^k$. An $\R^d$-valued sequence $\mathbf X_0, \mathbf X_1,\ldots$ is said to be \emph{jointly regularly varying}, with index $\alpha>0$, if all the finite-dimensional vectors $(\mathbf X_k,\ldots, \mathbf X_\ell)$, $k\leq \ell\in\N_0$, are jointly regularly varying, with index $\alpha>0$. 
\end{definition}
In \cite{BS09,BPS18,KS20}, for example, the original stochastic process is assumed to be jointly regularly varying. In what follows, we show that, in this setting, we recover the convergence of the point process of clusters stated in these works and give the particular relation between the tail process and the transformed anchored tail process.

For a jointly regularly varying sequence $\mathbf X_0, \mathbf X_1,\ldots$, in particular, condition~\eqref{eq:heavy-normalisation} holds. Hence, taking $\tau=x^{-\alpha}$ and $u_n(\tau)=a_n\tau^{-\frac1\alpha}$, equation \eqref{un} holds. We also have that $u_n^{-1}(z)=z^{-\alpha}a_n^\alpha$. 

\begin{remark}
\label{rem:tail-process-existence}
Observe that joint regular variation of  $(\X_j)_{j\in\Z}$ implies the existence of the sequence $(Y_j)_{j\in\Z}$ assumed in \eqref{Y-def} (see \cite[Theorem~2.1]{BS09}) and also the independence of the respective polar decomposition assumed in \eqref{spectral-process} (see \cite[Theorem~3.1]{BS09}). Moreover, If, instead of  $\D'_{q_n}$, we assume the much stronger conditions considered in \cite[Assumption~1.1]{BPS18} and other previous works, going back to  \cite{S92}, one can show that both \eqref{Y-infinity} and \eqref{positive-EI} also hold. See \cite[Proposition~4.2]{BS09}. 
\end{remark}

In order to make the connection between the tail and the transformed tail processes, we start by defining the map: 
\begin{align}
\label{eq:xi}
\xi \colon \left(\R^d\cup\{\infty\}\right)\setminus\{0\} &\longrightarrow \R^d \nonumber\\
x &\mapsto\begin{cases}
    (\|x\|)^{-\frac1\alpha}\frac{x}{\|x\|}\,,& \text{if } x\neq \infty\\
    0,              & \text{otherwise}
\end{cases},
\end{align}
Then, we define $\Xi \colon l_\infty \longrightarrow l_0$ given by $\Xi((x_j)_j)=(\xi(x_j))_j$. 
Observe that, like $\Proj$ (see Lemma~\ref{lem:Pinv}), the function $\Xi$ is invertible and we may define $\tilde \Xi:\tilde l_\infty\to \tilde l_0$ so that $\tilde \Xi (\tilde \pi(\mathbf x))=\tilde \pi(\Xi(\mathbf x))$.
Note that 
\begin{equation}
\xi\left({u_n^{-1}(\|\X_{j}\|)}\frac{\X_{j}}{\|X_{j}\|}\right)=\frac{\X_j}{a_n}.
\end{equation}
\begin{corollary}
\label{cor:point-process-heavytail}
Let $\X_0, \X_1, \ldots$ be a stationary $\R^d$-valued jointly regularly varying sequence, with tail index $\alpha>0$, satisfying conditions $\D_{q_n}$ and $\D'_{q_n}$ and for which the transformed anchored tail process given in Definition~\ref{def:piling-process} is well defined. Then  the point process 
\begin{equation}
\label{eq:heavy-tail-finite-pp}
N'_n=\sum_{i=1}^\infty \delta_{(i/k_n,\tilde \Xi(\tilde \pi(\mathbb X_{n,i})))}=\sum_{i=1}^\infty \delta_{\left(\frac{i}{k_n},\tilde \pi\left(\frac{\X_{(i-1)r_n}}{a_n},\ldots, \frac{\X_{ir_n-1}}{a_n}\right)\right)}
\end{equation} 
converges weakly in $\mathcal N_{\R_0^+\times\tilde l_0\setminus\{\mathbf 0\}}^\#$ to the Poisson point process $N'$ given by 
\begin{equation}
\label{eq:heavy-tail-limit-pp}
N'=\sum_{i=1}^\infty \delta_{(T_{i}, U_{i}^{-\frac1\alpha}\tilde \Xi(\tilde{\mathbf Q}_i))},
\end{equation}
where $(T_i)_{i\in\N}$, $(U_i)_{i\in\N}$ and $(\tilde{\mathbf Q}_i)_{i\in\N}$ are as in \eqref{eq:limit-process}.
\end{corollary}
This corollary of Theorem~\ref{thm:complete-convergence} follows from a direct application of the CMT for the map $\Xi^{\#}:\mathcal N_{\R_0^+\times\tilde l_\infty\setminus\{\tilde \infty\}}^\#\to\mathcal N_{\R_0^+\times\tilde l_0\setminus\{\mathbf 0\}}^\#$ defined by 
$$\Xi^{\#}\left(\sum_{i=1}^\infty \delta_{(t_i,\tilde{\mathbf x}_i )}\right)=\sum_{i=1}^\infty \delta_{\left(t_i,\tilde \Xi\left(\tilde{\mathbf x}_i \right)\right)}.$$

Observe that the Poisson component of the process $N'$ can be written as $M_\alpha=\sum_{i=1}^\infty\delta_{(T_i,P_i)}=\sum_{i=1}^\infty\delta_{(T_i,U_i^{-1/\alpha})}$, with intensity measure $\leb\times \theta\,\nu_\alpha$, with $\nu_\alpha(y)=\d(-y^{-\alpha})$, while the angular component associated to the tail process, which we denoted earlier by $(\mathcal Q_{j})_{j\in\Z}$, is given by the equation: 
\begin{equation}
\label{eq:Q-relation}
\mathcal Q_j=\xi(Q_j).
\end{equation}

Recall that as observed in Remark~\ref{rem:tail-process-existence}, the existence of transformed anchored tail process and the respective properties are guaranteed by the joint regular variation of the process $\X_0, \X_1, \ldots$ and \cite[Assumption~1.1]{BPS18} (assumption $\mathcal{AC}(r_n,c_n)$ in \cite{KS20}).

\section{Proofs of the dynamical enriched functional limit theorems}
\label{sec:proof-FLT}

The major step to obtain the invariance principles stated in Theorems~\ref{thm:sums-heavy-tails}, \ref{thm:extremal-process} and \ref{thm:record-theorem} for dynamically defined stochastic processes is the complete convergence of point processes stated as follows
\begin{theorem}
\label{thm:point-process-convergence}
Let $T:\mathcal X \to \mathcal X$ be a dynamical system  as described in Section~\ref{subsec:systems} and $\X_0,\X_1,\ldots$ be obtained from such a system as described in \eqref{eq:dynamics-SP} and assume that the transformed anchored tail process given in Definition~\ref{def:piling-process} is well defined.
Consider the point process $N_n$ defined as in \eqref{eq:point-process-def}. Then $N_n$ converges weakly in $\mathcal N_{\R_0^+\times\tilde l_\infty\setminus\{\tilde\infty\}}^\#$ to the Poisson point process $N$ given by \eqref{eq:limit-process}.
\end{theorem}

It is sufficient to show that the systems and the observables that we consider give rise to stochastic processes for which $\D_{q_n}$ and $\D'_{q_n}$ hold, since then the conclusion follows immediately by Theorem~\ref{thm:complete-convergence}.   A property we use to prove these conditions is in the next definition

\begin{definition}
\label{def:DC}
Let $\C_1$ and $C_2$ be Banach spaces of real-valued measurable functions on $\mathcal X$.  Define the \emph{correlation} of non-zero functions $\phi\in \C_1$ and $\psi\in \C_2$ with respect to $\mu$ at time $n\in \N$ by
$$\text{Cor}_\mu(\phi, \psi, n): = \frac1{\|\phi\|_{\C_1} \|\psi\|_{\C_2}} \left|\int\phi(\psi\circ T^n)~d\mu- \int \phi~d\mu\int \psi~d\mu\right|.$$
Then say that the system has \emph{decay of correlations}, with respect to $\mu$, for observables in $\C_1$ \emph{against} observables in $\C_2$ if there exists a rate function $\rho:\N\to [0, \infty)$ with 
$$\lim_{n\to \infty}\rho(n)=0,$$
and for every $\phi\in \C_1$, $\psi\in \C_2$, 
$$\text{Cor}_\mu(\phi, \psi, n)\le \rho(n).$$
\end{definition}

The uniformly expanding systems in Section~\ref{sssec:non-invertible} have decay of correlations against $L^1(\mu)$, i.e., where $\C_2=L^1(\mu)$.  In that setting we also require a suitable space $\C_1$, a key feature being that the characteristic functions on our sets of interest, like the annuli $A^{(j)}$ in \eqref{eq:Aq} do not have large norm.  In fact in the interval setting $\C_2$ will be the $BV$ norm which we recall here.  If $\psi:I \to \R$ is a measurable function on an interval $I$ then its \emph{variation} is defined as
$$\text{Var}(\psi):= \sup\left\{\sum_{i=0}^{n-1}|\psi(x_{i+1})-\psi(x_i)|\right\},$$
where the supremum is taken over all finite ordered sequences $(x_i)_{i=0}^{n-1}$ in $I$.  The $BV$ norm is $\|\psi\|_{BV}:=\sup|\psi|+\text{Var}(\phi)$ and $BV:=\{\psi:I\to \R: |\psi\|_{BV}<\infty\}$.  In Saussol's class of higher-dimensional expanding maps $\C_1$ is a quasi-H\"older norm.

\begin{remark}
While our conditions on the \emph{rate} of decay of correlations here may appear very weak, in fact summable decay of correlations against $L^1(\mu)$ implies exponential decay of correlations for H\"older observables against $L^\infty$, as in \cite[Theorem~B]{AFLV11}.
\label{rmk:AFLV}
\end{remark}

We make a brief list of references where one can find the arguments to prove  $\D_{q_n}$ and $\D'_{q_n}$ for the systems mentioned in Section~\ref{subsec:systems}. For non-invertible systems admitting decay of correlations against $L^1$ and observables with maximal sets $\mathcal M$ consisting on periodic points or a countable number of points in the same orbit, we note that our conditions on the system and the observable can be expressed as requiring, for $A_{n,\ell}^{(q_n)}$, as in \eqref{eq:rectangles-n-bi} and \eqref{eq:Aq},

\begin{enumerate}
\item \label{item:nuh1} $\lim_{n\to \infty} \|\I_{A_{n,\ell}^{(q_n)}}\| n\rho(t_n)=0$ for some sequence $(t_n)_n$ with $t_n=o(n)$,
\item \label{item:nuh2} $\lim_{n\to \infty} \|\I_{A_{n,\ell}^{(q_n)}}\|\sum_{j=q_n}^n\rho(j)=0$.
\end{enumerate}

These conditions can be found in, for example \cite{AFFR17, FFRS20}.   Along with decay of correlations against $L^1(\mu)$, they imply $\D_{q_n}$, $\D'_{q_n}$:  for proofs see those references. See also \cite[Section~4.3]{FFM20}. 

Regarding the Benedicks-Carleson maps equipped with observables maximised at periodic points, the required estimates to satisfy $\D_{q_n}$ can be found in \cite{FF08} and \cite[Sections~5,6]{FFT13}.

For Anosov linear diffeomorphisms on the torus and observables maximised at periodic points, we refer to \cite{CFFH15}. 

For the slowly mixing systems of the Manneville-Pomeau type and observables maximised at periodic points distinct from the indifferent fixed point, we could use a direct approach to prove conditions $\D_{q_n}$ and $\D'_{q_n}$ using the ideas in \cite[Section~4]{FFTV16}, but since decay of correlations is stated for H\"older continuous functions, an approximation to indicator functions is needed which restricts the domain of application to $\gamma\in(0,0.289)$, whereas one would expect the results to hold at least for $\gamma<1/2$ (when one has summable decay of correlations). In order to show that the results hold for observables with a spike at periodic points (or even with a finite number of spikes belonging to the same orbit) for all $\gamma\in(0,1/2)$ and even for  $\gamma\in[1/2,1)$, if we take the observables vanishing at the origin, in Section~\ref{subsec:inducing} we introduce a new point process which incorporates the idea of inducing.

\subsection{Proof of Theorem~\ref{thm:sums-heavy-tails}}
\label{subsec:proof-heavytails}

As a consequence of Theorem~\ref{thm:point-process-convergence}, since in this case we are assuming that $g$ appearing in \eqref{eq:dynamics-SP-norm} is of type $g_2$ and condition \eqref{eq:heavy-normalisation} holds, we obtain, by direct application of the CMT for the map $\Xi^\#$ given in Section~\ref{subsec:regular-variation}, that the convergence of point processes stated in Corollary~\ref{cor:point-process-heavytail} holds. Therefore, we are left to show that such point process convergence implies the convergence in $F'$ of the continuous time process $S_n$ to $\underline V$.

\begin{proposition}
\label{prop:convPP=>convSum}
Let $\X_0, \X_1, \ldots$ be a stationary $\R^d$-valued process for which condition \eqref{eq:heavy-normalisation} holds and moreover $N'_n$ defined in \eqref{eq:heavy-tail-finite-pp} converges weakly$^\#$ in $\mathcal N_{\R_0^+\times\tilde l_0\setminus\{\mathbf 0\}}^\#$ to $N'$ given in  \eqref{eq:heavy-tail-limit-pp}
then, under the same assumptions of Theorem~\ref{thm:sums-heavy-tails}, the conclusion regarding the convergence of $S_n$ to $\underline V$ holds, in $F'$.

\end{proposition}

\begin{proof}

Recall that the convergence in $F'([0,1],\R^d)$ consists of showing that the respective projections into $E([0,1],\R^d)$ and $\tilde D([0,1],\R^d)$ converge. 
The choice of metric in $E([0,1],\R^d)$ (see \eqref{eq:me}) implies that the convergence in this space will follow from the convergence of the coordinate projections, in $E([0,1],\R)$, which follows immediately from \cite[Theorem~4.5]{BPS18}. Hence, we are left to check the convergence of the $\tilde D([0,1],\R^d)$ counterparts, which we prove by splitting the argument into the same steps considered in \cite[Theorem~4.5]{BPS18} so that we can keep track of the required adjustments.

We start by defining a projection  $\Upsilon\colon\mathcal N_{\R_0^+\times\tilde l_0\setminus\{\mathbf 0\}}^\#\to F'([0,1], \R^d)$. Suppose we are given $\gamma=\sum_{i=1}^\infty \delta_{(t_i, \tilde x^i)}\in \mathcal N_{\R_0^+\times\tilde l_0\setminus\{\mathbf 0\}}^\#$.  At time $t_i$ we have $\tilde x^i\in \tilde l_0$: let $(\ldots, x_{-1},x_0,x_1,  \ldots)$ be a representative of this in $l_0$ and define 
$$\mathring{e}_x^{t_i}(t)=\sum_{i=-\infty}^{\left\lfloor \h(t)\right\rfloor}{x_i},\quad x(t)=\sum_{T_i\leq t}\mathring{e}_x^{T_i}(1), \quad S^x=\{t_i\}_i,\quad e_x^{t_i}=x(t_i^-)+\mathring{e}_x^{t_i}.$$
Note that since $e_x^{t_i}$ is really an element of an equivalence class, the particular representative of $\tilde x$ chosen does not matter. Finally, let $\Upsilon(\gamma)=(x,S^x, \{e_x^{t_i}\}_{i})$.

We start assuming $\alpha\in (0, 1)$.

\textbf{Step 1.} For $\eps>0$, we define an $\eps$-truncated projection $\Upsilon^\eps$.  We do this as follows.  For $\tilde x\in \tilde l_0$ set
$$
u^\eps(\tilde x)(t) = \sum_{j=-\infty}^{\left\lfloor\h(t)\right\rfloor} x_j \mathbbm{1}_{\{|x_j|>\eps\}}.
$$
Then for $\gamma = \sum_{i=1}^\infty \delta_{T_i, \tilde x^i}\in \mathcal{N}_{\X}^{\#}$, define $ S^\eps= \{T_i: \|\tilde x_i\|_\infty>\eps\}$ and
$$\Upsilon^\eps(\gamma)=\left(\left\{\sum_{T_i\le t} u^\eps(\tilde x^i)(1)\right\}_{t\in [0, 1]}, \ S^\eps ,  \{u^\eps(\tilde x^s)\}_{s\in S^\eps} \right)$$
(here we understand $\tilde x^s$ for $s=T_i\in S^\eps$ as being $\tilde x^i$).

To show that given  $\gamma_n\rightarrow_{w^\#}\gamma$,  we have  $\tilde \pi (\Upsilon^\eps(\gamma_n)) \to \tilde \pi(\Upsilon^\eps(\gamma))$, we first define a simpler space which is not essential here, but is intended to help the reader's visualisation of the situation.

Given an element $z\in \tilde D$, let $(t_i)_i$ be the set $disc(z)$ of discontinuities of a representative of $z$, which we abuse notation and also call $z$.  Let $v_i(z)=\sup_{t_i\le s_1, s_2<t_{i+1}} \|z(s_1)-z(s_2)\|$ and $j(t_i)=\|z(t_i^-)-z(t_i)\|$.  Then for $\eps>0$ set
$$ \tilde D_\eps:=\{z\in \tilde D: \#disc(z)<\infty, v_i(z)=0 \text{ and } j_i(z)>\eps \text{ for all } i\}.$$

Now notice that since we are based in $\mathcal N_{\R_0^+\times\tilde l_0\setminus\{\mathbf 0\}}^\#$, which only allows a finite number of jumps of norm higher than $\eps$,  each element of $  \{u^\eps(\tilde x^s)\}_{s\in S^\eps} $ lies in $\tilde D_\eps$.  Moreover for $\ux=\Upsilon^\eps(\gamma)$, $\tilde \pi(\ux)=\ux^{\tilde D}\in \tilde D_\eps$.  We can easily see in this simpler space that if $\gamma_n\rightarrow_{w^\#} \gamma$ then for  $\ux_n=\Upsilon^\eps(\gamma_n)$ and  $\ux=\Upsilon^\eps(\gamma)$, we have ${\ux}_n^{\tilde D}\to \ux^{\tilde D}$.

These arguments together complete the proof of Step 1.

\textbf{Step 2.} 
The aim here is to show that $\tilde \pi(\Upsilon^\eps(N'))=\ux_\eps^{\tilde D} \to \tilde \pi(\Upsilon(N'))=\ux^{\tilde D}$ as $\eps\to 0$. 
It is easy to see that
$$d_{\tilde D}(\ux_\eps^{\tilde D}, \ux^{\tilde D}) \le \sum_{i=1}^{\infty}\sum_{j\in \Z} U_{i}^{-\frac1\alpha}\|\xi(Q_{ij})\|\cdot \mathbbm{1}_{\{ U_{i}^{-\frac1\alpha}\|\xi(Q_{ij})\|\le \eps\}}$$
And this is shown to converge to zero almost surely because condition \eqref{eq:Q-condition->1} implies that  $W_i=\sum_{j\in \Z} \|\xi(Q_{ij})\|$ is a.s. finite. Moreover, as noted in \cite[Remark~4.6]{BPS18}, we have 
$$
\sup_{i\in\N}\sum_{j\in\Z} U_i^{-\frac1\alpha}\|\xi(Q_{i,j})\|\I_{U_i^{-\frac1\alpha}\|\xi(Q_{i,j})\|\leq \eps}\to 0, \qquad \text{a.s. as $\eps\to0$}.
$$ 
Note that condition \eqref{eq:Q-condition->1}, which is assumed by hypothesis for $1<\alpha<2$, holds for $0<\alpha<1$, as a byproduct of the convergence of the point processes stated in Theorem~\ref{thm:complete-convergence} 
as observed in \cite[Theorem~2.6]{DH95}. For $\alpha=1$, a similar argument holds by making use of assumption \eqref{eq:Q-condition-=1}.

\textbf{Step 3.} 
This step looks to compare the projection into $\tilde D$ of empirical process $\Upsilon(N_n')$ with its $\eps$-truncated version $\Upsilon(N_n')$. Namely, we set $\ux_{n,\eps}^{\tilde D}:=\tilde \pi(\Upsilon^\eps(N_n'))$, $\ux_n^{\tilde D}:= \tilde \pi(\Upsilon(N'_n))$ and observe that
$$d_{\tilde D}(\ux_{n,\eps}^{\tilde D}, \ux_n^{\tilde D}) \le \sum_{j=1}^{k_nr_n}\frac{\|\X_j\|}{a_n}\I_{\{\|\X_j\|\leq a_n \eps\}}.$$
Then, as in \cite[Proof of Theorem~4.5]{BPS18}, it follows by Markov's inequality and Karamata's Theorem that
\begin{align*}
\limsup_{n\to\infty}&\p(d_{\tilde D}(\ux_{n,\eps}^{\tilde D}, \ux_n^{\tilde D})>\delta)\leq \limsup_{n\to\infty} \frac{k_n r_n}{\delta a_n} \E(\|\X_1\|\I_{\{\|\X_1\|\leq a_n \eps\}})\\
&\leq \lim_{n\to\infty} \frac{n}{\delta a_n}\frac{\alpha a_n\eps\p(\|\X_1\|> a_n \eps)}{1-\alpha}=\frac{\alpha}{\delta(1-\alpha)}\eps^{1-\alpha}\xrightarrow[\eps\to0]{} 0,  
\end{align*}
since we are assuming $\alpha<1$.

\textbf{Step 4.}
Here we wish to compare the $\tilde D$ component of the embedding of $S_n$ into $F'$ with the respective $\tilde D$ projection of $\Upsilon(N_n')$.  But we need only notice that the representations of each of these elements is in fact identical so there is nothing to do: in the $\Upsilon(N_n')$ case all the jumps in a block are compressed into a single sequence, but the $\tilde D$ component undoes this compression, capturing all the jumps as an element of $D([0,1])$, which is equivalent to that obtained from $S_n$.

The proof of the $1\le \alpha<2$ case does not require further arguments, though we do need to employ condition \eqref{eq:small-jumps} in order to bound the error term in the equivalent to step 3 above. 
\end{proof}

\subsection{Proof of Theorem~\ref{thm:extremal-process}}
\label{subsec:proof-extremal}

We split the proof of Theorem~\ref{thm:extremal-process} into several steps given in a series of lemmas.

We start by defining a map $\Upsilon\colon \mathcal N_{\R_0^+\times\tilde l_\infty\setminus\{\tilde\infty\}}^\#\to F'((0,\infty), \R)$, which to each $\gamma\in \mathcal N_{\R_0^+\times\tilde l_\infty\setminus\{\tilde\infty\}}^\#$, such that $\gamma=\sum_{i=1}^\infty \delta_{(t_i, \tilde \x^i)}$, assigns
$$
\Upsilon(\gamma)=\left(h(\gamma),\{t_i\colon\,i\in\N\}, e_{\gamma}^{t_i}\right),
$$
where $h(\gamma)$ is defined for all $t>0$ by
\begin{equation}
\label{eq:running-min}
h(\gamma)(t):=\begin{cases}
\inf\{\|\tilde \Proj(\tilde \x^i)\|^{-1}_\infty\colon\,t_i\leq t\} &\mbox{if } t>\underline{t}\\
\overline{y} &\mbox{if }  t\leq\underline{t}
\end{cases}\end{equation}
with $\underline{t}=\inf\{(t_i)_{i=1}^{\infty}\}$, $\overline{y}=\sup\{\inf\{\|\tilde \Proj(\tilde \x^i)\|^{-1}_\infty\colon\,t_i\leq t\}\colon \,t>\underline{t}\}$ and, for $t\in[0,1]$
$$ 
e_{\gamma}^{t_i}(t)=\min\{h(\gamma)(t_i^-),\inf_{j\leq\left\lfloor\h(t) \right\rfloor} \|x^i_j\|\}.
$$
Note that since we are using $\Proj$ here, we are looking for our observations to go \emph{down} in value, so the objects here are all non-increasing functions.

Let  $\Lambda=\left\{\gamma=\sum_{i=1}^\infty \delta_{(t_i, \tilde \x^i)}\colon t_i\neq t_j,  \forall i\neq j; \gamma((0,a)\times \tilde l_\infty\setminus\{\tilde\infty\})>0, \forall a>0\right\}.$ 
\begin{lemma}
\label{lem:continuous-map}
The map $\Upsilon\colon \mathcal N_{\R_0^+\times\tilde l_\infty\setminus\{\tilde\infty\}}^\#\to F'(\R^+, \R)$ is continuous on $\Lambda$. 
\end{lemma}

\begin{proof}
Assume that $\gamma_n\rightarrow_{w^\#}\gamma=\sum_{i=1}^\infty \delta_{(t_i, \tilde \x^i)}\in\Lambda$. Let $0<a<b\in \{t_i\colon i\in\N\}^c$,  
we need to verify that
$$
m(\Gamma_{a,b,\ux^E_n},\Gamma_{a,b, \ux^E} )+d_{\tilde D}\left({\underline x_n}_{a,b}^{\tilde D},\underline x_{a,b}^{\tilde D}\right) \to 0,
$$
where $\Gamma_{a,b,\ux^E}$ and ${\underline x}_{a,b}^{\tilde D}$ denote the restrictions to the time interval $[a,b]$ of the respective objects.
Let $\alpha=h(\gamma)(a)\in (0,\infty)$ and  consider the bounded set $\tilde B_\alpha=\{\tilde \x\in \tilde l_\infty \colon \|\tilde \Proj(\tilde \x)\|^{-1}_\infty <\alpha \}$ of $\tilde l_\infty\setminus\{\tilde\infty\}$. There are finitely many $t_i$ such that $(t_i, \tilde \x^i)\in [a,b]\times \tilde B$. 
Let $(t_{i_1}, \tilde \x^{i_1}), \ldots, (t_{i_k}, \tilde \x^{i_k})$ be an enumeration of those points. Then, as in \cite[Proposition~3.13]{R87}, one can show that for $n$ sufficiently large there exists a shadow enumeration $\left(t^{(n)}_{i_1}, \tilde \x^{(n),i_1}\right), \ldots, \left(t^{(n)}_{i_k}, \tilde \x^{(n),i_k}\right)$ of the mass points of $\gamma_n$ in $[a,b]\times \tilde B_\alpha$ so that 
\begin{equation}
\label{eq:point-approx}
t^{(n)}_{i_j}\to t_{i_j}\quad\text{and}\quad\tilde \x^{(n),i_j}\to \tilde \x^{i_j},\quad \text{as $n\to\infty$, for all $j=1,\ldots,k$}
\end{equation}

Let $t_i$ be such that $h(\gamma)(a)=\|\tilde \Proj(\tilde \x^i)\|^{-1}_\infty$. Then applying the same argument to the bounded set $[t_i-\epsilon,a]\times B_{\alpha+\epsilon}$ where $\epsilon$ is chosen so that no mass point of $\gamma$ lies on the border of this set, one obtains that $h(\gamma_n)(a)\to h(\gamma)(a)$, as $n\to\infty$. It follows that 
\begin{equation}
\label{eq:approx-normal}
h(\gamma_n)(t)\xrightarrow[n\to\infty]{} h(\gamma)(t),\text{ $\forall t\in[a,b]\setminus\left(\bigcup_{n\in\N}\left\{t^{(n)}_{i_1},\ldots, t^{(n)}_{i_k}\right\}\cup\{t_{i_1}, \ldots,t_{i_k}\}\right)$}.
\end{equation}
The convergence stated in \eqref{eq:point-approx} regarding the quotient space $\tilde l_\infty\setminus\{\tilde\infty\}$ means that we can take representatives for which the convergence is obtained componentwise, which allows us to choose corresponding representatives on the space of excursions $\tilde D([0,1], \R)$ for which the following holds 
\begin{align}
&e_{\gamma_n}^{t^{(n)}_i}(t)\to e_{\gamma}^{t_i}(t)\quad\text{as $n\to\infty$, for all $t\in [0,1]$, which, in turn, implies}\label{eq:excurs-approx}\\
&\inf\left\{e_{\gamma_n}^{t^{(n)}_i}(t)\colon t\in[0,1]\right\}\xrightarrow[n\to\infty]{} \inf\{e_{\gamma}^{t_i}(t)\colon t\in[0,1]\},\label{eq:excurs-inf-approx}
\end{align}

From \eqref{eq:point-approx}, \eqref{eq:excurs-inf-approx} and \eqref{eq:approx-normal}, it follows that $m(\Gamma_{a,b,\ux^E_n},\Gamma_{a,b, \ux^E} )\xrightarrow[n\to\infty]{}0.$ In a similar way, \eqref{eq:point-approx}, \eqref{eq:excurs-approx} and \eqref{eq:approx-normal} imply that $d_{\tilde D}\left({\underline x_n}_{a,b}^{\tilde D},\underline x_{a,b}^{\tilde D}\right)\xrightarrow[n\to\infty]{}0$.
\end{proof}

\begin{lemma}
\label{lem:finite-time-approx}
We have that $d_{F',\infty}(Z_n, \Upsilon(N_n))\to0$ as $n\to\infty$, in probability, where $N_n$ is defined in \eqref{eq:point-process-def}
\end{lemma}
\begin{proof}
The processes $Z_n$ and $\Upsilon(N_n))$ have different internal clocks, but $h(\Upsilon(N_n)))(i/k_n)=Z_n(ir_n/n)$, for all $i\in\N_0$. Observe that  $h(\Upsilon(N_n)))$ is constant between $i/k_n$ and $(i+1)/k_n$, having a possible jump at $(i+1)/k_n$ if $\inf\{u_n^{-1}(\X_{j})\colon\;j=ir_n/n,\ldots, (i+1)r_n/n-1\}<h(\Upsilon(N_n)))(i/k_n)$. Note that when such a jump occurs, $Z_n$ is not necessarily constant between $ir_n/n$ and $(i+1)r_n/n$, namely, it may jump several times corresponding to the several moments $j=ir_n/n,\ldots,(i+1)r_n/n$ at which $u_n^{-1}(\X_{j})$ beats the running minimum until the threshold $\inf\{u_n^{-1}(\X_{j})\colon\;j=ir_n/n,\ldots, (i+1)r_n/n-1\}$ is reached. Nevertheless, these oscillations are recorded by the excursion decorating $\Upsilon(N_n)$ at the discontinuity point $(i+1)/k_n$.   

Note that by definition of $r_n$,
\begin{equation}
\label{eq:time-diff}
\frac i{k_n}\ge \frac{ir_n}{n}\ge  \frac i{k_n}-\frac i n\quad\Rightarrow\quad  \frac i n \ge \frac i{k_n}-\frac{ir_n}{n} \ge 0. 
\end{equation}

Hence, the maximum distance between one instant point of the time interval $[i/k_n,(i+1)/k_n)$ and another from $[ir_n/n,(i+1)r_n/n)$ is at most $\frac{i+1}{n}$.

Let $0<a<b\in\R$, $i^-=\min\{j\in\N_0\colon j/k_n\geq a\}$ and $i^+=\max\{j\in\N_0\colon jr_n/n<b\}$. Note that the range of the graphs $\Gamma_{i^-jr_n/n, i^+jr_n/n, Z_n^E}$ and $\Gamma_{i^-jr_n/n, i^+jr_n/n, \Upsilon(N_n)^E}$ is exactly the same, which means that the Hausdorff distance between the graphs is given by the time component, \ie
$$
m\left(\Gamma_{i^-jr_n/n, i^+jr_n/n, Z_n^E},\Gamma_{i^-jr_n/n, i^+jr_n/n, \Upsilon(N_n)^E}\right)\leq \frac{i+1}{n}\xrightarrow[n\to\infty]{}0.
$$
We only need to worry with the observations corresponding to $j\in\N_0$ such that $a\leq j/n<i^-/k_n$ where the range of the process $Z_n$ may differ from the value of $h(\Upsilon(N_n))(a)=:\tau$, by either being above or below. Note that there are at most $k_n$ such observations. This means that for $|\epsilon|<\tau$, we have that 
$m\left(\Gamma_{a, i^-/k_n, Z_n^E},\Gamma_{a, i^-/k_n, \Upsilon(N_n)^E}\right)>|\epsilon|$ implies that there must be at least one of these possible $k_n$ observations such that $\|\X_j\|>u_n(\tau+\epsilon)$. Since $k_n\p( \|\X_j\|>u_n(\tau+\epsilon))\to 0$, we obtain that $m\left(\Gamma_{a, i^-/k_n, Z_n^E},\Gamma_{a, i^-/k_n, \Upsilon(N_n)^E}\right)$ goes to $0$ in probability. A similar argument applies for the time frame between $i^+jr_n/n$ and $b$. Hence, $m\left(\Gamma_{a, b, Z_n^E},\Gamma_{a, b, \Upsilon(N_n)^E}\right)\xrightarrow[n\to\infty]{}0$, in probability. This takes care of the $E$ component of the distance in $F'$ and we need now to conclude that the same applies to the projections into $\tilde D$.

Let $i^-<i\leq i^+$ and assume that there is a jump of $h(\Upsilon(N_n))$ at $i/k_n$. Recall that while $h(\Upsilon(N_n)))$ is constant between $(i-1)/k_n$ and $i/k_n$, the process $Z_n$ may have oscillations in the corresponding interval $[(i-1)r_n/n, ir_n/n]$. But, by construction, the excursion $e_{\Upsilon(N_n)}^{i/k_n}$ exactly mimics these oscillations of $Z_n$. Observe now that the projection $\Upsilon(N_n)^{\tilde D}$ reconstructs the behaviour of $Z_n$ by incorporating in the time frame the excursions. Then, since the metric $\tilde D$ allows for time deformations, it is clear that the distance between the projections $\Upsilon(N_n)^{\tilde D}$ and $Z_n^{\tilde D}$, when restricted to $[i^-r_n/n, ir_n/n]$ is actually equal to $0$. Again, we are left to analyse the time periods $[a,i^-r_n/n)$ and $[i^+r_n/n,b]$, where some pieces of information may be missing. However, as observed with the projections into $E$, the missing information corresponds to at most $\max\{k_n,r_n\}$ random variables and the probability of these producing oscillations that will not be mimicked is bounded by $\max\{k_n,r_n\}\p( \|\X_j\|>u_n(\tau+\epsilon))\to 0$. Again, we conclude that $d_{\tilde D}\left({\Upsilon(N_n)}_{a,b}^{\tilde D},{Z_n}_{a,b}^{\tilde D}\right) \xrightarrow[n\to\infty]{} 0$, in probability.
\end{proof}
\begin{proof}[Proof of Theorem~\ref{thm:extremal-process}]
By Lemma~\ref{lem:continuous-map}, we may apply the CMT and conclude from Theorem~\ref{thm:point-process-convergence} that $\Upsilon(N_n)$ converges weakly to $\Upsilon(N)$, in $F'((0,+\infty))$. By Lemma~\ref{lem:finite-time-approx} and a Slutsky argument we conclude that $Z_n$ converges weakly to $\Upsilon(N)$, in $F'((0,+\infty))$. Therefore, we are only left to show that $h(\Upsilon(N))$ has the prescribed finite-dimensional distributions of $Z_H$.

For each $\alpha>0$, let $\tilde B_\alpha=\{\tilde \x\in \colon \|\tilde \Proj(\tilde \x)\|^{-1}_\infty <\alpha \}\subset\tilde l_\infty\setminus\{\tilde\infty\}$. For the unidimensional distribution, with $t,y\geq 0$ and $\bar H(\tau)=\e^{-\theta\tau}$, we have
\[\p(h(\Upsilon(N))(t)\geq y)=\p( N([0,t]\times\tilde B_y)=0)=\e^{-\theta t y}=\bar H^t(y).\]

For the bidimensional distribution, with $0\leq t_1<t_2$ and $y_1\geq y_2\geq 0$,
\begin{align*}
\p&\left(h(\Upsilon(N))(t_1)\geq y_1,h(\Upsilon(N))(t_2)\geq y_2\right)=\\
&=\p\left(N([0,t_1]\times B_{y_1})=0,N((t_1,t_2]\times B_{y_2})=0\right)=\bar H^{t_1}(y_1)\bar H^{t_2-t_1}(y_2).
\end{align*}
In case $0\leq y_1<y_2$,
\[\p\left(h(\Upsilon(N))(t_1)\geq y_1,h(\Upsilon(N))(t_2)\geq y_2)=\p(h(\Upsilon(N))(t_2)\geq y_2\right)=\bar H^{t_2}(y_2),\]
so in general
\[\p(h(\Upsilon(N))(t_1)\geq y_1,h(\Upsilon(N))(t_2)\geq y_2)=\bar H^{t_1}(y_1\vee y_2)\bar H^{t_2-t_1}(y_2).\]

By induction we get for the $k$-dimensional distribution
\begin{align*}
\p(h(\Upsilon(N))(t_1)&\geq y_1,\Upsilon(N))(t_2)\geq y_2\ldots, h(\Upsilon(N))(t_k)\geq y_k)=\\
&=\bar H^{t_1}\!\left(\bigvee_{i=1}^k\{y_i\}\!\right)\!\bar H^{t_2-t_1}\!\left(\bigvee_{i=2}^k\{y_i\}\!\right)\!\cdots\!\bar H^{t_k-t_{k-1}}(y_k)\\
&=\p(Z_H(t_1)\geq y_1,\ldots,Z_H(t_k)\geq y_k).
\end{align*}
Regarding the statements about the Markov structure of the process $Z_H$, we give brief indications how to prove them and refer to \cite[Chapter~4.1]{R87} for some useful properties of the theory of Markov processes. 
The fact that $Z_H$ is Markov with the given transition probability follows from the form of the finite-dimensional distributions. Namely, for $0<y<z$,
\begin{align}
\p(Z_H(t+s)\geq y \mid Z_H(s)=z)&=\p(h(\Upsilon(N))(t+s))\geq y \mid h(\Upsilon(N))(s))=z)\nonumber\\
&=\p(N((s,s+t]\times B_{y})=0)=\e^{-\theta t y}.\label{eq:transition}
\end{align}
The exponential holding time parameter at state $z$, which is denoted by $\lambda(z)$, can be derived from
$$
\p(Z_H(t+s)=Z_H(s) \mid Z_H(s)=z)=\p(N((s,s+t]\times B_{z})=0)=\e^{-\theta t z}.
$$
Thus, $\lambda(z)=\theta z.$ The jump distribution $\Pi(z,[0,y))$ can be computed from 
$$
\lim_{t\to0}t^{-1}\p(Z_H(t+s)<y) \mid Z_H(s)=z)=\lambda(z)\Pi(z,[0,y)).
$$
But from \eqref{eq:transition}, for $y<z$, we have $t^{-1}\p(Z_H(t+s)<y) \mid Z_H(s)=z)=t^{-1}(1-\e^{-\theta t y})\xrightarrow[t\to0]{}\theta y$ and therefore $\Pi(z,[0,y))=y/z$.
\end{proof}

\subsection{Proof of Theorem~\ref{thm:record-theorem}}\label{subsec:proof-records}

In order to count the number of records during the clusters we define, for $y\geq 0$ and $\x=(x_j)_j\in l_\infty$, 
\begin{equation}
R^\x(y)=\sum_{j\in\Z}\I_{\{\|x_j\|<(y\wedge \inf _{i<j} \|x_i\|)\}},
\label{eq:R^x}
\end{equation}
which gives the number of record asymptotic frequencies corresponding to the smallest observations in $\x$ that have beaten the  frequency $y$, which must be finite because $\lim_{|j|\to\infty}\|x_j\|=\infty$. For $\gamma=\sum_{i=1}^\infty \delta_{(t_i, \tilde \x^i)}\in \mathcal N_{\R_0^+\times\tilde l_\infty\setminus\{\tilde\infty\}}^\#$, let
$h(\gamma)\in D((0,\infty),\R)$ be defined as in \eqref{eq:running-min} and define the record point process in $ \mathcal N_{(0,+\infty)}^\#$ by
$$
R_\gamma=\sum_{i\in\N} \delta_{t_i} R^{\x^i}(h(\gamma)(t_i^-)), \quad\text{where $\x^i\in l_\infty$ is any chosen representative of $\tilde \x^i$.}
$$
In order to be able to relate to the actual count of the number of records at finite time we consider the empirical process:
$$
N''_n:=\sum_{i=0}^\infty \delta_{\left(\frac i n, u_n^{-1}(\|\X_{i}\|)\frac{\X_{i}}{\|\X_{i}\|}\right)},
$$
which we consider as defined in $\mathcal N_{\R_0^+\times\tilde l_\infty\setminus\{\tilde\infty\}}^\#$ by embedding the second coordinate in $\R^d$ into $l_\infty$ by adding a sequence of $\infty$ before and after that entrance as in \eqref{eq:embedding}. Now, observe that $R_{N''_n}$ coincides with $\mathfrak R_n$ given in \eqref{eq:Record-n-def} and indeed counts the number of records of the process $\X_0,\X_1,\ldots$

Consider the subset of $ \mathcal N_{\R_0^+\times\tilde l_\infty\setminus\{\tilde\infty\}}^\#$ defined by 
$$\Lambda=\left\{\gamma=\sum_{i=1}^\infty \delta_{(t_i, \tilde \x^i)}\colon \; \text{$t_j\neq t_\ell$ and $\|x^i_j\|=\|x^i_\ell\| \Rightarrow j= \ell$ or $\|x^i_j\|=\infty$}\right\}.$$
Since $(0,\infty)$ is locally compact and separable, then  the weak$^\#$ topology in $\mathcal N_{(0,+\infty)}^\#$ coincides with the vague topology and by a trivial adjustment of \cite[Lemma~5.1]{BPS18} one obtains that the map $\gamma\to R_\gamma$ from $\mathcal N_{\R_0^+\times\tilde l_\infty\setminus\{\tilde\infty\}}^\#$ to $\mathcal N_{(0,+\infty)}^\#$ is continuous at every $\gamma\in \Lambda$.

\begin{proposition}
Let $\X_0, \X_1, \ldots$ be a stationary $\R^d$-valued sequence, with tail index $\alpha>0$, for which the point process $N_n$ defined in \eqref{eq:point-process-def} converges weakly in $\mathcal N_{\R_0^+\times\tilde l_\infty\setminus\{\tilde\infty\}}^\#$ to the Poisson point process $N$ given by \eqref{eq:limit-process}. Then, under the same assumptions of Theorem~\ref{thm:record-theorem}, then the conclusion regarding the convergence of $R_{N''_n}$ to $R_{N}$ holds, in  $\mathcal N_{(0,+\infty)}^\#$, and the limit process has the representation given there.
\end{proposition}

Together with Theorem~\ref{thm:complete-convergence}, this concludes the proof of Theorem~\ref{thm:record-theorem}.

We elaborate a bit further on the representation of the limit record point process, which in this case is given in a more natural way through the use of the variable $U$ with a uniform distribution. Namely, the limiting process $R_{N}$ can be written as $\mathfrak R$, in \eqref{eq:Record-def}, where the $\kappa_i$ have the same distribution as the integer-valued random variable 
   $R^{\mathbf Q} (U^{-1})$, defined as in \eqref{eq:R^x}, where $U$ is a uniformly distributed random variable independent of $\mathbf Q=(Q_j)_j$. The rest of the proof of this proposition follows from trivial adjustments to the proof of \cite[Theorem~5.2]{BPS18}.

Let $M=\sum_{i=1}^\infty \delta_{(T_i,U_i)}$ be the Poisson point process on $\R_0^+\times \R_0^+$ defining the point process $N$ in \eqref{eq:limit-process}. Recall that $\|Q_j\|\geq 1$ for all $j$, but there exists one index $j$ for which the equality holds. Let $(\tau_n)_{n\in\Z}$ denote the ordered sequence of record times of $M$ counted by $R_M$. For definiteness, fix a certain $s>0$ and assume that $\tau_1$ is the first record time larger than $s$, \ie $\tau_1=\inf\{\tau_j\colon \tau_j>s\}$, and now denote the respective record frequencies by $Y_n=\inf_{T_i\leq \tau_n} U_i$, with $Y_0=\inf_{T_i\leq s} U_i$. The fact that $\sum_{n\in\Z}\delta_{\tau_n}$ is a Poisson process with intensity $x^{-1}\,dx$ follows for example from \cite[Theorem~5.7]{FFM20}. Note that for all $n$ we have $Y_n/Y_{n-1}<1$ and, in fact, from the jump distribution $\Pi$ computed in Theorem~\ref{thm:extremal-process} we obtain that $(Y_n/Y_{n-1})_{n\in \N}$ is a sequence of independent and uniformly distributed random variables. 

The number of records of $R_{N}$ observed at $\tau_n=T_i$, say, corresponds to the number of $j$'s for which $Y_n\|Q_{i,j}\|<(Y_{n-1}\wedge \inf_{\ell<j}\{Y_n\|Q_{i,\ell}\|\})\Leftrightarrow \|Q_{i,j}\|<(Y_{n-1}/Y_n\wedge \inf_{\ell<j}\{\|Q_{i,\ell}\|\})$, \ie is equal to $R^{\mathbf Q_i}(Y_{n-1}/Y_n)$, where $\mathbf Q_i=(Q_{i,j})_j$ is independent of the sequence $(Y_n/Y_{n-1})_{n\in \N}$ for all $i$. Since $s$ was arbitrary, the given representation of $R_N$ holds.  

\subsection{Point processes encompassing inducing}
\label{subsec:inducing}
Inducing is a very powerful technique used to study statistical properties of non-uniformly hyperbolic systems (see \cite{A20}, for example). 
Let $T\colon\mathcal X\to\mathcal X$ be a dynamical system defined on a Lebesgue space $(\mathcal X, \mathcal B_\mathcal X, m)$. Let $\Delta_0\in \mathcal B_\mathcal X$ be a base set such that $m(\Delta_0)>0$. Let $\mathcal P$ be a countable$\mod0$ partition of $\Delta_0$, such that, for all $\omega\in \mathcal P$, $T|\omega$ is invertible and there exists a return time function $R:\Delta_0\to\N$, constant in each element of $\mathcal P$ and such that $T^{R(\omega)}(\omega)$ is a union elements of $\mathcal P$. We define the induced system $\bar T=T^R\colon \Delta_0\to\Delta_0$ by $\bar T(x)=T^{R(x)}(x)$. It is well known that if $\mu_{0}$ is a $\bar T$-invariant probability measure, then, assuming that $R\in L^1(\mu_0)$ (which we do throughout), its saturation defined by $\mu (A) =\sum_{j=0}^\infty \mu_{0}(\{R>j\}\cap T^{-j}(A))$ is $T$-invariant. Note that $R$ is not necessarily the first return to $\Delta_0$. We consider a measurable observable function $\varphi:\mathcal X\to\R$ and consider the potentials $\bar\varphi, \bar\Phi:\Delta_0\to\R$ defined by 
$$
\bar\varphi(x)=\max_{j=0,\ldots, R(x)-1}|\varphi(T^j(x))|\qquad \bar\Phi(x)=\sum_{j=0,\ldots, R(x)-1}|\varphi(T^j(x))|.
$$
We start by considering the induced stochastic process $X_0, X_1,\ldots$ such that $X_j=\bar\varphi\circ \bar T^j$. Assume the existence of normalising levels $(u_n(\tau))_{n\in\N}$ as in \eqref{un}, \ie such that $\lim_{n\to\infty}n\mu_0(\bar\varphi>u_n(\tau))=\tau$, for all $\tau>0$. Then, for each $n\in\N$, define $F_n:\mathcal X\to \R$ as
$$
F_n(x)=u_n^{-1}(|\varphi(x)|)\frac{\varphi(x)}{|\varphi(x)|}.
$$
Let $d=1$ and define $l_\infty, l_0, \tilde l_\infty, \tilde l_0$ as above. For $\x\in\tilde l_0$, define $|\tilde \x|'_\infty :=\|\x\|_\infty$, where $\x\in l_0$ is such that $\tilde \x=\tilde \pi(\x)$. For $\tilde \x\in\tilde l_\infty$, set $|\tilde \x|_\infty:=|\tilde P(\tilde\x)|'_\infty$. Define the product space
\begin{equation}
{l_\infty}^*=\left\{ \mathbf{\underline x}=(\tilde{\mathbf x}_j)_j\in {\tilde l_\infty}^{\Z}\colon\;\lim_{|j|\to\infty}|\tilde \x_j|_\infty=0\right\}\end{equation}
For 
$\mathbf{\underline x} \in {l_\infty}^*$, we define 
$$
|\mathbf{\underline x}|^*:=\sup_{j\in\Z}|\tilde{\mathbf x}_j|_\infty.
$$
and the complete metric defined on  $l_\infty^*\setminus \{\underline \infty\}$ ($\underline \infty$ denotes the only element of ${l_\infty}^*$ such that $|\underline \infty|^*=0$) given by:
$$
{d}^*(\mathbf{\underline x},\mathbf{\underline y})=\left(\sup_{j\in\Z}\tilde d(\tilde{\textbf x}_j,\tilde{\textbf y}_j)\wedge 1\right)\vee\left|\frac1{|\mathbf{\underline x}|^*}-\frac1{|\mathbf{\underline y}|^*}\right|.
$$ 
We define the quotient spaces $\tilde l_\infty^*$ 
and the respective metric $\tilde {d}^*$, 
 accordingly. We embed the finite product space ${\tilde l_\infty}^{n}$ into the infinite product space ${l_\infty}^*$ simply by adding a sequence of $\tilde \infty$ before and after the $n$ entrances of the elements of ${\tilde l_\infty}^{n}$ and write $\tilde \pi^*((\tilde \x_1,\ldots,\tilde \x_n))\in \tilde l_\infty^*$ for the projection of the natural embedding of $(\tilde \x_1,\ldots,\tilde \x_n)$ into $l_\infty^*$ to the quotient space $\tilde l_\infty^*$. 
 
We define now the new point process $N_n^*$ as a random element of $\mathcal N_{\R_0^+\times\tilde l_\infty^*\setminus \{\underline \infty\}}^\#$ of boundedly finite point measures on $\R_0^+\times\tilde l_\infty^*\setminus \{\underline \infty\}$. 
For $j\in\N_0$, we consider
$$
\tilde{\mathbf X}_j=\tilde\pi\left(F_n(\bar T^{j}(x))),F_n(T(\bar T^{j}(x))), 
\ldots ,F_n(T^{R(\bar T^{j}(x))-1}(\bar T^{j}(x)))\right)\in {\tilde l_\infty}
$$
and define
\begin{equation}
\label{eq:point-process-induced}
N_n^*=\sum_{i=1}^\infty\delta_{\left(i/k_n,\tilde\pi^*\left(\tilde{\mathbf X}_{(i-1)r_n},\ldots,\tilde{\mathbf X}_{ir_n-1} \right)\right)}.
\end{equation}
Note that these point processes have in the time direction the information gathered by the induced dynamics $\bar T$ (as with $N_n$ defined in \eqref{eq:point-process-def}) and, at each recorded time event, $\tilde{\mathbf X}_j$ includes the information regarding the excursion performed during the induced time by the original dynamics $T$.
In order to obtain the convergence of these point processes, we need to adapt the previous conditions and definitions, which assumed that $\V=\R^d$ to the present situation where $\V=\tilde l_\infty$. In particular, for the definition of the transformed anchored tail process, we assume the existence of a process $(\tilde {\mathbf Y}_j)_{j\in\Z}\in l_\infty^*$ satisfying the following assumptions: 
$$\mathcal L\left(\frac1\tau \left(\tilde{\mathbf X}_{r_n+s},\ldots,\tilde{\mathbf X}_{r_n+t} \right)\;\middle\vert\; |\tilde \X_{r_n}|_{\infty}>\tau^{-1} \right)\xrightarrow[n\to\infty]{}\mathcal L\left((\tilde {\mathbf Y}_j)_{j=s,\dots,t}\right),$$ 
for all $s<t\in\Z$ and all $\tau>0$. Note that 
\begin{align*}
&|\tilde \X_{r_n}|_{\infty}>\tau^{-1}  
\Leftrightarrow  \max_{j=0,\ldots, R(\bar T^{r_n}(x))-1}1/u_n^{-1}(|\varphi(T^j(\bar T^{r_n}(x)))|)>\tau^{-1} \\
& \Leftrightarrow  \min_{j=0,\ldots, R(\bar T^{r_n}(x))-1}u_n^{-1}(|\varphi(T^j(\bar T^{r_n}(x)))|)<\tau
\Leftrightarrow \bar\varphi(\bar T^{r_n}(x))>u_n(\tau).
\end{align*}
The transformed anchored tail process is then defined by
\begin{equation}
\label{eq:piling-induced}
\mathcal L\left((\tilde{\textbf{Z}}_j)_{j\in\Z} \right)=\mathcal L\left((\tilde{\textbf Y}_j)_{j\in\Z}\;\middle\vert\; \sup_{j\leq -1}|\tilde{\textbf Y}_j|_\infty\leq 1\right),
\end{equation}
whose spectral decomposition is given by
\begin{equation}
\label{eq:Z-polar-special}
L_Z= \left(\sup_{j
\in\Z}|\tilde{\mathbf Z}_j|_\infty\right)^{-1}, \qquad \tilde{\textbf Q}_j=\frac{\tilde{\mathbf Z}_j}{L_Z}, \quad\text{for all}\quad j\in\Z.
\end{equation}
Let $\underline{\textbf Q}=(\tilde{\textbf Q}_j)_{j\in\Z}$, which can be seen as a random element in $\mathbb S^*=\{\underline \x\in \tilde l_\infty^*\colon |\underline \x|^*=1\}$ (and thus of the form $(\textbf Q_{j, \kappa})_{j\in \Z, \kappa\in \Z}$).
Observe also that conditions $\D_{q_n}$, $\D'_{q_n}$ also need to be adapted. Namely, the sets $H_j\in \mathcal F_\V$ in the definition of $\mathscr F$ in \eqref{eq:A-field} have to be taken as sets of the form $\tilde H_j$ as defined in \eqref{eq:A-tilde}.
\begin{proposition}
\label{prop:point-process-convergence-induced}
Let $T:\mathcal X \to \mathcal X$ be a dynamical system  and consider an observable function $\varphi\colon\mathcal X\to\R$. We assume that $T$  admits an induced system $\bar T:\Delta_0\to\Delta_0$ as described above and such that $\bar T$ is uniformly expanding. We assume that 
the adapted conditions $\D_{q_n}$, $\D'_{q_n}$ are satisfied and, moreover, the transformed anchored tail process given in \eqref{eq:piling-induced} is well defined. Consider the $N_n^*$ defined as in \eqref{eq:point-process-induced}, then $N_n^*$ converges weakly in $\mathcal N_{\R_0^+\times\tilde l_\infty^*\setminus \{\underline \infty\}}^\#$ to the Poisson point process $N^*$, which can be written as:
\begin{equation}
\label{eq:limit-process-induced}
N^*=\sum_{i=1}^\infty \delta_{(T_{i}, U_{i}\underline{\textbf Q}_i)},
\end{equation} 
where $(T_{i})_{i\in\N}$ and  $(U_{i})_{i\in\N}$ are such that $\sum_{i=1}^\infty\delta_{(T_i,U_i)}$ is a bidimensional  Poisson point process on $\R_0^+\times \R_0^+$ with intensity measure $\leb\times\theta\,\leb$ and  $(\underline{\textbf Q}_i)_{i\in\N}$ is an \iid  sequence of random elements in $\mathbb S^*$ such that each $\underline{\textbf Q}_i$ has a distribution given by \eqref{eq:Z-polar-special}. All the sequences $(T_{i})_{i\in\N}$, $(U_{i})_{i\in\N}$ and $(\underline{\textbf Q}_i)_{i\in\N}$ are mutually independent. 

\end{proposition}

The proof of this result follows as in the proof of Theorem~\ref{thm:point-process-convergence} with the necessary straightforward adjustments. 

\subsubsection{Functional limit theorem with clustering through inducing}

The point process convergence given in Proposition~\ref{prop:point-process-convergence-induced} allows to easily obtain the conclusions of Theorems~\ref{thm:extremal-process} and \ref{thm:record-theorem} as they were obtained in Sections~\ref{subsec:proof-extremal} and \ref{subsec:proof-records} from the convergence of the point processes stated in Theorem~\ref{thm:point-process-convergence}. 

However, in order to obtain the conclusions of Theorem~\ref{thm:sums-heavy-tails} from the convergence of the point processes, as we did in Proposition~\ref{prop:convPP=>convSum}, one needs to add an extra condition to ensure that projection of the information registered during the inducing period does not pile up to create discontinuities in the sum. Essentially, we need to forbid the accumulation of many very small contributions which add up to have a significant impact on the sums: the sums should be mostly influenced by heavy tailed observations corresponding to entrances of the orbit near the set $\mathcal M$, where $\varphi$ is maximised.  Namely, we assume that, for every $j\in\Z$,
\begin{equation}
\label{eq:condition-induced}
\mathcal L\left(\tfrac{\bar\Phi\circ \bar T^{r_n+j}}{\tau^{-1/\alpha}a_n} \middle\vert \left\{\bar\varphi\circ\bar T^{r_n}>u_n(\tau)\right\},\bigcap_{i=1}^{q_n} \left\{\bar\varphi\circ \bar T^{r_n-i}\leq u_n(\tau)\right\}\right)
\to\mathcal L\left( \Sigma\left(\tilde \Xi(\tilde{\mathbf{Z}}_j)\right)\right)
\end{equation}
where $a_n$ is as in \eqref{eq:heavy-normalisation} and the function $\Sigma\colon \tilde l_0\to\R$, to each $\x\in\tilde l_0$, assigns $\Sigma(\tilde \x):=\sum_{j\in\Z}x_j$, where $\x=(\ldots,x_{-1},x_0,x_1,\ldots)\in l_0$ is such that $\tilde \x=\tilde \pi(\x)$.
This condition guarantees some sort of tightness so that the aggregate effect of the excursions performed during the induced periods is completely captured by the sum of the entrances of each component, $\tilde{\mathbf{Z}}_j$, of the transformed anchored tail process.

In order to illustrate the adjustments needed to be performed to obtain the conclusions of Theorems~\ref{thm:sums-heavy-tails}, \ref{thm:extremal-process} and \ref{thm:record-theorem} from the point process convergence stated in Proposition~\ref{prop:point-process-convergence-induced}, we will consider the most complicated case regarding the functional limit for sums.   The excursions here have two levels: one corresponding to the clustering observed for the induced system (see the middle term of \eqref{eq:excursion-induced} below) and the digression performed during the induced period (see the last term of \eqref{eq:excursion-induced}): our time parametrisation $\h(t)$ thus runs through the contributions from the induced observable, which are added when $\h(t)$ is an integer, and the excursions during the inducing time, which are added when $\h(t)$ is between integers.

\begin{theorem}
\label{prop:convPP=>convSum-induced}
Let $T:\mathcal X \to \mathcal X$ be a dynamical system  as in Proposition~\ref{prop:point-process-convergence-induced} such that $N_n^*$,  defined as in \eqref{eq:point-process-induced}, converges weakly in $\mathcal N_{\R_0^+\times\tilde l_\infty^*\setminus \{\underline \infty\}}^\#$ to the Poisson point process $N^*$, described in \eqref{eq:limit-process-induced}. Let $\mathcal M\subset \mathcal X$ be such that $\mu(\mathcal M)=0$ and $\varphi\colon\mathcal X\to\R$ be such that $\varphi(x)=g(\dist(x,\mathcal M))$, where $g$ is of type $g_2$ and condition \eqref{eq:heavy-normalisation} holds, with $0<\alpha<1$. Assume also that condition \eqref{eq:condition-induced} holds. Then $S_n$ defined in \eqref{eq:sum} converges, in $F'$, to $\underline V:=(V, disc(V),\{e_{V}^s\}_{s\in disc(V)})$, where $V$ is an $\alpha$-stable L\'evy process on $[0,1]$ which can be written as
\begin{align*}
V(t)&=\sum_{T_i\leq \mu(\Delta_0) t}\sum_{j\in\Z} U_{i}^{-\frac1\alpha}\Sigma(\tilde{\mathcal{Q}}_{i,j})
\end{align*}
and the excursions  can be represented for $t\in[0,1]$ by
\begin{equation}
\label{eq:excursion-induced}
e_{V}^{T_i}(t)=V(T_i^-)+U_{i}^{-\frac1\alpha}\sum_{j<\left\lfloor \h(t)\right\rfloor}\Sigma(\tilde{\mathcal{Q}}_{i,j})+U_{i}^{-\frac1\alpha}\sum_{\kappa\leq\left\lfloor \h\left(\h(t)-\left \lfloor \h(t)\right\rfloor \right)\right\rfloor}{\mathcal{Q}}_{i,\left\lfloor \h(t)\right\rfloor,\kappa}, 
\end{equation}
where $({\mathcal{Q}}_{i,j,\kappa})_{i\in\N, j\in\Z,\kappa\in\Z}$ is a representative of  $\tilde{\mathcal{Q}}_{i,j}=\tilde \Xi (\tilde Q_{i,j})$ and $(\tilde Q_{i,j})_{i\in\N, j\in\Z}$, $(T_i)_{i\in\N}$, $(U_i)_{i\in\N}$ are as in Proposition~\ref{prop:point-process-convergence-induced}.
\end{theorem} 

\begin{remark}
This result can be generalised to include the cases $1\leq \alpha\leq 2$ and $d>1$ as we did earlier.
\end{remark}

\begin{proof}
We start by defining a projection  $\Upsilon\colon\mathcal N_{\R_0^+\times\tilde l_\infty^*\setminus \{\underline \infty\}}^\#\to F'([0,1], \R^d)$. Suppose we are given $\gamma=\sum_{i=1}^\infty \delta_{(t_i, \underline\x_i)}\in \mathcal N_{\R_0^+\times\tilde l_\infty^*\setminus \{\underline \infty\}}^\#$.  At time $t_i$ we have $\underline\x_i \in \tilde l_\infty^*$: let $(\ldots, \tilde \x_{i,-1},\tilde \x_{i,0},\tilde \x_{i,1},  \ldots)$ be a representative of this in $l_\infty^*$ and for each $i,j\in\Z$ let  $(\ldots, x_{i,j,-1},x_{i,j,0},x_{i,j,1},  \ldots)$ be a representative of $\tilde \x_{i,j}$ in $l_\infty$. Define 
$$\mathring{e}_x^{t_i}(t)=\sum_{j=-\infty}^{\left\lfloor \h(t)\right\rfloor-1}{\Sigma(\tilde\Xi(\tilde\x_{i,j}))}+\sum_{\kappa\leq\left\lfloor \h\left((\h(t)-\left \lfloor \h(t)\right\rfloor)-\frac12\right)\right\rfloor}{\xi(x_{i,\left\lfloor h(t)\right\rfloor,\kappa})}, \; x(t)=\sum_{t_i\leq t}\mathring{e}_x^{t_i}(1)$$
and $S^x=\{t_i\}_i.$ Finally, set $e_x^{t_i}=x(t_i^-)+\mathring{e}_x^{t_i}$ and let $\Upsilon(\gamma)=(x,S^x, \{e_x^{t_i}\}_{i})$.  
Following the same steps as in the proof of Proposition~\ref{prop:convPP=>convSum} and using condition \eqref{eq:condition-induced} to guarantee non-degeneracy of the limits of the sums calculated during the induced excursions we obtain that 
$\Upsilon(N_n^*)=\underline{\bar V}_n=(\bar V_n, \text{disc}(\bar V_n), \{e^s_{\bar V_n}\})$ converges in $F'$ to $\Upsilon(N^*)=\underline{\bar V}=(\bar V, \text{disc}(\bar V), \{e^s_{\bar V}\})$, where
$$
\bar V(t)=\sum_{T_i\leq t}\sum_{j\in\Z} U_{i}^{-\frac1\alpha}\Sigma(\tilde{\mathcal{Q}}_{i,j}),
$$
and the excursions (defined in $\tilde D$, which allows time deformation) are as in \eqref{eq:excursion-induced}.

We define the $n$-the return time $R_n\colon\Delta_0\to\N$ by $R_n(x)=\sum_{j=0}^{n-1}R\circ \bar T^j(x)$ and the occupation times $N_n\colon \Delta_0\to\N$ by $N_n(x)=\max\{j\in\N_0\colon \;R_j\leq n\}$. By the ergodic theorem, we have that $R_n/n\to\int R \;d\mu_0=\mu(\Delta_0)^{-1}$, a.e. As a consequence, we obtain a strong law for the renewal process, which implies that $N_n/n\to\mu(\Delta_0)$, a.e.  and, in fact, we have $\sup_{t\in[0,1]}|N_{\lfloor tn\rfloor}/n-t\mu(\Delta_0)|\xrightarrow[n\to\infty]{} 0$, a.e. 

We produce a time change to the process $\underline{\bar V}_n$ in order to approximate $S_n$. Namely, we consider the sequence  random elements of $F'$ denoted by $\underline V_n=(V_n, \text{disc}(V_n), \{e^s_{V_n}\})$, where $V_n(t)=\bar V_n(\mu(\Delta_0) t)$, $s\in \text{disc}(V_n)$ iff $\mu(\Delta_0)s\in\text{disc}(\bar V_n)$ and $e^s_{V_n}=e^{\mu(\Delta_0)s}_{\bar V_n}$. Note that clearly, $\underline V_n$ converges in $F'$ to $\underline V$ given in the statement of the proposition.

By a Slutsky argument, the conclusion follows once we prove that 
$$
d_{F'}(S_n, \underline V_n)\xrightarrow[n\to\infty]{}0\quad \text{in probability}.
$$
Note that the excursions of $\underline V_n$ keep track of all the oscillations of $S_n$, except for a possible discrepancy near $t=1$, where there may be a lack or excess of data. Away from $t=1$, for $n$ large enough, since $\tilde D$ allows for time deformation, the distance in $F'$ comes from the projection into the $E$ component. But since the excursions keep all the information, then the range of values in vertical direction will coincide and the distance between the graphs will result from the time component. The time correction needed to align the graphs comes from two sources. One is deterministic and results from the fact that the clock of $S_n$ moves at steps $1/n$ while the clock of $\underline{\bar V_n}$ has a step size of $1/k_n$. This means we need to match intervals $[ir_n/n,(i+1)r_n/n)$ and $[i/k_n,(i+1)/k_n)$, whose time gap is bounded by $2k_n/n$. In fact, $ir_n/n\leq i/k_n$ and $i/k_n-ir_n/n\leq i/k_n\left(1-\lfloor n/k_n\rfloor/(n/k_n)\right)\leq k_n/n.$ The second is random and depends on the distance between $N_{\lfloor nt\rfloor}/n$ and $\mu(\Delta_0)t$, which converges uniformly in $t$ to 0, $\mu$-a.e. Near $t=1$, we have to consider also possible mismatches resulting from the fact that $n-k_nr_n=\O(1/k_n)$, which means that, besides the time contributions for $d_{F'}$ already considered earlier, we need to consider differences in the vertical range that are bounded by $\sum_{j=k_nr_n+1}^n \frac{\bar\Phi(\bar T^{j}(x))}{a_n}$, which converges to $0$ in probability. 
\end{proof}

\subsubsection{Application to the  Manneville-Pomeau case}
We apply these tools to the LSV maps $T_\gamma$ defined in \eqref{eq:LSV}. Recall that, for $\gamma\in(0,1)$, the map $T_\gamma$ has an invariant measure $\mu_\gamma$ absolutely continuous with respect to Lebesgue such that its density $h_\gamma=\frac{d\mu_\gamma}{dx}$ is Lipschitz on any interval of the form $(\eps,1]$ and $\lim_{x\to0}\frac{h(x)}{x^{-\gamma}}=C_0>0$ (see \cite{H04}, for example). Consider the canonical inducing domain $\Delta_0=(1/2,1]$ and define $R:\Delta_0\to\N$ as the first return time, \ie $R(x)=\inf\{j\in\N\colon\; T_\gamma^j(x)\in\Delta_0\}$. As before we denote by $\mu_0$ the restriction of $\mu_\gamma$ to $\Delta_0$. The induced map $\bar T_\gamma\colon \Delta_0\to\Delta_0$ given by $\bar T_\gamma=T_\gamma^{R(x)}(x)$ qualifies as a Rychlik map and therefore it has exponential decay of correlations of BV functions against $L^1$, so conditions $\D_{q_n}$, $\D'_{q_n}$ can be shown to hold by adapting the argument used in the proof of \cite[Theroem~4.3]{FFM20}. With the application to heavy tailed sums in mind, for definiteness we assume that $\varphi\colon [0,1]\to\R$ is such that $\varphi(x)=|x-\zeta|^{-1/\alpha}$, with $0<\alpha<1$ and $\zeta$ the period two point belonging to $\Delta_0$. 
The fact that the inducing map has such nice mixing properties allows us to apply Proposition~\ref{prop:point-process-convergence-induced}, where $\underline {\mathbf Q_i}=(\ldots,\tilde\infty, \tilde\infty,\tilde{ \mathbf K}_0, \tilde{ \mathbf K}_1, \tilde{ \mathbf K}_2,\ldots)$, $\tilde{ \mathbf K}_j=(\ldots,\infty,\infty,\chi^j,\infty,\infty,\ldots)$ and $\chi=DT_\gamma^2(\zeta)$. The computation of this transformed anchored tail process follows the exact same argument used in Appendix~\ref{subappendix:systems-perioidc-points}. The fact that all entries but one are $\infty$ in $\tilde{ \mathbf K}_j$ results from the canonical embedding and from the fact that during the induced time the orbits are always outside $\Delta_0$, which means that they are at fixed distance from $\zeta$ and therefore the normalisation dictates that $\infty$ should appear for the corresponding limits (see also Appendix~\ref{subappendix:systems-perioidc-points} for an argument on how the normalisation leads to $\infty$ for points at a fixed distance from $\mathcal M$).

We focus now on condition  \eqref{eq:condition-induced}. Notice that if $|\varphi(0)>0$, then, for an high value $u$, the set 
$$\left\{x\in\Delta_0\colon\;\sum_{i=0}^{R(x)-1}|\varphi(T^i(x))|>u\right\}=A_u\cup B_u$$  is the union of two intervals. One corresponding to a small neighbourhood of $\zeta$, $A_u=(\zeta-\delta_u,\zeta+\delta^u)$, where the value of $\varphi$ is already sufficiently high and the other corresponds to an interval of type $B_u=(1/2,1/2+\eps_u)$ because $T_\gamma(1/2)=0$ and since $DT_\gamma(0)=1$ this means that the orbit will linger for a very long time near $0$ so that the sum $\sum_{i=0}^{R(x)-1}|\varphi(T^i(x))|$ will add up $|\varphi(0)|$ so many times that it will ultimately exceed the level $u$. The shape of the observable near $\zeta$ dictates that $\mu_0(A_u)=\O(u^{-\alpha})$. We note that $\mu_0$ is equivalent to Lebesgue measure on $\Delta_0$, in the sense that its density is bounded above and away from 0 (see \cite[Lemma~2.3]{LSV99}). Using the same computation as in the proof of Theorem~1.3 of \cite{G04a}, one obtains that $\mu_0(B_u)=\O(u^{-1/\gamma})$. Now, since $\alpha<1<1/\gamma$, then clearly $\mu_0(B_u)=\oo(\mu_0(A_u))$ and therefore condition \eqref{eq:condition-induced} is satisfied.

Note that we could consider $1<\alpha<2$, but in order to guarantee that \eqref{eq:condition-induced} still holds when $\gamma\in(1/2,1)$, we would need the extra assumptions: $\varphi(0)=0$ and $\varphi(x)<C x^\beta$ for some $\beta >\gamma-1/2$. These would still guarantee that  $\mu_0(B_u)=\oo(\mu_0(A_u))$. (See proof of Theorem~1.3 of \cite{G04a}).

As in Example~\ref{example:oscillatory}, we can also consider an oscillatory behaviour creating overshooting by considering $\mathcal M=\{\zeta,T_\gamma(\zeta(\}$ and the corresponding observable $\psi(x)=|x-\zeta|^{-1/\alpha}-|x-T_\gamma(\zeta)|^{-1/\alpha}$, where $\zeta$ is as above. For simplicity, we assume that $0<\alpha<1$, so that in a similar way we have that condition \eqref{eq:condition-induced} is easily satisfied. In this case, to describe $\underline {\mathbf Q_i}$, we consider the random variable $E_i$ taking values in $\{-1,1\}$ and such that
$$
\p(E_i=1)=\lim_{u\to\infty}\frac{\mu_\gamma(\{|\psi|>u\}\cap\Delta_0)}{\mu_\gamma(\{|\psi|>u\})}.
$$
In a similar way to the computation of the transformed anchored tail process in Appendix~\ref{subappendix:overshooting} we obtain that $\underline {\mathbf Q_i}=(\ldots,\tilde\infty, \tilde\infty,\tilde{ \mathbf K}_0(E_i), \tilde{ \mathbf K}_1(E_i), \tilde{ \mathbf K}_2(E_i),\ldots)$, where 
\begin{align*}
\tilde{ \mathbf K}_0(-1)&=(\ldots,\infty,\infty,-1,\infty,\infty,\ldots)\\ 
\tilde{ \mathbf K}_j(-1)&=(\ldots,\infty,\infty,(-1)^{2j}\chi_{-1}^j\chi_1^{j-1},(-1)^{2j+1}\chi_{-1}^j\chi_1^{j},\infty,\infty,\ldots)\\
\tilde{ \mathbf K}_0(1)&=(\ldots,\infty,\infty,1,\chi_1,\infty,\infty,\ldots)\\ 
\tilde{ \mathbf K}_j(1)&=(\ldots,\infty,\infty,(-1)^{2j-1}\chi_{-1}^j\chi_1^{j},(-1)^{2j}\chi_{-1}^j\chi_1^{j+1},\infty,\infty,\ldots)
\end{align*}
with $\chi_1=DT_\gamma(\zeta)$ and $\chi_{-1}=DT_\gamma(T_\gamma(\zeta))$. The resulting excursion is:
$$
e^{T_i}_V(t)=V(T_i^-)+U_i^{-1/\alpha}\sum_{0\leq j\leq\left\lfloor\h(t) \right\rfloor} E_i(-1)^j \left(\chi_{-1}^{|\lfloor j/2\rfloor-(E_i+1)/2|_+}\chi_1^{|\lfloor j/2\rfloor+(E_i-1)/2|_+}\right)^{-1/\alpha},$$
where $t\in[0,1]$ and $|\cdot|_+=\max\{0,\cdot\}$.
\appendix

\section{Completeness and separability of the space $F'$} 
\label{sec:completeness-separability}

As in Whitt's space $E$, there are two natural metrics to use for the $E$ component of our space $F'$.  These are denoted $m_E$ and $m_E^*$ (see \eqref{eq:me} and \eqref{eq:me*}).  We keep the same metric on the $\tilde D$ part.

\begin{lemma}
$D$ with the sup norm is complete.
\label{lem:Dsupcomp}
\end{lemma}

\begin{proof}
Let $x_n\in D$ define a Cauchy sequence w.r.t. $\|\cdot\|_\infty$.  Then set $x(t)=\lim_{n\to\infty}x_n(t)$.  Note that $\cup_n disc(x_n)$ is an at most countable set. 
Observe also that in this metric (unlike $M_1$ say), each discontinuity $t$ of $x$ must 
correspond to a limit of discontinuities $(t_n)_n$ of $(x_n)_n$ (indeed we can take a sequence $(t_n)_n$ which is eventually constant).  So $x$ must have at most countably many discontinuities $disc(x)$.  The fact that $x$ is continuous on the right of the discontinuities follow from the continuity of $x_n$ on the right.  Hence $x\in D$.
\end{proof}

\begin{lemma}
$\tilde D$ is separable and complete.
\label{lem:tDsepcomp}
\end{lemma}

\begin{proof}
Completeness follows since $D$ with the sup norm is complete: if $[x_n]\in \tilde D$ gives a Cauchy sequence then there are $y_n\in [x_n]$ such that $(y_n)_n$ is Cauchy in $D$ with the sup norm.  Then the limit $y$, which exists as in Lemma~\ref{lem:Dsupcomp}, defines a class in $\tilde D$.

Separability follows since the set of piecewise constant $\Q$-valued functions in $D$ with rational discontinuity points (the discontinuity points can be any countable set here) is countable and defines a set of equivalence classes which is dense in $\tilde D$.
\end{proof}

As shown in \cite[Theorem 15.4.3]{W02}, $(E, m_E)$ is separable, and in \cite[Example 15.4.2]{W02} this space is not complete.   On the other hand for $(E,m_E^*)$ is complete, but not separable.

Assume we are dealing with real-valued functions rather than $\R^d$-valued for $d\ge 2$.

\begin{proposition}
\begin{enumerate}
\item[(a)] $F'$ with the $m_E$-metric on the $E$-component is not complete, but is separable.

\item[(b)] $F'$ with the $m_E^*$-metric on the $E$-component is complete, but is not separable.
\end{enumerate}
\label{prop:compsep}
\end{proposition}

\begin{proof}
\textbf{(a)}
\cite[Example 15.4.2]{W02} would not give a convergent sequence in $F'$ since we would not be able to get convergent excursions: there are large jumps in adjacent terms of the sequence which cannot be matched up.  Hence we need an alternative example for non-completeness.

For each $n$ we define 
$$\phi_n(x) =\begin{cases} nx &\text{ if } x\in [0, \frac1n)\\
 1 &\text{ if } x\in \left[ \frac1n, 1\right]\\
\end{cases}
$$
Also let, for each $k\in \N$, $x_{k, 1}, \ldots, x_{k, n+k}$ be the set of $n+k$ points in $[\frac1{2^k}, \frac1{2^{k-1}}]$ so that  these points, along with the boundary points, are equidistributed through this interval.  Define $\eps_{n, k}=\frac{1}{n2^{k+2}}$.  Then define
$$x_n=\sum_{k=1}^\infty\sum_{i=1}^{n+k} \phi_n\cdot \mathbbm{1}_{\left[\left. x_{k, i}-\eps_{n, k}, \right. x_{k, i}+\eps_{n, k}\right)}.$$
Then $S_n=\{x_{k, i}\pm \eps_{n, k}\}_{k, n}$ and the excursions are for example $\phi_n(x_{k, i}-\eps_{n, k})\cdot \mathbbm{1}_{[1/2, 1]}$ at the points $x_{k, i}-\eps_{n, k}$ and  $\phi_n(x_{k, i}+\eps_{n, k})\cdot\mathbbm{1}_{[0, 1/2)}$ at the points $x_{k, i}+\eps_{n, k}$, producing an element $\ux_n\in F'$ (note the damping effect of $\phi_n$ ensuring that \eqref{eq:finitebigjumps} holds).
Then $(\ux_n)_n$ is Cauchy with the $m$-metric, but this sequence does not have a limit in $F'$ (eg the graph of the limit would have to be $[0, 1]\times[0,1]$ as in \cite[Example 15.4.2]{W02}).  Hence the space is not complete.

To show separability, we use the usual approximation by objects with rational coordinates.  As shown in \cite[Theorem 15.4.3]{W02}, the $E$-part of the space is separable.  We then use Lemma~\ref{lem:tDsepcomp} to give separability of the excursion part.

\textbf{(b)} Completeness follows from the $E$-component being complete as in \cite[Section 15.4]{W02}, which is the same type of argument as completeness in $D$ (see Lemma~\ref{lem:Dsupcomp}) and $\tilde D$ since this is uniform convergence.  The latter facts give convergence to a limit in $\tilde D$ for the excursions also.

Non-separability of $E$ (with $m^*$) implies the non-separability of $F'$: in fact this is the same type of argument as for $D$.  Note that for non-separability it is sufficient, for a given positive distance, to find an uncountable collection of elements of the set all at least that distance apart.  So for example for $D$, $\{\mathbbm{1}_{[a, 1]}\}_{a\in [0, 1)}$ are all distance 1 apart in the $m^*$ metric.
\end{proof}

\begin{remark}
\cite[Example 15.4.2]{W02} would not give a convergent sequence in $F'$ since we would not be able to get convergent excursions.  Hence we needed an alternative example above for non-completeness. \end{remark}

\begin{remark}
We recall that as described in \cite[Section~11.5.2]{W02} for the space $D$, the hierarchy of the Skorohod's topologies  dictates that convergence in $J_1$ implies convergence in $M_1$, which in turn implies convergence in $M_2$. In the richer space $E$, after embedding, the $M_2$ metric actually corresponds to the $m_E$ metric. Moreover, as observed earlier, convergence in $F'$ implies convergence both in $E$ and $F$, when endowed with the corresponding metrics.
\end{remark}

\section{Weak and weak$^\#$ convergence}

\label{sec:app_weak_conv}

The purpose of this section is to review notions of convergence of measures on the metric spaces $\tilde l_\infty$ and $\tilde l_0$. We are particularly interested in weak convergence of probability measures and, since we want to consider point processes which are random elements corresponding to boundedly finite measures, we are also interested in weak$^\#$ convergence. Recall that $\tilde l_\infty$ and $\tilde l_0$ are complete, separable, metric spaces, which are not locally compact and therefore vague convergence is not useful here. We remark that weak$^\#$ convergence is equivalent to vague convergence when the ambient space is locally compact (see \cite[Appendix~A2.6]{DV03}).

The Portmanteau Theorem (see for example \cite[Theorem~2.1]{B99}) is the following.

\begin{theorem}
Given a metric space, the following conditions are equivalent to weak convergence of Borel probability measures.
\begin{enumerate}
\item \label{item:weak-conv} $\lim_{n\to\infty}\int f d\p_n=\int f d\p$, for all bounded, uniformly continuous real $f$;

\item \label{item:weak-closed} $\limsup_{n\to\infty} \p_n(F)\leq\p(F)$ for all closed set $F$;

\item \label{item:open} $\liminf_{n\to\infty} \p_n(E)\geq\p(E)$ for all open set $E$;

\item \label{item:weak-continuity} $\lim_{n\to\infty} \p_n(A)=\p(A)$ for all set $A$ such that $\p(\partial A)=0$.
 
\end{enumerate}
\label{thm:portmanteau}
\end{theorem}

A set $A$ such that $\p(\partial A)=0$ is called a $\p$-continuity set. A class $\mathscr U$ of sets is called \emph{a convergence determining class} 
if the convergence $\lim_{n\to\infty} \p_n(A)=\p(A)$ for all $\p$-continuity sets $A\in\mathscr U$ implies the weak convergence of $\p_n$ to $\p$.

\begin{lemma}
\label{lem:cdc}
The class of sets $\tilde{\mathscr J}$ as in \eqref{eq:A-tilde} is a convergence determining class for weak convergence on the metric space $\tilde l_\infty\setminus\{\tilde\infty\}.$
\end{lemma}
\begin{proof}
We start by proving a claim.

\begin{claim} For all $\tilde\x\in \tilde l_\infty\setminus\{\tilde\infty\}$ and all $\varepsilon>0$, there exists a $\p$-continuity set $\tilde A\in \tilde{\mathscr J}$ such that $\tilde\x\in\mathring{\tilde A}\subset\tilde A\subset B(\tilde\x,\varepsilon)$.
\end{claim}

\begin{proof}[Proof of Claim]
Let  $\z,\w\in l_0\setminus\{\mathbf 0\}$. Since $\left|\|\z\|_\infty-\|\w\|_\infty\right|\leq \|\z-\w\|_\infty$ and $h(x)=\frac 1x$ is continuous on $\R\setminus\{0\}$, then there exists $\delta>0$ such that $\|\z-\w\|_\infty<\delta$ implies that $d'(\z,\w)<\varepsilon$, where $d'$ is the metric on $l_0\setminus\{\mathbf 0\}$ defined in \eqref{eq:d'-metric}. 

Let $\z=\Proj(\x)\in l_0\setminus\{\mathbf 0\}$ and let $k\in\N$ be such that $\|z_j\|<\delta$ for all $|j|\geq k$. Define 
$$B_{\rho}=\bigcap_{|j|<k}\TT^{-j}([{z_j}_1-\rho,{z_j}_1+\rho)\times\ldots\times[{z_j}_d-\rho,{z_j}_d+\rho)),$$ 
for some $\rho<\delta/\sqrt d$.  Let $A_\rho=\Proj^{-1}(B_\rho)$ and let $\tilde A_\rho\in \tilde{\mathscr J}$ be associated to $A_\rho$ as in \eqref{eq:A-tilde}. We choose $\rho$ so that $\tilde A_\rho$ is a $\p$-continuity set. We can always choose such $\rho$ because $\partial \tilde A_\rho\subset \tilde \pi\left(\cup_{i\in\Z}(\TT^{-i}(\cup_{|j|<k}\cup_{m=1}^d\{\w: |{w_j}_m-{z_j}_m|=\rho\}))\right)$, which means that each $\partial \tilde A_\rho$ can intersect at most countably many other such sets.  Therefore, there is an uncountable number of disjoint sets $\partial \tilde A_\rho$, for $0<\rho<\delta/\sqrt d$, and since there cannot be an uncountable number of them with positive probability, at least one of them must have 0 probability, which means one of the $\tilde A_\rho$ is a $\p$-continuity set.

We claim that $\tilde A_\rho\subset B(\tilde \x, \varepsilon)$. To see this, let $\w\in l_\infty$ be such that $\tilde \w=\tilde \pi(\w)\in \tilde A_\rho$. By definition of $\tilde A_\rho$ we must have $\TT^\ell(\w)\in A_\rho$, for some $\ell\in\Z$. Hence, $\Proj(\TT^\ell(\w))=\TT^\ell(\Proj(\w))\in B_\rho$, which means that $\|\Proj(\x)- \TT^\ell(\Proj(\w))\|_\infty <\delta$ and therefore $d'(\Proj(\x), \TT^\ell(\Proj(\w)))<\varepsilon$. It follows that, by definition of the metric $\tilde d'$ on $\tilde l_0$, given in \eqref{eq:metric-d-prime-tilde}, we have $\tilde d'(\tilde \Proj(\tilde \x), \tilde \Proj(\tilde \w))<\varepsilon$ and hence
$\tilde d(\tilde \x, \tilde \w)<\varepsilon$.
\end{proof}

Let $E$ be an open set.  By the claim, for each $\tilde\x\in E$ there exists a $\p$-continuity set $\tilde A\in\tilde{\mathscr J}$ such that $\tilde \x\in \mathring{\tilde A}\subset \tilde A\subset E$. This means we have an open cover of $E$ by sets $\mathring{\tilde A}$, where $\tilde A\in\tilde{\mathscr J}$ is a $\p$-continuity set. Since on a separable metric space (such as $\tilde l_\infty\setminus\{\tilde\infty\}$) each open cover admits a countable subcover (see \cite[Appendix~M3]{B99}, for example)then there exists a sequence $(\tilde A_i)_{i\in\N}$ such that each $\tilde A_i\in\tilde{\mathscr J}$ is a $\p$-continuity set, $\tilde A_i\subset E$ and, moreover, $E\subset \cup_{i\in\N} \mathring{\tilde A}_i$, which means that $E=\cup_{i\in\N} \tilde A_i$.

We are assuming that $\lim_{n\to\infty}\p_n(\tilde A)=\p(\tilde A)$ for all $\p$-continuity sets $\tilde A\in\tilde{\mathscr J}$. Since   $\tilde{\mathscr J}$ is closed for finite unions and $\partial (\tilde A\cup\tilde B)\subset \partial \tilde A\cup\tilde B$ (see \cite[Proposition~A1.2.I]{DV03}), we have that $\cup_{i=1}^m\tilde A_i\in\tilde{\mathscr J}$ is a $\p$-continuity set and therefore
$$
\lim_{n\to\infty}\p_n(\cup_{i=1}^m\tilde A_i)=\p(\cup_{i=1}^m\tilde A_i).
$$

Since $E=\cup_{i\in\N} \tilde A_i$, for $\varepsilon>0$, let $m\in\N$ be such that $\p(\cup_{i=1}^m\tilde A_i)>\p(E)-\varepsilon$. Then,
$$
\p(E)-\varepsilon<\p(\cup_{i=1}^m\tilde A_i)=\lim_{n\to\infty}\p_n(\cup_{i=1}^m\tilde A_i)\leq \liminf_{n\to\infty} \p_n(E).
$$
Since the last inequalities hold for all $\varepsilon>0$ then 
$$
\p(E)\leq \liminf_{n\to\infty} \p_n(E), \quad \text{for all $E$ open}
$$
and therefore by Theorem~\ref{thm:portmanteau}\eqref{item:open}, $\p_n$ converges weakly to $\p$.
\end{proof}

In order to study point processes, we need to consider $\sigma$-finite measures taking finite measures on bounded sets. Namely, we define:
\begin{definition}
A Borel measure $\mu$ on a complete, separable, metric space is boundedly finite if $\mu(A)<\infty$ for every bounded Borel set $A$.
\end{definition}
Let $\mathcal X$ denote a complete, separable, metric space such as $\tilde l_\infty\setminus\{\tilde\infty\}$. We denote by $\mathcal M_{\mathcal X}^\#$ the space of boundedly finite Borel measures on $\mathcal X$. Following \cite{DV03, DV08}, we consider a notion of convergence in $\mathcal M_{\mathcal X}^\#$ called the \emph{weak hash} or \emph{weak$^\#$} convergence, denoted by $\mu_k\rightarrow_{w^\#}\mu$ which can be defined by any of the  following equivalent conditions (see \cite[Proposition~A2.6.II]{DV03}):
\begin{enumerate}[label=(\roman*)]

\item \label{item:weak-hash} $\lim_{k\to\infty}\int f d\mu_k=\int f d\mu$ for all bounded continuous functions $f$ defined on $\mathcal X$ and vanishing outside a bounded set;

\item \label{item:weak-hash-weak}There exists an increasing sequence of bounded open sets $B_n$ converging to $\mathcal X$ such that if $\mu_k^{(n)}$ and $\mu^{(n)}$ denote the restrictions of the measures $\mu_k$ and $\mu$ to $B_n$, respectively, then $\mu_k^{(n)}$ converges weakly to $\mu^{(n)}$, as $k\to\infty$, for all $n\in\N$. Note that $\mu_k^{(n)}$ and $\mu^{(n)}$ are not necessarily probability measures, so when we say that there is weak convergence we mean that either \eqref{item:weak-conv} and \eqref{item:weak-continuity} from Theorem~\ref{thm:portmanteau} apply or \eqref{item:weak-closed} and \eqref{item:open} apply with the extra assumption that $\lim_{k\to\infty}\mu_k^{(n)}(B_n)=\mu^{(n)}(B_n)$, for all $n\in\N$.

\item \label{item:weak-hash-continuity} $\lim_{k\to\infty} \mu_k(A)=\mu(A)$ for all bounded Borelean $A$ such that $\mu(\partial A)=0$.

\end{enumerate}

Similarly to Lemma~\ref{lem:cdc}, we show that in order to prove weak$^\#$ convergence on $\tilde l_\infty\setminus\{\tilde\infty\}$,  we only need to check $\lim_{k\to\infty} \mu_k(\tilde A)=\mu(\tilde A)$ for all bounded, $\mu$-continuity sets $\tilde A\in\tilde{\mathscr J}$, as in property \ref{item:weak-hash-continuity}.
\begin{lemma}
\label{lem:cdc-hash}
The class of bounded sets in $\tilde{\mathscr J}$ is a convergence determining class for weak$^\#$ convergence on the metric space $\tilde l_\infty\setminus\{\tilde\infty\}.$
\end{lemma}
\begin{proof}
Suppose $\mu_k\rightarrow_{w^\#}\mu$.  
For all $a>1$, let $\mathscr F \ni B_a:=\{\x\in \dot{\V}^{\Z}\colon\; 1/a<\|x_0\|<a\}$ and $\tilde B_a\in \tilde{\mathscr J}$ be associated to $B_a$ as in \eqref{eq:A-tilde}. Observe that $\tilde d(\tilde \x, \tilde\w)\leq a-1/a$, for all $\tilde x,\tilde w \in \tilde B_a$.  Since there are uncountably many sets of the form $B_a$ and $\partial B_a\subset\{\x\in \dot{\V}^{\Z}\colon\; 1/a=\|x_0\|\;\text{or}\;\|x_0\|=a\}$, then we can find a strictly increasing, diverging sequence $(a_n)_{n\in\N}$, such that $B_n:=B_{a_n}$ is such that $\mu(\partial \tilde B_n)=0$, for all $n\in\N$.  
 
Hence, $(B_n)_{n\in\N}$ is an increasing sequence of bounded, open, $\mu$-continuity sets converging to $\tilde l_\infty\setminus\{\tilde\infty\}.$ Note that by hypothesis we are assuming that 
 $\lim_{k\to\infty} \mu_k(\tilde A)=\mu(\tilde A)$ for all bounded, $\mu$-continuity sets $\tilde A\in\tilde{\mathscr J}$ and therefore that applies to all the $\tilde B_n$, namely, $\lim_{k\to\infty}\mu_k(\tilde B_n)=\mu(B_n)$. Therefore, we can apply \eqref{item:open} to prove weak convergence of $\mu_k^{(n)}$ (the restriction of $\mu_k$ to $B_n$) to $\mu^{(n)}$ (the restriction of $\mu$ to $B_n$), which was the main tool used in the proof of Lemma~\ref{lem:cdc}. It follows then, by assumption and Lemma~\ref{lem:cdc} that $\mu_k^{(n)}$ converges weakly to $\mu^{(n)}$, for all $n\in\N$ and, by \ref{item:weak-hash-weak}, we obtain the weak$^\#$ convergence of $\mu_k$ to $\mu$.\end{proof}

\section{Convergence of point processes on non locally compact spaces}
\label{appendix:convergence-point-processes}

We closely follow  \cite[Appendix~A2.6]{DV03} and \cite[Sections~9 and 11]{DV08}. As before, let $\mathcal X$ be a complete separable metric space. The notion of weak$^\#$ convergence is metrizable, in the sense that $\mathcal M_{\mathcal X}^\#$ admits a metric generating what the so-called $w^\#$-topology so that the weak$^\#$ convergence corresponds to convergence in the $w^\#$-topology. Denote by $\mathcal B(\mathcal M_{\mathcal X}^\#)$ the corresponding  Borel $\sigma$-algebra. A \emph{random measure} is a random element in $(\mathcal M_{\mathcal X}^\#,\mathcal B(\mathcal M_{\mathcal X}^\#))$. A point process $N$ is an integer-valued random measure. Let $\mathcal N_{\mathcal X}^\#$ denote the space of boundedly finite integer-valued measures. We have that $\mathcal N_{\mathcal X}^\#$ is a closed subset of $\mathcal M_{\mathcal X}^\#$ (\cite[Proposition~9.1.V]{DV08}) and let $\mathcal B(\mathcal N_{\mathcal X}^\#)$ denote the corresponding $\sigma$-algebra for the $w^\#$-topology. We remark that $\mu\in\mathcal N_{\mathcal X}^\#$ has the following form (\cite[Proposition~9.1.III]{DV08}):
\begin{equation}
\label{eq:point-measure}
\mu=\sum_{i\in\N} k_i\delta_{x_i}, \quad \text{where $\delta_{x_i}$ is the Dirac measure at distinct $x_i\in\mathcal X$, and $k_i\in\N$}.
\end{equation}
 For any $\mu\in\mathcal N_{\mathcal X}^\#$ given by \eqref{eq:point-measure}, we define its \emph{support counting measure} $\mu^*$ as:
\begin{equation}
\label{eq:simple-point-measure}
\mu^*=\sum_{i\in\N} \delta_{x_i}.
\end{equation}
The boundedly finite measure $\mu\in\mathcal N_{\mathcal X}^\#$ is simple if and only if $\mu=\mu^*$, which is equivalent to say that $k_i=1$ for all $i\in\N$.

Now, we give formally now the definition of point process. 
\begin{definition}
\label{def:point-process}
A point process $N$ on state space $\mathcal X$ is measurable mapping from a probability space $(\Omega, \mathcal B, \p)$ into $(\mathcal N_{\mathcal X}^\#,\mathcal B(\mathcal N_{\mathcal X}^\#))$. A point process $N$ is said to be \emph{simple} if $\p(N(\{x\})>1)=0$, for all $x\in\mathcal X$. To each point process $N$ on state space $\mathcal X$ we denote by $N^*$ the corresponding \emph{support point process} obtained from $N$ as in \eqref{eq:simple-point-measure}.
\end{definition}
 
Let $(N_n)_{n\in\N}$ be a sequence of point processes and $N$ another point process, all with state space $\mathcal X$. We say that $N_n$ converges weakly to $N$, when the respective distributions $P_n$, defined by $P_n(A):=\p(N_n\in A)$, for all $A\in \mathcal B(\mathcal N_{\mathcal X}^\#))$, converge weakly (in the sense of weak convergence of probability measures on the metric space $\mathcal N_{\mathcal X}^\#$) to the distribution $P$ associated to $N$.

Tightness is a very useful property which gives relative compactness and ultimately weak convergence. We state two conditions that are necessary and sufficient to show that a family of probability measures $\{P_t,\;t\in\mathcal T\}$ is \emph{uniformly tight}, see \cite[Proposition~11.1.VI]{DV08}. Given any closed sphere $S\subset \mathcal X$ and any $\epsilon,\delta>$, there exists a real number $M$ and a compact set $C\subset S$ such that for $t\in\mathcal T$,
\begin{align}
P_t(N_t(S)>M)&<\epsilon,\label{eq:tight1}\\
P_t(N_t(S\setminus C)>\delta)&<\epsilon\label{eq:tight2}.
\end{align}
We remark that if $\mathcal X$ was locally compact, \eqref{eq:tight2} would be unnecessary.   
 
Based on these criteria for tightness, in \cite[Proposition~11.1.VII]{DV08} it is shown that weak convergence of $N_n$ to $N$ follows from the convergence:
$$
(N_n(A_1),\ldots, N_n(A_k))\xrightarrow[]{}(N(A_1),\ldots, N(A_k)),
$$ 
of joint distributions as random vectors in $\R^k$, for all finite collections $\{A_1,\ldots, A_k\}$ of bounded continuity sets $A_i\in\mathcal B_{\mathcal X}$, for all $i=1,\ldots,k$ and all $k\in\N$. Here, \emph{continuity set} means that $N(\partial A_i)=0$ a.s. 

When the limiting point process is simple, a simpler criterion for convergence can be used. 
\begin{proposition}
\label{prop:weak-limit-PP}
Let $(N_n)_{n\in\N}$  be a sequence of point processes on the state space $\mathcal X$ and let $N$ be simple point process on the same state space. Let $\mathcal R$ be a covering dissecting ring\footnote{A dissecting ring is a ring generated by finite unions and intersections of the elements of a dissecting system, which consists on a nested sequence of finite partitions of the whole space that eventually separate points. See \cite[Appendices~A1.6 and A2.1]{DV03} for further details.} of continuity sets of $\mathcal X$. Then, $N_n$ converges weakly to $N$ if the two following conditions hold:
\begin{enumerate}[label=(\roman*)]

\item $\lim_{n\to\infty}\p(N_n(A)=0)=\p(N(A)=0)$ for all bounded $A\in\mathcal R$; \label{eq:kallenberg1}

\item \label{eq:orderiless} for all bounded $A\in\mathcal R$ and a nested sequence of partitions $\mathcal T_r=\{A_{ri}\colon\; i=1,\ldots, k_r\}$ of $A$ by sets of $\mathcal R$ that ultimately separate the points of $A$,
$$
\inf_{\mathcal T_r}\limsup_{n\to\infty} \sum_{i=1}^{k_r} \p(N_n(A_{ri})\geq 2)=0.
$$
\end{enumerate}
Alternatively, we may replace condition \ref{eq:orderiless} by the following:
\begin{enumerate}[label=(\Roman*)]

\setcounter{enumi}{1}

\item \label{eq:kallenberg2} $\lim_{n\to\infty}\E(N_n(A))=\E(N(A))$, for all bounded $A\in\mathcal R$.

\end{enumerate}

\end{proposition}

\begin{remark}
This proposition is very similar to \cite[Proposition~11.1.IX]{DV08}. However, note that the corresponding statement to our condition \ref{eq:orderiless}, namely \cite[equation (11.1.4)]{DV08},  is incorrectly stated there. Essentially, this condition is supposed to require that there is no accumulation of mass points as the point processes approach the limiting point process. This is related to a property called \emph{ordinary} in  \cite[equation (11.1.4)]{DV08}, which is then used in the proof of \cite[Proposition~11.1.IX]{DV08}. When one observes the condition giving the notion of ordinary point process one realises the need to use the infimum over all partitions, as we did in \ref{eq:orderiless}, which contrasts with the use of the supremum over all partitions used in \cite[equation (11.1.4)]{DV08}. 

For completeness we redo the proof of \cite[Proposition~11.1.IX]{DV08}, correcting this typo. We also extend the result by showing that condition \ref{eq:kallenberg2} can be used to replace condition \ref{eq:orderiless} when trying to prove weak convergence to a simple point process.

\end{remark}

\begin{proof}
We closely follow the proof of \cite[Proposition~11.1.IX]{DV08}.
At the core of the proof this proposition is a result attributed to R\'enyi and M\"onch, see \cite[Theorem~9.2.XII]{DV08}, which states that the distribution of a simple point process on a complete separable metric spaces $\mathcal X$ is determined by the values of the avoidance function $P_0$ on the bounded sets of a dissecting ring $\mathcal R$ for $\mathcal X$, where $P_0$ is defined by:
$$
P_0(A)=\p(N(A)=0),\quad \text{for $A\in \mathcal R$}.
$$
We are then left to show that the family $\{P_n\colon n\geq n_0\}$ is uniformly tight and that the limit of any convergent subsequent is a simple point process, where $P_n$ is the distribution of the point process $N_n$.

Assuming that \ref{eq:orderiless} holds, let $A$ be a closed ball in $\mathcal R$. Observing that $\{N_n(A)>k_r\}$ implies $\{N_n(A_{ri})\geq 2\}\}$ for at least one $i$, then
$$
\sum_{i=1}^{k_r}\p(N_n(A_{ri})\geq 2)\geq \p(N_n(A)\geq k_r).
$$
Note that the sum on the left is non-increasing with $r$. Hence, given $\epsilon>0$, by \ref{eq:orderiless}, there is an $r_0$ such that  $\forall r\geq r_0$, we have $\limsup_{n\to\infty} \sum_{i=1}^{k_r} \p(N_n(A_{ri})\geq 2)<\epsilon.$ In particular, there exists $n_0\in\N$ such that for all $n>n_0$, we have $\p(N_n(A)\geq k_{r_0})\leq \sum_{i=1}^{k_{r_0}} \p(N_n(A_{r_0i})\geq 2)<\epsilon$. Therefore, choosing $M>k_{r_0}$ large enough, we have  $\p(N_n(A)\geq M)<\epsilon$, for all $n\in\N$, which means that \ref{eq:tight1} holds.

Now, we verify that the same happens when we assume \ref{eq:kallenberg2} instead of \ref{eq:orderiless}. Using Chebyshev's inequality, we obtain for all $n\in\N$
$$
\p(N_n(S)>M)\leq\frac{ \E(N_n(S))}{M}.
$$
By \ref{eq:kallenberg2}, there exists $K\in\N$ such that $\E(N_n(S))<K$, for all $n\in\N$. Therefore, taking $M$ sufficiently large so that $K/M<\epsilon$, we obtain that $\p(N_n(S)>M)<\epsilon$, for all $n\in\N$ and, therefore, \ref{eq:tight1} holds.

We next show that \ref{eq:kallenberg1} implies \ref{eq:tight2}, which can be restated here in the following form: given $\epsilon>0$. there exists a compact set $C$ such that $\p(N_n(S\setminus C)=0)>1-\epsilon$, for all $n\in\N$. We choose $C$ so that for the limit distribution (which corresponds to that of a simple point process), we have
$$
\p(N(S\setminus C)=0)>1-\epsilon/2.
$$  
From \ref{eq:kallenberg1}, we have that $\lim_{n\to\infty}\p(N_n(S\setminus C)=0)=\p(N(S\setminus C)=0)$, from which we obtain that $\p(N_n(S\setminus C)=0)>1-\epsilon$, for all $n$ sufficiently large and, by taking a larger $C$ if necessary, for all $n\in\N$, which means \ref{eq:tight2} holds.

Since, both conditions \ref{eq:tight1} and \ref{eq:tight2} hold then $N_n$ admits a weakly convergent subsequence, say $(N_{n_\ell})_{\ell\in\N}$, which converges weakly to $\bar N$. By   \ref{eq:kallenberg1}, we have
$$
\p(N(A)=0)=\p(\bar N=0),
$$
but we still do not know if $\bar N$ is simple, hence this only gives us that $N$ and $\bar N^*$ have the same distribution, where $\bar N^*$ is the support point process associated to $\bar N$.

Hence, we are left to prove that both \ref{eq:orderiless} and \ref{eq:kallenberg2} imply that $\bar N$ is simple, \ie $\bar N=\bar N^*$, a.s.

Note that for all $r$, we have
$$
\sum_{i=1}^{k_r} \p(\bar N(A_{ri})\geq 2)=\lim_{\ell\to\infty}\sum_{i=1}^{k_r} \p( N_{n_\ell}(A_{ri})\geq 2).
$$ 
Assuming \ref{eq:orderiless}, we obtain 
$$
\inf_{\mathcal T_r}\sum_{i=1}^{k_r} \p(\bar N(A_{ri})\geq 2)
\leq \inf_{\mathcal T_r}\limsup_{n\to\infty}\sum_{i=1}^{k_r} \p( N_{n}(A_{ri})\geq 2)=0.
$$
This means that the point process $\bar N$ is ordinary, as defined in \cite[Definition~9.3.XI]{DV08}, and, therefore, by \cite[Proposition~9.3.XII]{DV08}, $\bar N$ is simple. 

Now, assuming  \ref{eq:kallenberg2} instead, we have for all bounded $A\in \mathcal R$,
$$
\E(N(A))=\E(\bar N^*(A))\leq \E(\bar N(A))=\lim_{\ell\to\infty}\E(N_{n_\ell}(A))
=\E(N(A)),
$$
which means that $\E(\bar N^*(A))= \E(\bar N(A))$ for all bounded $A\in \mathcal R$, which implies that $\bar N$ must be simple.
\end{proof}

\begin{remark}
\label{rem:kallenberg2}
In fact, since given any sequence of positive random variables $(X_n)_{n\in\N}$ converging in distribution to $X$, we have $\liminf_{n\to\infty} \E(X_n)\geq \E(X)$, then we can strengthen Proposition~\ref{prop:weak-limit-PP} by replacing \ref{eq:kallenberg2} by the condition $\limsup_{n\to\infty}\E(N_n(A))\leq \E(N(A))$, which actually only need verifying for all bounded sets in semi-ring generating $\mathcal R$.

\end{remark}

We consider now the particular case where $\mathcal X=\R_0^+\times \tilde l_\infty\setminus\{\tilde\infty\}$. In this space, we consider the ring $\mathcal R$ and its subclass of sets $\mathcal I$ defined by:  
\begin{align}
\mathcal R:&=\left\{\bigcup_{\ell=1}^m J_\ell\times \tilde A_\ell\colon\; m\in\N,\; J_\ell=[a_\ell,b_\ell),\; \text{and $\tilde A_\ell\in \tilde{\mathscr R}$} \right\};\label{eq:ring-def}\\
\mathcal I:&=\left\{\bigcup_{\ell=1}^m J_\ell\times \tilde A_\ell\colon\; m\in\N,\; J_\ell=[a_\ell,b_\ell),\; \text{and $\tilde A_\ell\in \tilde{\mathscr J}$} \right\}.\label{eq:sub-ring-def}
\end{align}
\begin{lemma}
\label{lem:simplified}
The conclusion of Proposition~\ref{prop:weak-limit-PP} holds if the conditions 
are verified only for all bounded sets of $\mathcal I$.
\end{lemma}
See related results in \cite{K17} (for example Theorem~2.2).
\begin{proof}
We focus on the condition \ref{eq:kallenberg1}, which is the strongest. The others follow easily as one can already guess from Remark~\ref{rem:kallenberg2}. 
The statement follows from the fact that $\tilde{\mathscr J}$ is closed for unions and then using the inclusion exclusion formula, one can write the probability involving any set $\tilde A_\ell\in \tilde{\mathscr R}$ using the probability of events involving only elements of $ \tilde{\mathscr J}$. 

For example, assume that $\tilde A_1=\tilde B_1\cap \tilde D_1$, where $\tilde B_1, \tilde D_1\in\tilde{\mathscr J}$. Observe that 
\begin{align*}
\p(N(\cup_{\ell=1}^m J_\ell\times \tilde A_\ell)&=0)=\p\left(N(J_1\times (\tilde B_1\cap\tilde D_1))=0, \;N(\cup_{\ell=1}^m J_\ell\times \tilde A_\ell)=0\right)\\
=&\p\left(N(J_1\times \tilde B_1)=0,\;N(\cup_{\ell=1}^m J_\ell\times \tilde A_\ell)=0\right)\\
&+\p\left(N(J_1\times \tilde D_1)=0,\;N(\cup_{\ell=1}^m J_\ell\times \tilde A_\ell)=0\right)\\
&-\p\left(N(J_1\times (\tilde B_1\cup\tilde D_1))=0,\;N(\cup_{\ell=1}^m J_\ell\times \tilde A_\ell)=0\right),
\end{align*}
which means that all events on the 3 last terms correspond to the value of the avoidance function over sets of  $\mathcal I$, as we wanted.
In the same way, if $\tilde A_1=\tilde B_1\setminus \tilde D_1$, for $\tilde B_1, \tilde D_1\in\tilde{\mathscr J}$, for example, we could use the fact that $\p(N(\tilde A_1)=0)=\p(N(\tilde B_1\cup \tilde D_1)=0)-\p(N(\tilde D_1)=0)$ and obtain a similar formula for the avoidance function using only sets in $\mathcal I$ as we did above. 

Hence, we have just shown that we can handle unions, intersections and exclusions of sets of $ \tilde{\mathscr J}$.
Noting that $\tilde A_\ell\in \tilde{\mathscr R}$ can always be written by using a finite number of these set operations involving elements of $ \tilde{\mathscr J}$, the conclusion follows. 
\end{proof}

\section{Computation of the transformed anchored tail processes for the dynamical examples}
\label{Appendix:dyn}

Section~\ref{subsec:systems} gave concrete dynamical examples where the theory in this paper holds.  As noted in the discussion after Theorem~\ref{thm:point-process-convergence}, previous work shows that many of the required conditions hold.  In this section we fill in the missing proofs by showing that the transformed anchored tail process is well-defined for observables whose norm is maximised at periodic points.  Note that the existence of the limit in \eqref{eq:EI} follows from a simpler version of the ideas here.

\subsection{Applications to observables maximised at hyperbolic periodic points for general 1D systems}
\label{subappendix:systems-perioidc-points} 
Recall that we are dealing with an acip $\mu$ and that for our periodic point $\zeta$ we are always assuming that the density exists and is bounded, say $\frac{d\mu}{\text{Leb}}=D\in (0, \infty)$. For definiteness assume that we are dealing with a non-invertible map, the observable is as in \eqref{eq:dynamics-particular-observable}, $\mathcal M=\{\zeta\}$ and $\zeta$ is a repelling fixed point (therefore $p=1$). In this case the process $(Y_j)_{j\in \Z}$ is such that, for some $s\in \N_0$, we have $Y_j=\infty$, for all $j\leq -s$ and $Y_j=U\cdot \|(DT_\zeta)^{j}(\Theta)\|^d\frac{(DT_\zeta)^{j}(\Theta)}{\|(DT_\zeta)^{j}(\Theta)\|}$, for all $j>-s$,
where $U$ is a uniformly distributed random variable on the interval $[0,1]$ independent of $\Theta$, which has a uniform distribution on $\mathbb S^{d-1}$ and $(DT)^i_\zeta$ denotes the $i$-fold product of the derivative of $T$ at $\zeta$ (which is invertible since we are assuming that $\zeta$ is repelling).

This clearly satisfies \eqref{spectral-process}--\eqref{positive-EI} for the transformed anchored tail process, so we need to show \eqref{Y-def}, i.e., 
$$\mathcal L\left(\frac1\tau \mathbb X_n^{r_n+s, r_n+t}\;\middle\vert\; \|\X_{r_n}\|>u_n(\tau) \right)\xrightarrow[n\to\infty]{}\mathcal L\left((Y_j)_{j=s,\dots,t}\right), $$
 for all  $s<t\in\Z \text{ and all }\tau>0$.

Without loss of generality we assume that $g$ is positive in a neighbourhood of $\zeta$, so $\|\X_{r_n}(x)\|>u_n(\tau)$ can be written $g \left(\dist(T^{r_n}(x), \zeta)\right)> u_n(\tau)$.  From \eqref{un} we see that asymptotically $u_n(\tau)\sim g\left(\left(\frac{\tau}{Cn}\right)^{1/d}\right)$, where $C=D L_\zeta$ for $L_\zeta =\lim_{r\to 0}\frac{\text{Leb}(B_{r}(\zeta))}{r}$ (this exists in the cases considered here).  In this section we will assume that this is an equality since all our estimates are asymptotic in $n$, so similarly $u_n^{-1}(z)=Cn\left(g^{-1}(z)\right)^d$,  where we will also assume this is well-defined.  Since we are assuming that $\|\X_{r_n}(x)\|>u_n(\tau)$, let $v<1$ be such that $\|\X_{r_n}(x)\|=u_n(v\tau)$, which translates as $\dist(T^{r_n}(x), \zeta)=\left(\frac{v\tau}{Cn}\right)^{1/d}$.  Using the linearisation domain around the fixed point $\zeta$ from Hartman-Grobman theory, $\dist(T^{r_n+k}(x), \zeta)\sim \left(\frac{v\tau}{Cn}\right)^{1/d}\|(DT_\zeta)^k(w)\|$, where $w=\frac{\Phi^{-1}_\zeta(T^{r_n}(x))}{\|\Phi^{-1}_\zeta(T^{r_n}(x))\|}$ and $\Phi_\zeta$ is as in \eqref{eq:dynamics-particular-observable}. It follows that 
$$\X_{r_n+k}(x) \sim g\left(\left(\frac{v\tau}{Cn}\right)^{1/d}\|(DT_\zeta)^k(w)\|\right)\frac{(DT_\zeta)^k(w)}{\|(DT_\zeta)^k(w)\|}, $$ from which we find
$$\frac{u_n^{-1}\left(\|\X_{r_n+k}(x)\| \right)}\tau\frac{\X_{r_n+k}(x)}{\|\X_{r_n+k}(x)\|} \sim v \|(DT_\zeta)^k(w)\|^d\frac{(DT_\zeta)^k(w)}{\|(DT_\zeta)^k(w)\|}.$$
All of these asymptotics become equalities as $n\to \infty$. The fact that $T^{r_n}(x)$ is chosen according to the invariant measure $\mu$ on the small neighbourhood $B_{g^{-1}(u_n(\tau))}(\zeta)=\{\|X_{r_n}(x)\|>u_n(\tau)\}$ of $\zeta$ and the fact that $\mu$ behaves asymptotically like Lebesgue measure on small neighbourhoods of $\zeta$ give that the finite $Y_j$s are indeed of the form required.  The infinite terms appear because, due to the repelling nature of $\zeta$, the entrance of the orbit of $x$ at time $r_n$ in $B_{g^{-1}(u_n(\tau))}(\zeta)$ can only be preceded by a finite number, $s\in\N_0$, of consecutive hits to $B_{g^{-1}(u_n(\tau))}(\zeta)$. But then $T^{r_n-s-1}(x)$, for example, must belong to neighbourhood of $T^{-1}(\zeta)\setminus\{\zeta\}$, which are at a fixed distance from $\zeta$: let $c=\dist(\zeta, T^{-1}(\zeta)\setminus\{\zeta\})>0$.  Then for $n$ sufficiently large we have $\dist(T^{r_n-s-1}(x),\zeta)>c/2$ and hence $u_n^{-1}(\|X_{r_n-s-1}\|)\sim Cn (c/2)^d$, which clearly diverges to $\infty$ as $n\to\infty$.

In the invertible case, for the toral diffeomorphisms considered, we have $T^{-1}(\zeta)=\zeta$ and hence for all $j\in\Z$, we have $Y_j=U\cdot \|(DT_\zeta)^{j}(\Theta)\|^d\frac{(DT_\zeta)^{j}(\Theta)}{\|(DT_\zeta)^{j}(\Theta)\|}$, where $U$ and $\Theta$ are as above.

We observe that if instead $\zeta$ were a periodic point of period $q>1$ then elements of $Y$ would be the same as for the fixed point case, but with $q-1$ $\infty$ terms between the entries $qk$ and $q(k+1)$, for all $k\in\Z$.   

Now that we understand $(Y_j)_j$ it is easy to see that the condition 
$$ \p ( \mbox{all finite }\allowbreak Q_{j} \allowbreak \mbox{'s are mutually different}) = 1$$ in Theorem~\ref{thm:record-theorem} is also trivially satisfied in these examples as it can be interpreted as $\zeta$ being hyperbolic.  

In the one-dimensional case, concerning most of the examples given in the noninvertible scenario in Section~\ref{subsec:systems}, if we consider $\zeta$ as a repelling fixed point such that $|DT_\zeta|=\chi>1$ then, in this case, $(Y_j)_j$ assumes the simpler form:
$$
(\ldots,\,\infty,\,\infty,\,U\chi^{-s},\,\ldots, U\chi^{-1},\, U,\, U\chi,\, U\chi^2,\ldots) 
$$

We remark that before anchoring the process to obtain $(Z_j)_j$, we do not know if the extreme observation $\{\|X_{r_n}(x)\|>u_n(\tau)\}$, corresponding to the entrance of the orbit in $B_{g^{-1}(u_n(\tau))}(\zeta)$ at time $r_n$, is the first cluster observation and $s$ is precisely accounting for the preceding exceedances in such cluster. 

Finally, to obtain the transformed anchored tail process we need to condition the $Y_j$ on the event $ \inf_{j\leq -1}\|Y_j\|\geq 1$, which is to say that in the computation of the $Y_j$, the exceedance occurring at time $r_n$ is indeed the \emph{first} exceedance of the cluster, which means $s=0$ and therefore $(Z_j)_j$ is of the form 
$$
(\ldots,\,\infty,\,\infty, U,\, U\chi,\, U\chi^2,\ldots) 
$$

\subsection{Particular application with overshooting}
\label{subappendix:overshooting}
We consider now the application in Example~\ref{example:oscillatory}. We start by noting that $\{|\psi(x)|>u\}=B_{1/8}(\eps_u)\cup B_{3/8}(\tilde \eps_u)$, where $\eps_u,\tilde\eps_u\sim u^{-1/2}$, which means that $\mu(|\psi(x)|>u)\sim4 u^{-1/2}$. Having \eqref{un} in mind, we take $u_n(\tau)=\frac{16n^2}{\tau^2}$ and consequently $u_n^{-1}(z)=4nz^{-1/2}$ (note that, as we did earlier, we write equalities since all our estimates are asymptotic in $n$).  If $x\in B_{1/8}(\eps_{u_n(\tau)})$ (recall that $\eps_{u_n(\tau)}\sim\tau/4n$) then $u_n^{-1}(|\psi(x)|)\sim 4n |x-1/8|$, while if $x\in B_{3/8}(\eps_{u_n(\tau)})$ then $u_n^{-1}(|\psi(x)|)\sim 4n |x-3/8|$. For simplicity, let $0<v<1$ be such that $|X_{r_n}|=u_n(v\tau)$, which translates as $|T^{r_n}(x)-1/8|\sim v\tau/4n$ or $|T^{r_n}(x)-3/8|\sim v\tau/4n$. Now, recalling that $\{1/8,3/8\}$ corresponds to a period two orbit, $T$ triples distances and $\psi$ has opposite signs on the neighbourhood of these points, then 
$$\frac{u_n^{-1}\left(|X_{r_n+k}(x)| \right)}\tau\frac{X_{r_n+k}(x)}{|X_{r_n+k}(x)|} \sim v (-1)^s(-3)^k,$$
where $s=0$ if $T^{r_n}(x)\in B_{1/8}(\eps_{u_n(\tau)})$ and $s=1$ if $T^{r_n}(x)\in B_{3/8}(\tilde \eps_{u_n(\tau)})$.
All of these asymptotics become equalities as $n\to \infty$. The fact that $T^{r_n}(x)$ is chosen according to the invariant measure $\mu=\text{Leb}$ on the set $B_{1/8}(\eps_{u_n(\tau)})\cup B_{3/8}(\tilde \eps_{u_n(\tau)})$ and $\lim_{n\to\infty}\frac{\mu(B_{1/8}(\eps_{u_n(\tau)}))}{\mu(|X_0|>u_n(\tau))}=\frac12=\lim_{n\to\infty}\frac{\mu(B_{3/8}(\tilde \eps_{u_n(\tau)}))}{\mu(|X_0|>u_n(\tau))}$ gives that $(Z_j)_{j\in\Z}$ is indeed of the form required, namely,  
$$
(\ldots,\infty,\infty, U E, U(-3) E, U(-3)^2 E, \ldots),
$$
where $U$ in uniformly distributed on $[0,1]$ and $\p(E=1)=1/2=\p(E=-1)$.

\end{document}